\documentclass[a4paper,10pt,oneside]{amsart}
\usepackage[utf8]{inputenc}
\usepackage{amsfonts,amsmath,amssymb,amsthm,fixmath,stmaryrd, mathtools, bbm, mathrsfs}
\usepackage[alphabetic]{amsrefs}
\usepackage{tikz-cd} 
\usepackage{bbm}
\usepackage{amstext} 
\usepackage{enumitem}
\usepackage[T1]{fontenc}
\usepackage{lmodern}
\usepackage{hyperref}
\usepackage{microtype}
\usepackage[english]{babel}

\usepackage{geometry}
\usepackage{verbatim}

\usepackage{longtable}
\usepackage{array}   
\newcolumntype{L}{>{$}p{0.11\linewidth}<{$}}
\newcolumntype{M}{>{$}p{0.80\linewidth}<{$}}

\allowdisplaybreaks

\usepackage{tikz}
\usepackage{xcolor}
\usetikzlibrary{matrix}

\usepackage{etoolbox}

\usepackage{dynkin-diagrams}

\usepackage{calligra}
\DeclareMathOperator{\Hom}{\mathscr{H}\text{\kern -3pt {\calligra\large om}}\,}

\numberwithin{equation}{section}

\theoremstyle{plain} 
\newtheorem{thm}[equation]{Theorem}
\newtheorem{cor}[equation]{Corollary}
\newtheorem{lem}[equation]{Lemma}
\newtheorem{prop}[equation]{Proposition}

\theoremstyle{definition}
\newtheorem{defn}[equation]{Definition}

\theoremstyle{remark}
\newtheorem{rem}[equation]{Remark}

\DeclareMathOperator{\GU}{GU}
\DeclareMathOperator{\SO}{SO}
\DeclareMathOperator{\Sp}{Sp}
\DeclareMathOperator{\Quot}{Quot}

\DeclareMathOperator{\Lie}{Lie}
\DeclareMathOperator{\End}{End}
\DeclareMathOperator{\Adm}{Adm}
\DeclareMathOperator{\Gal}{Gal}
\DeclareMathOperator{\id}{id}
\DeclareMathOperator{\lt}{lt}

\DeclareMathOperator{\trace}{Tr}
\DeclareMathOperator{\Nilp}{Nilp}
\DeclareMathOperator{\charpoly}{char}
\DeclareMathOperator{\inv}{inv}
\DeclareMathOperator{\supp}{supp}

\DeclareMathOperator{\Spf}{Spf}

\DeclareMathOperator{\Ad}{Ad}

\DeclareMathOperator{\Grass}{Grass}

\newcommand{\N}{\mathcal{N}}
\newcommand{\V}{\mathcal{V}}
\newcommand{\Q}{\mathbb{Q}}
\newcommand{\Z}{\mathbb{Z}}
\newcommand{\F}{\mathbb{F}}

\title{On the Rapoport-Zink space for $\GU(2, 4)$ over a ramified prime}
\author{Stefania Trentin}

\begin{document}

\begin{abstract}
  In this work, we study the supersingular locus of the Shimura variety associated to the unitary group $\GU(2,4)$ over a ramified prime. We show that the associated Rapoport-Zink space is flat, and we give an explicit description of the irreducible components of the reduction modulo $p$  of the basic locus. In particular, we show that  these are universally homeomorphic to either a generalized Deligne-Lusztig variety for a symplectic group or to the closure of a vector bundle over a classical Deligne-Lusztig variety for an orthogonal group. Our results are confirmed in the group-theoretical setting by the reduction method à la Deligne and Lusztig and the study of the admissible set.
\end{abstract}

\maketitle

\section{Introduction}\label{sec:intro}
\subsection{Motivation} 
Understanding arithmetic properties of Shimura varieties has been a fundamental question in recent developments in number theory and algebraic geometry. Shimura varieties of PEL type can be described as moduli spaces of Abelian varieties with additional structure, namely polarization, endomorphism and level structure, see \cite{kottwitz}*{Sec.\ 5}. The special fiber of a Shimura variety at a prime $p$ can be decomposed into finitely many \emph{Newton strata} according to the isogeny class of the $p$-divisible groups corresponding to each Abelian variety. Studying the Newton stratification of the special fiber of a suitable integral model has been a fundamental tool to understand the arithmetic of Shimura varieties.

There is a unique closed Newton stratum, called the basic locus, which in the Siegel case coincides with the supersingular locus of the Shimura variety. A good understanding of the basic Newton stratum is expected to be essential to prove results about general Newton strata and the whole special fiber using an induction process, as stated in the \textit{Harris-Viehmann conjecture} \cite{rv}*{Sec.\ 5.1}. Moreover, a concrete description of basic loci has been of great importance, among others, in the work of Rapoport, Terstiege and Zhang on the arithmetic fundamental lemma, see \cite{rtz}. For an overview of other applications in arithmetic geometry of the study of basic loci we refer to \cite{rtw} and \cite{vollaard}.

The aim of the present work is to study the supersingular locus of the reduction modulo $p$ of the Shimura variety for the unitary group $\GU(2,4)$ over a ramified prime $p$. In \cite{rz} Rapoport and Zink prove the Uniformization Theorem, which enables us to formulate this problem in terms of a closed subscheme of a moduli space of $p$-divisible groups with additional structure, called Rapoport-Zink space. Over a field of equal characteristic, for example over $\F_p(\!(t)\!)$, this corresponds to the study of some affine Deligne-Lusztig varieties associated to the group-theoretical datum underlying the Shimura variety. In this paper, we give a concrete description of the irreducible components of the reduced scheme underlying the reduction modulo $p$ of the basic locus of the Rapoport-Zink space corresponding to ramified $\GU(2,4)$. In addition, we prove that the Rapoport-Zink space is \emph{flat} over the ring of integers of the quadratic ramified extension of $\Q_p$ associated to the Shimura variety. 

Previous works on the supersingular locus of Shimura varieties for unitary groups include \cite{vollaard} and \cite{vw} for the group $\GU(1, n-1)$ over an inert prime, \cite{rtw} for $\GU(1, n-1)$ over a ramified prime and \cite{harris} for the same group over a split prime. The Rapoport-Zink space for $\GU(2,2)$ over an inert prime has been studied in \cite{howardpappas} and over a ramified prime in \cite{oki}. In all these cases, the irreducible components of the basic locus are isomorphic to a (generalized) Deligne-Lusztig variety. Moreover, the set of irreducible components and their intersection pattern can be described in terms of some Bruhat-Tits building. These observations have motivated the work of \cite{ghn_fully}, which gives a complete classification of the Shimura varieties whose basic locus has irreducible components  isomorphic to (generalized) Deligne-Lusztig varieties. These are called \emph{fully Hodge-Newton decomposable} Shimura varieties. This property does not depend on the level structure and can be completely formulated in terms of the associated group-theoretical datum.  It is also shown in \textit{loc.cit.}\ that the basic locus of a fully Hodge-Newton decomposable Shimura variety admits a decomposition as a union of classical Deligne-Lusztig varieties, and that this decomposition is actually a stratification. The index set and closure relations of the stratification are encoded in the associated Bruhat-Tits building. The cases mentioned so far all belong to this family of Shimura varieties, compare the table in \cite{ghn_fully}*{Sec.\ 3}. 

The starting point of this paper is the question whether there exists in more general cases a description of the basic locus \textit{in terms of} classical Deligne-Lusztig varieties and whether this decomposition is a stratification. The existence of a relation between the basic locus and classical Deligne-Lusztig varieties has already been conjectured for example in \cite{rtw}. The work of Fox and Imai \cites{fox} and of Fox,  Howard and Imai \cite{fox2} gives a positive answer for the supersingular locus of $\GU(2, n-2)$ over an inert prime. In this case the irreducible components are realized as closed subschemes of flag schemes over Deligne-Lusztig varieties. As a first approach to the problem over a ramified prime, we choose to study the first example of non-fully Hodge-Newton decomposable unitary group, that is the one with the lowest dimension and signature, namely $\GU(2,4)$. As we are going to see, the irreducible components of the basic locus are in this case describable in terms of some generalized Deligne-Lusztig varieties for a symplectic group or in terms of vector bundles over classical Deligne-Lusztig varieties for an orthogonal group. Moreover, there is a decomposition of each irreducible component which resembles the stratification in the fully Hodge-Newton decomposable case, but the closure relations are not as neat, \textit{i.e.}\ the closure of some strata cannot be expressed as a union of other strata. One can still find a bijection between the set of components and the vertices of the Bruhat-Tits building. 

Our results are confirmed by the behavior of the affine Deligne-Lusztig varieties associated to our problem in the equal characteristic case. Here the reduction method \`a la Deligne and Lusztig delivers exactly the same decomposition as the one we obtain for the Rapoport-Zink space. We are going to have to distinguish between two cases, according to the fixed $p$-divisible group used to define the Rapoport-Zink space. The associated affine Deligne-Lusztig varieties in one of these two cases actually belong to a larger family introduced in \cite{he_aug} called \textit{of finite Coxeter type}. These are characterized by the fact that the reduction method produces either $\mathbb{A}^n$- or $\mathbb{G}_m^n$-bundles over classical Deligne-Lusztig varieties for a Coxeter element. In the second case, not all associated affine Deligne-Lusztig varieties are of finite Coxeter type. Indeed, the reduction method produces an additional $\mathbb{A}^1$-bundle over a classical Deligne-Lusztig variety which is not of Coxeter type. 

\subsection{Overview and main results}
In order to introduce the main results of the present work, we need to fix some notation, for more precise definitions and statements we refer to the main body. Following the exposition of \cite{kr}*{Sec.\ 2} we briefly recall the global moduli problem that describes unitary Shimura varieties. Let $0 \le s \le n$. Let $K = \Q[\sqrt{\Delta}]$ be an imaginary quadratic extension of $\Q$, with ring of integers $\mathcal{O}_K$ and non-trivial Galois automorphism $\sigma$. We consider the functor $\mathcal{M}(K, s, n-s)$ over the category of locally Noetherian $\mathcal{O}_K$-schemes associating to such a scheme $S$ the groupoid of  triples of the form $(A, \iota, \lambda)$. Here $A$ is an Abelian scheme over $\mathcal{O}_S$, equipped with an action $\iota$ of $\mathcal{O}_K$ and a principal polarization $\lambda$, whose Rosati involution induces via $\iota$ the automorphism $\sigma$ on $\mathcal{O}_K$. The action $\iota$ is also required to satisfy Kottwitz' determinant condition and Pappas' wedge condition on the Lie algebra of $A$, see \cite{kr}*{2.1, 2.2}. It is proved in \cite{kr}*{Prop.\ 2.1} that $\mathcal{M}(K, s, n-s)$ is a Deligne-Mumford stack over $\mathcal{O}_K$. Moreover, the wedge condition ensures flatness for $s = 1$, as shown in \cite{pap}*{Thm.\ 4.5}. It is conjectured in \cite{pap}*{4.16} that this holds for any signature, which is supported by computational evidence. We also recall there are some variants of the moduli problem, which satisfy flatness in higher signature and dimension and have been introduced for example in \cite{hernandez} and \cite{zachos}. Our first main result is shown in Section \ref{sec:moduli} and concerns flatness of $\mathcal{M}(K, 2, 4)$. 
\begin{prop}
  Assume that $2 \nmid \Delta$. Then $\mathcal{M}(K, 2, 4)$ is flat over $\mathcal{O}_K$.
\end{prop}
The proof of this first result builds on the reduction of the problem to a question in algebraic geometry and commutative algebra presented in \cite{pap}*{4.16}. In particular, in \textit{loc.cit.}\ the author relates the flatness conjecture to an open question in invariant theory raised by \cite{dcp}. Our proof combines techniques from different mathematical fields, from computational algebra to model theory, and can be adapted to prove flatness for $n = 8$ or higher. We are optimistic that our results could serve as the basis for an induction process on the dimension $n$ to prove flatness of $\mathcal{M}(K, 2, n-2)$. 

Once we have established flatness, we can move to the description of the irreducible components of the basic locus of $\mathcal{M}(K, 2, 4)$. To do so we introduce the associated Rapoport-Zink space $\mathcal{N}$. It parametrizes  $p$-divisible groups with some additional structure and equipped with a quasi-isogeny to a fixed $p$-divisible group $\mathbb{X}$, we refer to Section \ref{sec:moduli} for a precise definition. In particular, we focus on $\bar{\N}^0_{\mathrm{red}}$, the reduced scheme underlying the reduction modulo $p$ of the closed subscheme $\mathcal{\N}^0$ of $\N$ where the  quasi-isogeny to $\mathbb{X}$ has height zero. The Uniformization Theorem \cite{rz}*{Thm.\ 6.30} gives an isomorphism of formal stacks between the completion of $\mathcal{M}(K, s,n- s)$ along its supersingular locus and a double quotient of $\bar{\N^0}_{\mathrm{red}}$.

Via Dieudonn\'e theory we associate to the fixed $p$-divisible group $\mathbb{X}$ a Hermitian $E$-vector space $C$ of dimension $n = 6$. Here $E$ is the quadratic extension of $\Q_p$ given by the completion of $K$. In $C$ we consider two families of $\mathcal{O}_E$-lattices, whose properties we study in Section \ref{sec:lattices}. As in \cite{rtw}, we say that $\Lambda$ is a vertex lattice of type $t$ if $p\Lambda \subset \Lambda^\sharp \subset \Lambda$, and the quotient $\Lambda/\Lambda^{\sharp}$ is a $\F_p$-vector space of dimension $t$. Here $\Lambda^{\sharp}$ is the dual of $\Lambda$ with respect to the Hermitian form and it contains $p\Lambda$. As in \cite{jac}, we say a lattice $\Lambda$ is $2$-modular if its dual $\Lambda^{\sharp}$ is equal to $p\Lambda$. This second type of lattices does not play any role in \cite{rtw}, and is a specific feature of signature $(2,n-2)$. As we are going to see, the behavior of the irreducible components of $\bar{\N^0}_{\mathrm{red}}$ is quite different depending on the sign of the discriminant of $C$. 

Before giving a description of the irreducible components of $\bar{\N}^0_{\mathrm{red}}$, we recall in Section \ref{sec:dlvs} some properties of classical Deligne-Lusztig varieties. In particular, we study three families of varieties, one for the symplectic group, which is the generalization to signature $(2, n-2)$ of the varieties introduced in \cite{rtw}*{Sec.\ 5}, and two for the orthogonal group. These varieties become relevant in the subsequent sections. As a preparation for the main result, we study in Section \ref{sec:closedpoints} the $k$-valued points of $\bar{\N}^0_{\mathrm{red}}$ for any algebraically closed field $k$. Section \ref{sec:geometry} is dedicated to the proof of the following theorem.
\begin{thm}\label{thm:intro}
  \par{i)} Assume $C$ is split, that is with discriminant equal to $1$. Then $\bar{\N}^0_{\mathrm{red}}$ has irreducible components of two types.
  \begin{enumerate}
    \item $\N_{\mathcal{L}}$, for every vertex lattice $\mathcal{L} \subset C$ of type $6$. These components are universally homeo\-morphic to generalized Deligne-Lusztig varieties for the symplectic group $\Sp_6$ and have dimension $5$.
    \item $\N_{\Lambda}$, for every $2$-modular lattice $\Lambda \subset C$. These components are universally homeomorphic to the closure of a line bundle over a generalized Deligne-Lusztig variety for the orthogonal group $\SO_6$ and have dimension $4$.
  \end{enumerate}
  \par{ii)} Assume $C$ is non-split, that is with discriminant equal to $-1$. Then $\bar{\N}^0_{\mathrm{red}}$ is pure of dimension $4$ and has irreducible components of two types.
  \begin{enumerate}
    \item One irreducible component $\N_{\Lambda}^1$ for every $2$-modular lattice $\Lambda \subset C$. These components are universally homeomorphic to the closure of a line bundle over a generalized Deligne-Lusztig variety for the non-split orthogonal group of rank $6$.
    \item Two irreducible components $\N_{\Lambda}^2$ for every $2$-modular lattice $\Lambda \subset C$. These components are universally homeomorphic to the closure of a rank-two vector bundle over a classical Deligne-Lusztig variety of Coxeter type for the non-split orthogonal group of rank $6$.
  \end{enumerate}
\end{thm}

As expected, there is a natural way to relate the irreducible components of $\bar{\N}^0_{\mathrm{red}}$ to classical Deligne-Lusztig varieties, which is however not an isomorphism. This is coherent with the fact that the Shimura variety for $\GU(2,4)$ is not fully Hodge-Newton decomposable in the sense of \cite{ghn_fully}. It is interesting to notice that in the split case the first type of irreducible components closely resembles those of the Rapoport-Zink space for signature $(1, n-1)$. One may ask whether it is possible to prove a stronger result, for example that the homeomorphisms are isomorphisms as in \cite{vw} and \cite{rtw}. This is discussed in detail in Remark \ref{rem:notiso}. In the non-split case, the fact that we have pairs of components of type $\N_{\Lambda}^2$, corresponds to the fact that in this case the orbit of a Coxeter element under the action of the Frobenius consists of two elements. 

Finally, in Section \ref{sec:adlvs} we study the group-theoretical datum associated to our problem. We recall some relevant definitions and results, and we study in detail the admissible set and the associated family of affine Deligne-Lusztig varieties for ramified $\GU(2,4)$. Using the reduction method \`a la Deligne and Lusztig, we show that the description of the irreducible components of $\bar{\N}^0_{\mathrm{red}}$ given in Theorem \ref{thm:intro} is mirrored by the behavior of the corresponding affine Deligne-Lusztig varieties. 

\subsection*{Acknowledgements} First and foremost I would like to thank my supervisor  Eva Viehmann for her support during my PhD. I am sincerely thankful for her constant help and feedback, which guided me through my studies. 

I wish to express my gratitude to Michael Rapoport and Torsten Wedhorn for very helpful discussions and for answering my questions on their papers \cite{rtw} and \cite{vw}. I am thankful to Felix Schremmer for sharing his knowledge on Coxeter groups and affine Deligne-Lusztig varieties, pointing me to the relevant literature for Section \ref{sec:dlvs}. I would like to thank Simone Ramello for introducing me to model theory and working out together the details of Remark \ref{rem:model}. I am also grateful to Urs Hartl and Damien Junger for helpful conversations.  

I was supported by the ERC Consolidator Grant 770936: \textit{NewtonStrat}, by the Ada Lovelace Fellowship of the Cluster of Mathematics Münster funded by the Deutsche Forschungsgemeinschaft (DFG, German Research Foundation) under Germany's Excellence Strategy EXC 2044 -390685587, Mathematics Münster: Dynamics-Geometry-Structure, and by the Deutsche Forschungsgemeinschaft (DFG, German Research Foundation) through the Collaborative Research Center TRR326 ``Geometry and Arithmetic of uniformized Structures'', project number 444845124.

\section{The moduli space}\label{sec:moduli}

In this section we introduce the Rapoport-Zink space associated with the Shimura variety for $\GU(2, n-2)$ over a ramified prime, and we prove its flatness in the case $n = 6$. We fix the notation, which we will use in the rest of this paper. Let $n$ be an integer greater or equal than $3$ and $p$ be an odd prime, we denote 
\begin{itemize}[nolistsep]
  \item[-] $E$ a ramified quadratic extension of $\Q_p$ with ring of integers $\mathcal{O}_E$,
  \item[-] $\pi$ a uniformizer of $E$ such that $\pi^2 = \pi_0$ is a uniformizer of $\Q_p$, this is possible as $p$ is odd,
  \item[-] $\mathbb{F}$ an algebraic closure of $\mathbb{F}_p$, its ring of Witt vectors is denoted by $W$ and its fraction field by $W_{\Q} = \Quot(W)$, 
  \item[-] $\breve{E} = E \otimes_{\Q_p} W_{\Q}$ and its ring of integers $\mathcal{O}_{\breve{E}} = \mathcal{O}_E \otimes_{\Z_p} W$,
  \item[-] $\sigma$ the Frobenius on $\mathbb{F}, W, W_{\Q}$ and also the map $1 \otimes \sigma$ on $\breve{E}$,
  \item[-] $\psi_0 : E \rightarrow \breve{E}$  the natural embedding and $\psi_1$ its conjugate, that is $\psi_1 = \psi_0 \circ \bar{\ }$.
\end{itemize}

Rapoport-Zink spaces were first introduced in \cite{rz}. They are moduli spaces parametrizing quasi-isogenies of $p$-divisible groups with additional structure. By the Uniformization Theorem, see \cite{rz}*{Thm.\ 6.30}, they play a crucial role in the study of the basic locus of the corresponding Shimura variety of PEL type. In this section we recall the definition of the Rapoport-Zink space $\N_{s,n}$ associated to the Shimura variety for $\GU(s, n-s)$. We follow the notation of \cite{rtw}*{Sec.\ 2} and of \cite{pap}*{Sec.\ 4}.

Fix a supersingular $p$-divisible group $\mathbb{X}$ of dimension $n$ and height $2n$ over $\mathbb{F}$ equipped with an action $\iota_{\mathbb{X}}: \mathcal{O}_E \rightarrow \End(\mathbb{X})$. Let $\lambda_{\mathbb{X}}$ be a principal quasi-polarization of $\mathbb{X}$ whose Rosati involution induces on $\mathcal{O}_E$ the non-trivial automorphism over $\Q_p$.

Let $\Nilp$ be the category of $\mathcal{O}_{\breve{E}}$-schemes $S$ such that $\pi\cdot \mathcal{O}_S$ is a locally nilpotent ideal sheaf. Fix $n \ge 3$ and $s\le n$. We study the moduli functor $\N_{s,n}$ associating to a scheme $S$ in $\Nilp$ the set of isomorphism classes of quadruples $(X, \iota, \lambda, \rho)$, where $X$ is a $p$-divisible group over $S$ and $\iota : \mathcal{O}_E \rightarrow \End(X)$ is a homomorphism satisfying the following two  conditions, introduced respectively by Kottwitz and Pappas, 
\begin{align}
  \charpoly(\iota(a)  \mid \Lie(X)) &= (T - \psi_0(a))^s(T-\psi_1(a))^{n -s} \\
  \bigwedge^{n-s +1} (\iota(\pi) - \pi \mid \Lie(X)) &= 0 ~~~~~~~~~~~
  \bigwedge^{s+1} (\iota(\pi) + \pi \mid \Lie(X) ) = 0. \label{eq:pap}                                 
\end{align}
Furthermore, $\lambda: X \rightarrow X^\vee$ is a principal quasi-polarization and $\rho : X \times_S (S \times_{\mathcal{O}_E} \F) \rightarrow \mathbb{X} \times_\F (S \times_{\mathcal{O}_E} \F)$ is an $\mathcal{O}_E$-linear quasi-isogeny such that $\lambda$ and $\rho^*\lambda_{\mathbb{X}}$ differ locally on $(S \times_{\mathcal{O}_E} \F)$ by a factor in $\Q_p^{\times}$. We also require that the Rosati involution associated to $\lambda$ induces on $\mathcal{O}_E$ the non-trivial automorphism over $\Q_p$. Last, two quadruples $(X, \iota, \lambda, \rho)$ and $(X', \iota', \lambda', \rho')$ are isomorphic if there is an $\mathcal{O}_E$-linear isomorphism $\alpha: X \rightarrow X'$ such that $\rho'\circ(\alpha \times_S (S \times_{\mathcal{O}_E} \F)) = \rho$ and $\alpha^*\lambda'$ is a $\Z_p^{\times}$-multiple of $\lambda$.

\begin{prop}\cite{rz}*{Sec.\ 6.9} 
 The moduli functor $\mathcal{N}_{s,n}$ is representable by a separated formal scheme $\N_{s,n}$ locally formally of finite type over $\Spf \mathcal{O}_{\breve{E}}$.
\end{prop}

\subsection{Flatness} The conditions (\ref{eq:pap}) on the exterior powers of the action of $\pi$ on the Lie algebra of $X$ were introduced by Pappas in \cite{pap}*{Sec.\ 4} and ensure flatness of the moduli space $\N_{1,n}$ over $\mathcal{O}_{\breve{E}}$, as proved in \cite{pap}*{Thm.\ 4.5}. It is conjectured that this holds for any signature $s$. In \cite{pap}*{4.16} the author presents his computations in dimension $n \le 6$ and for primes $p \le 31991$ which confirm flatness in these cases. We prove in this section that for signature $2$ and dimension $6$ the moduli space $\N_{2,6}$ is flat for any odd prime $p$. The first step of the proof is already in \cite{pap}*{Sec.\ 4.16}, where the author relates flatness of the Rapoport-Zink space to a conjecture by de Concini and Procesi \cite{dcp}*{Sec.\ 1} on ideals generated by matrix entries. In particular, it is sufficient to show that a certain polynomial ideal is radical. We prove that for signature $(2, n-2)$ some generators of this ideal are redundant. We consider then the case $n = 6$ and give a method to prove radicality almost independently of the characteristic $p$.

\begin{prop}\label{prop:pap}\cite{pap}*{Sec.\ 4.16}
  Let $X$ denote the generic matrix over $\F_p[x_{ij},1 \le i,j \le n]$
  \begin{equation*} 
    X = \left( \begin{array}{ccc}
      x_{11} & \cdots & x_{1n} \\ 
      \vdots &  & \vdots \\ 
      x_{n1} & \cdots & x_{nn}
    \end{array} \right).
\end{equation*} Consider the ideal $J({s,n}) \subset \F_p[x_{ij}, 1 \le i,j \le n]$ generated by the polynomials given by the entries of $X^2$,  the $(s+1)$-rank minors of $X$,  the entries of $X-X^t$ and by the (non-leading) coefficients of the characteristic polynomial of $X$. Then if $J(s,n)$ is radical, the Rapoport-Zink space $\N_{s,n}$ is flat over $\mathcal{O}_E$.
\end{prop}

We are then interested in showing that the ideal $J(2,6)$ is radical. First, we show that some generators of $J(2,n)$, for any $n$, are actually redundant. 
\begin{lem} 
  Let $X = [x_{ij} = x_{ji}]$ denote the $n$-dimensional generic symmetric matrix over $\F_p[x_{ij}, 1\le i \le j \le n]$. Then 
  \begin{equation}\label{eq:ideal}
    J({2,n}) = \langle X^2, \bigwedge^3 X, \trace(X) \rangle,
  \end{equation}
  where the right-hand side denotes the ideal of $\F_p[x_{ij}, 1\le i\le j\le n]$ generated by the polynomials given by the entries of $X^2$,  the rank-$3$ minors of $X$ and by its trace.
\end{lem}
\begin{proof}
  Since $J(2,n)$ contains the polynomials $x_{ij} - x_{ji}$ it is clear that we can reduce the number of variables and assume that $X$ is symmetric. Recall that the coefficient of the term of degree $n-k$ in the characteristic polynomial of $X$ is given by the sum of the $k\times k$ principal minors of $X$. By definition of $J(2,n)$, the polynomials corresponding to the minors of rank at least $3$ are already contained in it. It follows that the equations given by the coefficients of degree $n-k$ with $k \ge 3$ are redundant as generators of $J(2,n)$. Let $\sigma_2$ denote the coefficient of degree $n-2$ of the characteristic polynomial of $X$. It is easy to check that for any matrix $X$ the trace of $X$ is related to that of $X^2$ by the identity $\trace(X^2) = \trace(X)^2 - 2\sigma_2(X)$. Since $p \neq 2$ this tells us that $\sigma_2(X)$ is unnecessary as generator of $J(2,n)$. 
\end{proof}

Observe that if we change to $2$ the exponent in the exterior power of (\ref{eq:ideal}) we obtain again the ideal studied in \cite{pap}*{4.12}. 

\subsection{Flatness in small dimension} We fix for the rest of this section $n = 6$ and $s = 2$, and we let $p$ denote an odd prime. From now on we simplify the notation and just write $J$ for the ideal $J(2,6)$. Our goal is to prove the following proposition.

\begin{prop}\label{prop:radical6}
  The ideal $J = J(2,6) \subset \F_p[x_{ij}, 1 \le i \le j \le 6]$ is radical for all primes $p \neq 2$. By Proposition \ref{prop:pap}, it follows that the Rapoport-Zink space $\N_{2,6}$ is flat over $\mathcal{O}_E$. 
\end{prop}

\begin{rem}
  Observe that for $p = 2$ the ideal $J$ is not radical. Indeed, it contains for example the polynomial $(X^2)_{11} = x_{11}^2 + x_{12}^2 + \dotsm + x_{16}^2$, which is a square over $\F_2$, while the only polynomial of degree $1$ in $J$ is the trace.
\end{rem}

Proving that an ideal is radical is known in general to be a quite hard problem. There are several algorithms to compute the radical of a polynomial ideal, both in zero and positive characteristic, see for example \cite{kemper}. However, they all require fixing the field of coefficients beforehand and therefore are not a feasible choice for us, as we want to prove that $J \subset \F_p[x_{ij}, 1 \le i \le j \le 6]$ is radical for any $p$. As far as we could research in the literature there is also no algorithm for directly proving that an ideal is radical without first computing its primary decomposition or its radical. 

\begin{proof}
  Our strategy for proving that $J$ is radical is to reduce ourselves to solving the same problem for a sequence of polynomial ideals in one variable. This will turn out to be much easier as the resulting univariate ideals will be generated by polynomials of degree at most two. We will have to solve two other problems along the way. First, we have to explicitly describe this sequence of univariate ideals, that is we have to give a set of generators for each of them. Second, since our goal is to prove radicality independently of the characteristic, we will have to show that our arguments and computations hold over $\F_p$ for almost all primes $p$. 

  The key idea for reducing to univariate polynomial ideals is in the following easy observation from commutative algebra.

  \begin{lem}\label{lem:elim} 
    Let $I$ be an ideal in $R[x]$, where $R$ is any commutative ring with unit. Then $I$ is radical if and only if the image of $I$ in $(R/R\cap I)[x]$ is radical. Moreover, if $R$ is a reduced algebra and $I$ is radical, then so is the ideal $R \cap I$ in $R$. 
  \end{lem}
  \begin{proof}
    Let $\bar{I}$ be the image of $I$ in $(R/R\cap I)[x]$. If $I$ is radical and $f \in R[x]$ is such that $\overline{f^n} \in \bar{I}$, this means that $f^n + i \in I$ for some $i \in R \cap I$, that is $f^n \in I$. Therefore, $f \in I$, since $I$ is radical, from which it follows that $\overline{f} \in \bar{I}$. Conversely, if $\bar{I}$ is radical, and we have $f^n \in I$, then the image $\overline{f^n} \in \bar{I}$, which is radical, hence $\overline{f} \in \bar{I}$. This means that $f + i \in I$ for some $i \in R \cap I$, that is $f \in I$. 
    
    Last, observe that if $R$ is a reduced algebra, for any polynomial in $R[x]$, the degree of $f^n$ is equal to $n \deg(f)$. Therefore, if $f^n \in R \cap I$ it means that $f^n$ has degree zero and therefore $f \in R$, as well. Since $I$ is radical, $f \in I$, from which the statement follows.
  \end{proof}

  In order to prove that $J \subset \F_p[x_{ij}, 1 \le i \le j \le 6]$ is radical we can start for example by inspecting its intersection $J_{12} = J \cap \F_p[x_{12}, x_{13}, \dots, x_{66}]$. If $J_{12}$ is not radical, then by the previous lemma $J$ is not radical either, and we have to stop. Otherwise, to prove that $J$ is radical is equivalent by Lemma \ref{lem:elim} to prove that the image $\overline{J}$ of $J$ in $R_{12}[x_{11}]$ is radical, where $R_{12} = \F_p[x_{12}, x_{13}, \dots, x_{66}]/J_{12}$. If $J_{12}$ is radical, then the algebra $R_{12}$ is reduced, hence we are confronted with the easier problem of proving radicality for an ideal in a univariate polynomial ring with reduced coefficient ring. We can apply this reasoning recursively to each variable $x_{ij}$, so that we obtain a chain of ideals 
  \begin{equation}\label{eq:chain}
     J_{66} = J \cap \F_p[x_{66}] \subset J_{56} = J \cap \F_p[x_{56}, x_{66}] \subset \dotsm \subset J_{12} = J \cap \F_p[x_{12}, x_{13}, \ldots, x_{66}] \subset J.
  \end{equation}
  Our strategy will then consist of proving radicality twenty-one times, one for each variable $x_{ij}$, as follows.
  \begin{itemize}[nolistsep]
    \item We start with proving that $J_{66}$ is radical. 
    \item At step $ij$ we know that the previous ideal $J_{ij+1}$ (or $J_{i+1 i+1}$ if $j$ is $6$) is radical, and we prove that the image $\overline{J_{ij}}$ in $R_{ij+1}[x_{ij}]$ is radical, which by Lemma \ref{lem:elim} implies that $J_{ij}$ is radical as well. Here again $R_{ij+1} = \F_p[x_{ij+1}, \dots , x_{66}]/J_{ij+1}$. 
  \end{itemize}
  This technique is a standard method in computational algebra called \emph{elimination}. It was inspired to us by reading the primality testing algorithm of \cite{gtz}*{Sec.\ 4}.

  To apply our elimination strategy we are confronted with the problem of finding generators for each intersection ideal $J_{ij}$ and for each image $\overline{J_{ij}}$. To do so we have to first recall the notion of Gr\"obner basis and present some relevant results. 
  \begin{defn} Consider the polynomial ring $R[x_1, \dots, x_m]$, where $R$ is any commutative ring with unit.
    \begin{enumerate}
      \item The lexicographic order given by $x_1 > x_2 > \dotsm > x_m$ is the total order on the set of monomials in $R[x_1, \dots, x_m]$ defined by
      \begin{equation*}
        x_1^{a_1}\dotsm x_m^{a_m} \le x_1^{b_1}\dotsm x_{m}^{b_m} \Longleftrightarrow \exists i \text{ such that } a_j = b_j \text{ for all } j \le i, \text{ and } a_{i+1} < b_{i+1}.
      \end{equation*} Moreover, the lexicographic order is a \emph{monomial order}, that is, if $u, v$ are two monomials such that $u \le v$ and $w$ is a third monomial, then $uw \le vw$.
      \item For a polynomial $f \in R[x_1, \dots ,x_m]$ the leading term $\lt(f)$ is the highest monomial of $f$ with respect to a given monomial order. For an ideal $I \subset R[x_1, \dots ,x_m]$, the initial ideal $\mathrm{in}(I)$ is the ideal generated by the leading terms of all elements of $I$.
      \item A finite subset $G \subset I$ is a \emph{Gr\"obner basis} for $I$ if the leading terms of the polynomials in $G$ generate the initial ideal of $I$. Gr\"obner bases where first introduced in \cite{buchenberger}, where it is proved that for any ideal $I$ and any choice of monomial order, there exists a Gr\"obner basis, and that it generates $I$.
    \end{enumerate}
  \end{defn}
  We collect here some relevant results about Gr\"obner bases that we will need in this section, proofs can be found for example in \cite{gtz}*{Sec.\ 3}. 
  \begin{lem}\label{lem:grob} 
    Let $I$ be an ideal in $R[x_1,\dots,x_m]$ and $G$ a Gr\"obner basis for $I$ with respect to the lexicographic order given by $x_1 > x_2 > \dotsm > x_m$.
    \begin{enumerate}
      \item $G \cap R[x_i, \dots, x_m]$ is a Gr\"obner basis for the elimination ideal $I \cap R[x_i, \dots, x_m]$.
      \item Consider the quotient map $\pi : R[x_1, \dots ,x_m] \rightarrow (R/R\cap I)[x_1,\dots,x_m]$. Then $\pi(G \smallsetminus G\cap R)$ is a Gr\"obner basis for  $\pi(I)$.
      \item Let $S$ be a multiplicatively closed subset of $R[x_1, \dots ,x_m]$. Then $G$ is a Gr\"obner basis for $S^{-1}I$ in the localization $S^{-1}R[x_1,\dots ,x_m]$.
    \end{enumerate}
  \end{lem}

  Consider our chain of elimination ideals (\ref{eq:chain}). The theory of Gr\"obner bases provides us with an effective way to compute a generating set of each ideal $J_{ij} = J \cap \F_p[x_{ij}, \dots, x_{66}]$ and of its image $\overline{J_{ij}}$ in $R_{ij+1}[x_{ij}] = \F_p[x_{ij+1}, \dots, x_{66}]/(J_{ij+1})[x_{ij}]$. We fix the lexicographic order on $\F_p[x_{11}, \dots, x_{66}]$ given by $x_{11} > x_{12} > \dots > x_{16} > x_{22} > x_{23} > \dots > x_{66}$. By \cite{buchenberger} we know that we can compute a Gr\"obner basis $G$ for $J$ with respect to this lexicographic order. By Lemma \ref{lem:grob} we know then that a Gr\"obner basis for $\overline{J_{ij}}$ is given by the image of $G_{ij}$ in $R_{ij+1}[x_{ij}]$, where
  \begin{equation}\label{eq:Gij}
    G_{ij} = (G \cap \F_p[x_{ij}, x_{ij+1}, \dots, x_{66}]) \smallsetminus (G \cap \F_p[x_{ij+1}, \dots, x_{66}]).
  \end{equation}
  Here by $x_{ij+1}$ we mean again the variable directly after $x_{ij}$ in the lexicographic order. Since Gr\"obner bases are in particular generating sets, this proves that we can compute a set of generators of the ideal $\overline{J_{ij}}$. 
  
  We have then solved our first problem, as we have reduced the proof of radicality for $J$ to showing radicality for the sequence (\ref{eq:chain}) of univariate ideals $\overline{J_{ij}}$, and we have given a concrete way to compute a generating set for each of them. 
  Before showing that each ideal in the sequence is radical, we have to address the question of the characteristic of the coefficient ring. A priori, the computation of a Gr\"obner basis is sensitive of the characteristic, see \cite{win}*{Ex.\ 1} for some examples. In other words, the Gr\"obner basis $G$ computed for the ideal $J$ over $\F_p$ may differ from the basis $G'$ of $J$ over another coefficient ring $\F_{p'}$. For example, it could have a different number of elements or different degrees. Nevertheless, it is proved by Winkler in \cite{win} how to compute a Gr\"obner basis for $J$ that works for almost all primes. Roughly speaking, we can see $J$ as an ideal with coefficients in $\Q$ and compute a Gr\"obner basis for $J$ over $\Q$. Its image over $\F_p$ will be a Gr\"obner basis for $J$ for almost all primes $p$. For example, we need to exclude those primes dividing the coefficients of $G$. 
  
  In the following, by a normalized reduced Gr\"obner basis we mean a basis such that no proper subset is still a basis. Recall that the Syzygy matrix for a set of generators $G$ has as rows the coefficients of the polynomial relations between the generators of $G$.

  \begin{prop}\label{prop:win}\cite{win}*{Thm.\ 1}
    Let $F = (f_1,\dots , f_m)^{t}$ be a finite sequence of polynomials in $\mathbb{Q}[x_1, \dots ,x_n]$ and $G = (g_1, \dots ,g_r)^t$ the normalized reduced Gr\"obner basis for $F$ in $\mathbb{Q}[x_1, \dots ,x_n]$. Then, for almost all primes $p$ the images $\overline{F} = F~ \mathrm{ mod } ~p$ and $\overline{G} = G ~\mathrm{mod} ~p$ exist and $\overline{G}$ is the normalized reduced Gr\"obner basis for $\overline{F}$ in $\F_p[x_1, \dots ,x_n]$. 
    
    Moreover, the primes for which $\overline{G}$ is not a Gr\"obner basis, called \emph{unlucky primes}, are the divisors of the denominators of the coefficients of $F$ and $G$ and of the coefficients of the entries of the polynomial matrices $Z,Y,R$ defined as
    \begin{equation*}
      G = Z.F, ~~~~~~ F = Y.G, ~~~~~~~ R \text{ the Syzygy matrix of } G.
    \end{equation*}
  \end{prop}

  It follows that our elimination strategy so far is almost independent of the characteristic. Indeed, we can compute a Gr\"obner basis $G$ for $J$ as the ideal in $\Q[x_{11}, \dots, x_{66}]$ generated by the same polynomial equations as in (\ref{eq:ideal}). Then we compute the matrices $Z, Y$ and $R$ as in Proposition \ref{prop:win} and looking at the coefficients of their entries, together with the coefficients of $G$ we obtain the set $U$ of unlucky primes. Now we know that for $p \not \in U$ the image of $G$ modulo $p$ is a Gr\"obner basis for $J \subset \F_p[x_{11}, \dots, x_{66}]$ and the image of the subset $G_{ij}$ as in (\ref{eq:Gij}) is a basis for $\overline{J_{ij}}$. To compute $G$ and $U$ we use the computer algebra software Sagemath \cite{sagemath}. The Gr\"obner basis $G$ is listed in the Appendix \ref{Appendix} and a script for the calculation of $U$ is in Appendix \ref{AppendixB}. The set of unlucky primes turns out to be $U = \{2, 3\}$.

  For $p \not \in U = \{2,3\}$, by the previous discussion, we have to inspect the Gr\"obner basis $G$ of $J$ and its subsets $G_{ij}$. We observe that these satisfy one of the following.
  \begin{enumerate}
    \item $G_{ij}$ is empty. This is the case for the eight variables $\{x_{35}, x_{45}, x_{55}, x_{26}, x_{36}, x_{46}, x_{56}, x_{66}\}$. 
    \item $G_{ij}$ contains a linear polynomial in $x_{ij}$. For $j \le 4$ one possible linear polynomial is the $3 \times 3$ minor of $X$ corresponding to the rows $i, 5, 6$ and the columns $j, 5, 6$, which has leading coefficient $x_{55}x_{66} - x_{56}^2$. The subset $G_{15}$ contains a linear polynomial in $x_{15}$ as well, with leading coefficient $x_{16}$. This polynomial is given by the entry $(5,6)$ of $X^2$.
    \item The remaining subsets $G_{16}$ and $G_{25}$ consist of only one polynomial of degree $2$.
  \end{enumerate} 

  Consider the chain of ideals (\ref{eq:chain}) and our elimination strategy described above. We start with proving that $J_{66} = J \cap \F_p[x_{66}]$ is radical. Since $G_{66}$ is empty, this means that $J_{66} = 0$, so there is nothing to prove. By induction, at step ${ij}$, we have to prove that $J_{ij}$ is radical, knowing that the previous ideal $J_{ij+1}$ is radical. We discuss how to do this in each of the three cases above.

  \begin{proof}[Proof of the empty case] 
    If $G_{ij}$ is empty this means that the image of $J_{ij}$ in the quotient ring $R_{ij+1}[x_{ij}]$ is zero, or in other words that $J_{ij} = J_{ij+1}$. Since we know that the ideal $J_{ij+1}$ preceding $J_{ij}$ in the chain (\ref{eq:chain}) is radical, there is nothing to prove.
  \end{proof}
  
  As a side remark, we note that this is the case for eight variables, which means that $J \cap \F_p[x_{35}, x_{45}, x_{55}, x_{26}, \dots, x_{66}] = 0$. In other words these variables are an \emph{independent set} for $J$, in the sense of \cite{kredel}*{Sec.\ 1}. By \cite{kredel}*{Lem.\ 1.3}, this implies that $J$ has dimension eight, which has already been proved by other methods by Pappas in \cite{pap}*{Sec.\ 4.16}.
  
  \begin{proof}[Proof of the linear case]
    Consider $G_{ij}$ for $j \le 4$, together with $G_{15}$. As we have remarked above, $G_{ij}$, and therefore $\overline{J_{ij}}$ contains a linear polynomial in $x_{ij}$. However, $\overline{J_{ij}}$ is far from being principal and contains polynomials of degree two, as well. Our goal is to reduce to the case of a principal ideal generated by a monic linear polynomial, which is then clearly radical. To do so we can localize at the leading coefficient of the fixed linear polynomial of $G_{ij}$. Localization does not preserve radicality in general, but we can make the following observation.

    \begin{lem} 
      Let $I$ be an ideal in a reduced ring $R$. If $s \in R$ is not a zero divisor modulo $I$ and the localization $I_s$ is radical in $R_s$, then $I$ is radical, too.
    \end{lem} 
    \begin{proof}
      Indeed, if some element $f \in R$ belongs to the radical of $I$, then it belongs to the radical to $I_s$, too and by hypothesis then to $I_s$. This means that for some high enough power of $s$ we have $s^mf \in I$, and since $s$ is not a zero divisor modulo $I$, we deduce that $f \in I$.
    \end{proof}

    Suppose we know that the leading coefficient of the given linear polynomial in $G_{ij}$ is not a zero divisor modulo $J_{ij}$. By the previous lemma it suffices to prove that the localization of $\overline{J_{ij}}$ is radical. By Lemma \ref{lem:grob} we know that the localization of $G_{ij}$ is again a Gr\"obner basis for the localization of $\overline{J_{ij}}$. This basis is however not reduced. Indeed, since we have localized at the leading coefficient of a linear polynomial, it contains a \emph{monic} linear polynomial. It follows that the initial ideal of the localization of $\overline{J_{ij}}$ is generated by the leading term $x_{ij}$ of this monic linear polynomial, which is then a Gr\"obner basis, hence a set of generators. The localization of $\overline{J_{ij}}$ is then principal and generated by a monic linear polynomial, hence clearly radical.

    It remains to prove that the leading coefficient of the chosen linear polynomial in $G_{ij}$ is a non-zero divisor modulo $J_{ij}$. As we have observed this coefficient is $(x_{55}x_{66} - x_{56}^2)$ if $j \le 4$ or $x_{16}$ for $J_{15}$. In order to show that these polynomials are not zero-divisors modulo $J$, we want to use again the theory of Gr\"obner bases, so that with Proposition \ref{prop:win} we can argue almost independently of the characteristic. 

    First, observe that an element $s \in \F_p[x_{11}, \dots, x_{66}]$ is a non-zero divisor modulo $J$ if and only if the division ideal $(J : s) = \{ f \in \F_p[x_{11}, \dots, x_{66}] \mid fs \in J\}$ is equal to $J$. The division ideal can be computed using exclusively Gr\"obner bases by the following result, see for example \cite{gtz}*{Cor.\ 3.2} for a proof.

    \begin{lem}\label{lem:loc}
      Let $I = \langle f_1, \dots, f_r \rangle$ be an ideal in a polynomial ring $R[x_1,\dots, x_m]$, and $s \in R[x_1,\dots,x_m]$. Then it is possible to compute the division ideal $(I : s)$ as follows. Compute a Gr\"obner basis $G$ for the ideal $\langle tf_1, \dots, tf_r, ts - s \rangle \subset R[t, x_1, \dots,x_n]$ with respect to a monomial order such that $t > x_1, \dots, x_n$. Then $(G \cap R[x_1,\dots,x_m])/s$ is a Gr\"obner basis for $(I : s)$.
    \end{lem} 

    Using Sagemath and with the previous lemma one can compute a Gr\"obner basis over $\Q$ for $(J:s)$ for $s$ equal to the leading coefficients $(x_{55}x_{66} - x_{56}^2)$ and $x_{16}$ of the chosen linear polynomials. We compare it to the basis $G$ of $J$, and we obtain that they coincide, hence these leading coefficients are non-zero divisors modulo $J$. It remains to compute the set of unlucky primes for these bases according to Proposition \ref{prop:win}, which is again $\{2, 3\}$. It follows that for $p \not = 2,3$ the elimination ideals $J_{ij}$ for $j \le 4$ as well as $J_{15}$ are radical in $\F_p[x_{11}, \dots, x_{66}]$. A script (with outputs) for the calculations so far can be found in Appendix \ref{AppendixB}.
  \end{proof}

  \begin{proof}[Proof of the quadratic case] 
    It remains to discuss the steps corresponding to the elimination ideals $\overline{J_{25}}$ and $\overline{J_{16}}$. As we have already seen these ideals are principal and generated by a polynomial of degree two. In order to prove that $J_{ij}$ is radical, it suffices then to show that the leading coefficients and discriminants of these quadratic polynomials are non-zero divisors modulo $J$. This implies computing four other Gr\"obner bases as in Lemma \ref{lem:loc} and the corresponding sets of unlucky primes, according to Proposition \ref{prop:win}. We use again Sagemath and obtain that $J_{25}$ and $J_{16}$ are radical for $p \not \in  U$. The set $U$ of unlucky primes is quite large in this case and is listed in Appendix \ref{AppendixB}.
  \end{proof}

  We can conclude that $J \subset \F_p[x_{11}, \dots, x_{66}]$ is radical for $p \not \in  U$. For $p=2$ we have already seen that $J$ is not radical. Observe that the set $U$ consists of primes $\le 809$, see Appendix \ref{AppendixB}, which have already been checked by Pappas in \cite{pap}*{Sec.\ 4.16}. Therefore, for any $p \neq 2$, the ideal $J$ is radical.
\end{proof}

\begin{rem}\label{rem:model}
  We note that the core of the proof of Proposition \ref{prop:radical6} is in the observation that all the arguments and computations we used to prove radicality hold in (almost) any odd characteristic. This is based on the result on Gr\"obner bases of Proposition \ref{prop:win}, which allows us to move to characteristic zero, prove radicality only with Gr\"obner bases calculations, and deduce the same result over almost all positive characteristics $p$. Roughly speaking, we have proved that the ideal $J$ of (\ref{eq:ideal}) is radical over $\Q$ if and only if it is radical over $\F_p$ for almost all primes $p$, and we have indicated how to find the finitely many primes for which this may not hold. 
  
  This is actually not as surprising as it may seem in light of some recent results in model theory. We give here the fundamental idea, which we worked out with S.\ Ramello, who, together with F.\ Jahnke,  pointed us to the relevant literature. In model theory, a \emph{language} consists of all sentences that can be formulated using a given set of symbols. For example, the language of rings consists of all the statements that can be expressed just using the symbols $+, \cdot, 0, 1$, see \cite{marker}*{Sec.\ 1} for a detailed explanation. As a consequence of the \emph{compactness theorem}, see \cite{marker}*{Cor.\ 2.2.10}, any statement in the language of rings is true in an algebraically closed field of characteristic zero if and only if it is true in an algebraically closed field of characteristic $p$ for every $p$ large enough. At a first glance, the statement ``the ideal $J \subset R$ is radical'', which is equivalent to the statement ``for every $f \in R$, if $f^n \in J $ then $f \in J$'', seems not to belong to the language of rings, as it requires using the quantifier $\forall$, the set of natural numbers (for the exponent and the degree of $f$) and the quantifier $\exists$ ($f \in I$ means that there exists a linear combination of the generators of $I$ that is equal to $f$). However, it is proved in \cite{barnicle}*{Sec.\ 5.1} that if $R$ is a polynomial ideal, the statement ``$J$ is radical'' can actually be formulated in an equivalent way without quantifiers and without the full set of natural numbers. Therefore, it can be expressed in this case in the language of rings. It follows that an ideal $J \subset \Z[x_1, \dots, x_m]$ is radical over $\Q$ if and only if it is radical over $\F_p$ for $p$ large enough.

  We note that the compactness theorem of model theory is highly non-constructive, that is, it does not indicate how to find the prime $p_0$ such that, if an ideal is radical in characteristic zero, then it is radical in characteristic $p > p_0$. Since our goal was to prove flatness of the Rapoport-Zink space in any odd characteristic, a purely model-theoretical approach would not have been sufficient.
\end{rem}

\begin{rem} 
  Another important idea in the proof of Proposition \ref{prop:radical6} is the reduction of the proof of radicality to the case of one variable. This approach can be applied to any ideal, as long as an algorithm or criterion for proving radicality of the resulting univariate polynomial ideals is known. In our case, we have linear or quadratic polynomials, and we have seen how to prove radicality in these cases by using only Gr\"obner bases. Our strategy can be applied to the ideals $J(2,n)$ of Proposition \ref{prop:pap}, as well, and we have carried out the computations for $n \le 8$, and obtained that these ideals are radical. 
\end{rem}

\section{Vertex lattices and modular lattices}\label{sec:lattices}

Now that we have proved that the scheme $\N_{2,6}$ is flat over $\mathcal{O}_E$, we can turn to the description of its geometry. The results in this and the next section are actually true for $\N_{2,n}$ in any dimension $n$, from Section \ref{sec:closedpoints} on we will restrict again to the case $n =6$. 

As we have mentioned in the introduction, the object of our studies is $\bar{\N}_{2,n}^{0}$, the reduction modulo $\pi$ of the open and closed formal subscheme of $\N_{2,n}$ consisting of quadruples where the height of the quasi-isogeny $\rho$ is zero. More precisely, the moduli functor $\bar{\N}_{2,n}^{0}$ parametrizes quadruples $(X,\lambda, \iota, \rho)$, where $X$ is a $p$-divisible group of height $2n$ and dimension $n$, and where $\lambda$ is a principal quasi-polarization whose Rosati involution induces on $\mathcal{O}_E$ the non-trivial automorphism over $\Q_p$. Since the conjugate embeddings $\psi_{0,1}: E \rightarrow \breve{E}$ coincide  modulo $\pi$, and since we have fixed $s = 2$, Pappas' and Kottwitz's conditions reduce to
\begin{equation}\label{eq:wedge}
  \bigwedge^{3} (\iota(\pi) \mid \Lie(X))= 0.
\end{equation}
Moreover, $\rho: X \rightarrow \mathbb{X}\times_{\F}S$ is now a quasi-isogeny of height $0$ which is $\mathcal{O}_E$-linear and such that $\rho^{*}(\lambda_{\mathbb{X}})$ and $\lambda$ differ locally on $S$ by a factor in $\Z_p^{\times}$. 

We first study the $\F$-valued points of $\bar{\N}_{2,n}^{0}$. By Dieudonn\'e theory to the fixed $p$-divisible group $\mathbb{X}$ corresponds a unique free $W(\F)$-module of rank equal to the dimension $n$ of $\mathbb{X}$. We consider $N$, the rational Dieudonn\'e module of $\mathbb{X}$, that is the vector space obtained by tensoring with the field of fractions $W_{\Q} = \Quot(W(\F))$. The action $\iota_{\mathbb{X}}$ of $\mathcal{O}_E$ induces an action of the field $E$ on $N$. Since by definition of $\iota_{\mathbb{X}}: \mathcal{O}_E \rightarrow \End(\mathbb{X})$ the action of any element in $\mathcal{O}_E$ on $\mathbb{X}$ is an endomorphism of $\mathbb{X}$ as $p$-divisible group, the action of $E$ on the rational Dieudonn\'e module $N$ commutes with the Frobenius and Verschiebung maps on $N$. We denote by $\Pi$ the action of $\pi$ on $N$. Last, the principal quasi-polarization $\lambda_{\mathbb{X}}$ induces a skew-symmetric $W_{\Q}$-bilinear form $\langle \cdot , \cdot \rangle$ on $N$ satisfying 
\begin{align*}
  & \langle Fx, y \rangle = \langle x, Vy \rangle^{\sigma} \\
  & \langle \iota_{\mathbb{X}}(a)x, y \rangle = \langle x, \iota_{\mathbb{X}} (\bar{a}) y \rangle,
 \end{align*}
for any $x, y \in N$ and any $a \in E$. For a $W(\F)$-lattice $M \subset N$, that is a free $W(\F)$-submodule of $N$ of rank $n$, we denote by $M^{\vee}$ the lattice $\{ x \in N \mid \langle x, M \rangle \subset W(\F)\}$, and call it the \emph{dual} of $M$ with respect to the alternating form on $N$. In the following, we write an exponent over an inclusion of lattices $M_1 \subset^m M_2$ to indicate the index, \textit{i.e.}\ the length of the quotient module $M_2/M_1$.

The following lemma is the analogue of \cite{rtw}*{Prop.\ 2.2}, and it is proved in the same way. For completeness, we recall here their proof with the modifications due to the different signature. 
\begin{lem}\label{lem:3.1}
  Associating to a point in $\bar{\N}_{2,n}^{0}(\mathbb{F})$ its Dieudonn\'{e} module defines a bijection of $\bar{\N}_{2,n}^{0}(\mathbb{F})$ with the set of $W(\F)$-lattices 
  \begin{equation*}
    \{ M \subset N \mid M^{\vee} = M,~ \Pi M \subset M,~ pM \subset VM \subset^n M,~ VM \subset^{\le 2} VM + \Pi M\}, 
  \end{equation*}
\end{lem}
\begin{proof}
  Given a quadruple $(X, \lambda, \iota, \rho)$ in $\bar{\N}_{2,n}^{0}(\mathbb{F})$, the quasi-isogeny $\rho$ from $X$ to the fixed $p$-divisible group $\mathbb{X}$ translates into an inclusion of the Dieudonn\'e module $M$ of $X$ into the rational module $N$ of $\mathbb{X}$. Since $\lambda$ is a principal polarization, $M$ is a self-dual lattice. The stability of $X$ under the action $\iota$ of $\mathcal{O}_E$, together with the $\mathcal{O}_E$-linearity of $\rho$, is equivalent to the stability of $M$ under the action $\Pi$ of $\pi$ on $N$. Condition (\ref{eq:wedge}) says that the action of $\Pi$ on $\Lie(X) = M/VM$ has rank at most $2$, which is equivalent to the index condition in the last inclusion. Conversely, if a $W(\F)$-lattice $M \subset N$ satisfies all these properties, by the inclusions $pM \subset V M \subset M$, we see that also $FM \subset M$. Then $M$ corresponds to a $p$-divisible group $X$ with additional structure $(\iota, \lambda)$ as claimed and with a quasi-isogeny $\rho$ to $\mathbb{X}$ induced by the inclusion of $M$ in $N$. 
\end{proof}
 
As in \cite{rtw}*{Sec.\ 2} we also consider the Hermitian $E$-vector space $C$ constructed as follows.  Let $\eta \in W^{\times}$ be such that $(\eta\pi)^2 = p$ and consider the $\sigma$-linear map $\tau \vcentcolon =  \eta \Pi V^{-1} : N \rightarrow N$. Recall that the $p$-divisible group $\mathbb{X}$ is supersingular, which means that all the slopes of its Newton polygon are $\tfrac{1}{2}$. Therefore, $\tau$ has all slopes zero.  We define $C$ as the $n$-dimensional $\Q_p$-vector space consisting of the points of $N$ that are fixed by $\tau$. Since the action of $E$ on $N$ commutes with the Frobenius and Verschiebung maps, the action of $\Pi$ commutes with $\tau$. The structure of $E$-vector space on $C = N^\tau$ is then induced by the action of $\Pi$ on $N$. Last, we note that there is an isomorphism $C \otimes_{\Q_p}W_{\Q} \xrightarrow{\sim} N$ such that $\mathrm{id}_C \otimes \sigma$ corresponds to $\tau$.   

As remarked in \textit{loc.cit.}, the restriction of the skew-symmetric form of $N$ induces an alternating bilinear form on $C$ with values in $\mathbb{Q}_p$, which we denote again by $\langle \cdot , \cdot \rangle$. In particular, it satisfies 
\begin{equation*}
  \langle \Pi x, y \rangle = - \langle x, \Pi y \rangle, ~~~ \text{for } x,y \in C.
\end{equation*}
Therefore, we can define a symmetric $E$-bilinear form on $C$ by setting
\begin{equation*}
  (x, y) \vcentcolon = \langle \Pi x, y \rangle.
\end{equation*}
As remarked in \cite{rtw}*{Sec.\ 2}, we can also define a Hermitian form $h$ on $C$ via the formula $$ h(x, y) \vcentcolon= \langle \Pi x, y \rangle + \langle x, y\rangle\pi. $$ This form in particular satisfies 
\begin{align}\label{eq:forms}
  \langle x, y \rangle &= \tfrac{1}{2} \trace_{E/\mathbb{Q}_p}(\pi^{-1} h(x,y)) \\
  ( x, y )  &= \tfrac{1}{2} \trace_{E/\mathbb{Q}_p}(h(x,y)),
\end{align}
for all $x, y \in C$. We extend the Hermitian form of $C$ (and consequently the symmetric and alternating forms, too) onto $C \otimes_{E} \breve{E}$ by setting
\begin{equation*}
  h(v\otimes a, w \otimes b) = a \cdot \sigma(b) \cdot h(v, w).
\end{equation*}
\begin{lem}
  We denote by $M^{\vee}, M^{\sharp}, M^{\perp}$ the duals of an $\mathcal{O}_{\breve{E}}$-lattice $M$ in $C \otimes_E \breve{E}$ respectively for the alternating, Hermitian and symmetric from. Then we have
\begin{equation*}
  M^{\vee} = M^{\sharp} = \Pi M^{\perp}.
\end{equation*} 
\end{lem} 
\begin{proof}
  If $x \in M^{\vee}$, then for every $m \in M$ the value of $\langle x, m \rangle$ is an element of $\mathcal{O}_{\breve{E}}$. Since $M$ is an $\mathcal{O}_{\breve{E}}$-lattice we have $\Pi M \subset M$, and therefore $\langle \Pi x, m \rangle = -\langle x, \Pi m \rangle$ is an integer, too. From the definition of the Hermitian form $h$, it follows then that $x \in M^{\sharp}$. The other inclusion is clear from the relation (\ref{eq:forms}) above between the alternating form and the trace of the Hermitian form.

  If $x \in M^{\perp}$, then by definition of the symmetric form the value of $\langle \Pi x, m \rangle$ is an integer for all $m \in M$, and therefore $\Pi x \in M^{\vee}$. Conversely, if $x \in M^\vee$, then $(\Pi^{-1}x, m) = \langle \Pi (\Pi^{-1} x), m \rangle = \langle x, m \rangle$ is an integer for all $m \in M$ and therefore $\Pi^{-1}x \in M^{\perp}$.
\end{proof}

\begin{lem}\label{eq:V(k)} Associating to a $p$-divisible group its Dieudonn\'e module defines a bijection of $\bar{\N}_{2,n}^0(\F)$ with the set of $\mathcal{O}_{\breve{E}}$-lattices
\begin{equation*} 
  \mathcal{V}(\mathbb{F}) = \{ M \subset C \otimes_E \breve{E} \mid M^{\sharp} = M, ~ \Pi \tau(M) \subset M \subset^n  \Pi^{-1}\tau(M), ~ M \subset^{\le 2} (M + \tau(M))  \}.     
\end{equation*}
\end{lem}
\begin{proof}
  This is simply a reformulation of Lemma \ref{lem:3.1} in terms of the map $\tau$ and the isomorphism $C \otimes_E \breve{E} \xrightarrow{\sim} N$.
\end{proof}

\begin{rem}
  In the following sections we will often have to distinguish between two, sometimes quite different, cases. Consider the discriminant of the Hermitian space $C$. It is given by the image of $(-1)^{\frac{n(n-1)}{2}} \det V$ in the order-$2$ group $\Q_p^{\times} /\text{Norm}_{E/\Q_p}(E^{\times})$. We say that the form is \emph{split} if the discriminant is the trivial element in this group, respectively \emph{non-split} if it is non-trivial. As noted in \cite{rtw}*{Rem.\ 4.2} for even dimension $n$ both cases, $C$ split and non-split, can appear. This only depends on the choice of $\mathbb{X}$ used to define the moduli space $\N_{2,n}$, and in \textit{loc.cit.}\ it is shown how to construct examples for both cases. If the dimension is odd, since we can multiply $\lambda_{\mathbb{X}}$ by a unit in $\Z_p$, one can assume without loss of generality that the discriminant of $C$ is $1$, compare also \cite{rtw}*{Rem.\ 4.2} and the references there. 
\end{rem}

We show now how to associate to any lattice $M$ in $\V(\mathbb{F})$ a unique minimal $\tau$-stable $\mathcal{O}_E$-lattice $\Lambda(M) \subset C$ such that $M \subset \Lambda(M) \otimes_{\mathcal{O}_E} {\mathcal{O}_{\breve{E}}}$. The construction is the same as that of \cite{rtw}*{Sec.\ 4}, however, due to the different index appearing in the last inclusion in Lemma \ref{eq:V(k)}, the resulting lattice $\Lambda(M)$ will satisfy a weaker property. In the following we denote by $\pi$ both the element of $E$ and its action $\Pi$ on $N$ or $C$.
\begin{defn}
  Let $\Lambda$ be an $\mathcal{O}_E$-lattice in $C$.  
  \begin{enumerate}[nolistsep]
    \item \cite{rtw}*{Def.\ 3.1} We say that $\Lambda$ is a \emph{vertex lattice} if it satisfies $\pi \Lambda \subset \Lambda^{\vee} \subset \Lambda$. The index of $\Lambda^{\vee}$ in $\Lambda$ is called the \emph{type} of $\Lambda$. It is proved in \cite{rtw}*{Lem.\ 3.2} that the type of a vertex lattice is an even integer. Moreover, all even integers between $0$ and $n$ appear as the type of some vertex lattice in $C$, unless $n$ is even and the form is non-split, in which case type $n$ does not appear.
    \item  \cite{jac}*{Sec.\ 2} We say that $\Lambda$ is an \emph{$m$-modular lattice} if $\pi^m \Lambda = \Lambda^{\vee} \subset \Lambda$. For even $m$, $m$-modular lattices always exist, and a full classification is given in \cite{jac}*{Sec.\ 8}. In particular, $0$-modular lattices are simply self-dual lattices, while $1$-modular lattices are vertex lattices of type $n$, and by the previous point, they only exist if the dimension $n$ is even and the form is split.
    \item In this paper we say that $\Lambda$ is a \emph{$2$-vertex lattice} if $\pi^2\Lambda \subset \Lambda^\vee \subset \Lambda$. Clearly vertex lattices and $2$-modular lattices are also $2$-vertex lattices.
  \end{enumerate}
\end{defn}

Given a lattice $M \in \V(\F)$,  for each positive integer $j$ we consider the lattice
\begin{equation*}
T_j \vcentcolon = M + \tau(M) + \dotsm + \tau^j(M).    
\end{equation*} 
We also denote by $\tau_j$ the image of $T_j$ under $\tau$. It is clear from the definition that $T_{j+1} = T_j + \tau_j$ and that $\tau_{j-1} \subset T_{j} \cap \tau_j$. From the properties of $M$ it follows that for every $j$ the lattice $T_j$ satisfies
\begin{equation}\label{eq:abc}
  \pi T_j \subset T_j,~~~  \pi\tau(T_j) \subset T_j \subset \pi^{-1} \tau(T_j), ~~~ T_j \subset ^{\le 2} T_j + \tau(T_j),
\end{equation}
and similarly for $\tau_j$. By \cite{rz}*{Prop.\ 2.17} there is an integer $d$ such that $T_d = T_{d+1}$ and the minimal such integer satisfies $d \le n-1$, where $n$ is again the dimension of the $\breve{E}$-vector space $N$.

Consider the chain of inclusions 
\begin{equation}
  M = T_0 \subset T_1 \subset \dotsm \subset T_d.
\end{equation} 
We now give a series of rather combinatorial remarks which will be of key importance for the proof of Proposition \ref{prop:1} later.
\begin{rem}\label{lem:1} 
  For any $i = 1, \dots , d$ the lattices $T_{i-1}$ and $\tau_{i-1}$ have the same index in $T_i$. This follows from the fact that they are both contained in $T_i$, by definition, and since $\tau$ has slopes zero, they have the same volume. By the second isomorphism theorem for modules, it also follows that the index of the inclusion $T_i \subset T_{i+1}$ is the same as that of the inclusion $T_i \cap \tau_i \subset T_i$. 
\end{rem}
\begin{rem}\label{lem:3}
  There is an index $1\le k \le d$ such that 
  \begin{equation}\label{eq:ek}
    M = T_0 \subset^2 \dotsm \subset^2 T_k \subset^1 \dotsm \subset^1 T_d.
  \end{equation}  Indeed, let $k$ be the minimal integer such that $T_{k-1}\subset^2 T_k \subset^1 T_{k+1}$, with the convention that if all inclusions have index $1$ or $2$, we simply say $k = 0$, respectively $k =d$. Assume $0 < k < d$, we show by induction that for all $k \le i <d$ the index of $T_i$ in $T_{i+1}$ is one.  For $i = k$ this is just the definition of $k$.  Assume $k < i< d$. By induction, we have $T_{i-1} \subset^1 T_i$ and by Remark \ref{lem:1} this implies $ \tau_{i-1 } \subset^1 T_i$. We know that
  \begin{equation*}
    \tau_{i-1 } \subset T_i \cap \tau_i \subsetneq T_i,
  \end{equation*} where the second inclusion is proper as $i < d$ and therefore $T_i$ is not $\tau$-stable. Since $\tau_{i-1}$ has index $1$ in $T_i$ we have that $\tau_{i-1} = T_i \cap \tau_i \subset^1 T_i$.  By the previous remark we conclude that $T_i \subset^1 T_{i+1}$, which concludes the proof of (\ref{eq:ek}).

  Let $k$ be as above, then we claim that
  \begin{align*}
    \tau_{i-1} &= T_i \cap \tau_i ~~~\text{if } i \neq k, \\
    \tau_{k-1} &\subset^1 T_k \cap \tau_k \subset^1 T_k.
  \end{align*} 
  We have already proved the case $i > k$. For $i < k$ we have $T_{i-1} \subset^2 T_i \subset^2 T_{i+1}$. Then by Remark \ref{lem:1} and the first inclusion it follows $\tau_{i-1} \subset^2 T_i$. By the same remark and the second inclusion we also have $T_i \cap \tau_i \subset^2 T_i$, from which we deduce equality. At step $k$ we have $T_{k-1} \subset^2 T_k \subset^1 T_{k+1}$.  From the first inclusion we obtain $\tau_{k-1} \subset^2 T_k$, while from the second inclusion and Remark \ref{lem:1} it follows $T_{k-1} \cap \tau_{k-1} \subset^1 T_{k}$.
\end{rem}

\begin{rem}\label{lem:5} 
  From the inclusions $\pi\tau(M) \subset M \subset \pi^{-1}\tau(M)$ in the definition (\ref{eq:V(k)}) of $\V(\F)$ it follows that $\pi T_2 = \pi M + \pi\tau(M) + \pi\tau^2(M) \subset \tau(M)$. As in the proof of \cite{rtw}*{Prop.\ 4.2} we deduce that for $i \ge 2$ 
  \begin{align*}
    T_i & = (M + \tau(M) + \dotsm + \tau^i(M) )\\
    & = (M + \tau(M) + \tau^2(M)) + \tau(M + \tau(M) + \tau^2(M)) + \dotsm + \tau^{i-2}(M + \tau(M) + \tau^2(M)) \\
    & = T_2 + \tau(T_2) + \dotsm + \tau^{i-2}(T_2) \\
    & \subset \pi^{-1}\tau(M) + \dotsm + \pi^{-1}\tau^{i-1}(M) \\
    & \subset \pi^{-1}\tau_{i-2}.
  \end{align*} 
  So for any $2 \le i \le d$ we have $\pi T_i \subset \tau_{i-2} \subset T_{i-1} \cap \tau_{i-1}$. In particular, it follows that $\pi T_d \subset T_{d-1}$. Since $T_d$ is $\tau$-stable we have 
  \begin{equation*}
    \pi T_d \subset \bigcap_{m \in \mathbb{Z}} \tau^m(T_{d-1}).
  \end{equation*} 
  By Remark \ref{lem:3} we know that for $k < i < d$ the intersection $T_i \cap \tau_i$ coincides with $\tau_{i-1}$. After applying this recursively to the previous equation we obtain 
  \begin{equation}\label{eq:k}
    \pi T_d \subset \bigcap_{m \in \mathbb{Z}} \tau^m(T_k) \subset T_k \cap \tau_k \subset T_k.
  \end{equation} 
  Since $\tau_{k-1} \subset^1 T_k \cap \tau_k$ it is in general not true that $\pi T_d \subset \tau_{k-1}$. However, by the previous discussion we know that $\pi T_k \subset \tau_{k-1} $ hence we can at least say that $\pi^2 T_d \subset \tau_{k-1}$ or equivalently, by $\tau$-stability, $\pi^2 T_d \subset T_{k-1}$. By $\tau$-stability, again, 
  \begin{equation*}
    \pi^2 T_d \subset \bigcap_{m \in \mathbb{Z}} \tau^m(T_{k-1}).
  \end{equation*} 
  Again we can apply Remark \ref{lem:3} recursively since for $i < k$ we still have $T_i \cap \tau_i=\tau_{i-1}$. We can then conclude that 
  \begin{equation}\label{eq:pi2}
    \pi^2 T_d \subset \bigcap_{m \in \mathbb{Z}} \tau^m(M) \subset M.
  \end{equation}
\end{rem}

\begin{rem}\label{rem:7} 
  If $k = 0$ or $k = d$ we know by Remark \ref{lem:3} that for all $i$ the intersection $T_i \cap \tau_i$ coincides with $\tau_{i-1}$. Therefore, when we apply this to (\ref{eq:k}) we obtain $\pi T_d \subset M$. If $d = k+1$ then arguing as in the second part of the previous remark we obtain $\pi T_d \subset M$. Note that these are not the only possible cases, one may still have $\pi T_d \subset M$ even if $0 < k < d - 2$.
\end{rem}

In order to prove the next proposition we need one more observation concerning $\tau$-stable lattices in $N$.

\begin{lem}\label{rem:basis}
  Let $\mathcal{L}$ be a $\tau$-stable $\mathcal{O}_{\breve{E}}$-lattice in $N$, then $\mathcal{L}$ has a basis consisting of $\tau$-stable elements.
\end{lem}
\begin{proof}
   By the isomorphism $C \otimes_{E} \breve{E} \xrightarrow{\sim} N$ given above it follows that $N$ has a $\tau$-stable basis. Let $\Lambda$ be the $\tau$-stable lattice spanned by such a basis. Since $\mathcal{L}$ is an $\mathcal{O}_{\breve{E}}$-lattice in $N$, there is an integer $i$ such that $\pi^i \Lambda \subset \mathcal{L}$. It follows that $\mathcal{L}$ contains at least one element that is $\tau$-stable. We show by induction that $\mathcal{L}$ has a basis consisting of $\tau$-stable elements. Suppose $N$ has dimension one. Up to multiplication by powers of the uniformizer $\pi$, we can assume that there is an element $v \in \mathcal{L}$ that is $\tau$-stable and such that if $av \in \mathcal{L}$ for some $a \in \breve{E}$, then $a \in \mathcal{O}_{\breve{E}}$. We show that $v$ generates $\mathcal{L}$. Again, observe that there is an integer $i$ such that $\pi^i\mathcal{L} \subset \mathcal{O}_{\breve{E}} \cdot v$. Therefore, for any element $l \in \mathcal{L}$ there is an integer $j$ such that $l = a\pi^j v$ for some $a \in \mathcal{O}_{\breve{E}}^{\times}$. By our choice of $v$, the coefficient $a\pi^j$ has to be an integer, hence $\mathcal{L} \subset \mathcal{O}_{\breve{E}} \cdot v \subset \mathcal{L}$, which concludes the proof for the one-dimensional case.

  Suppose now that $N$ has dimension $n +1 \ge 2$ and let $\mathcal{L} = \tau(\mathcal{L})$ be a lattice in $N$. We can again find a $\tau$-stable element $v \in \mathcal{L}$, and up to multiplication by powers of $\pi$ we can assume that if $av \in \mathcal{L}$ then $a \in \mathcal{O}_{\breve{E}}$. Consider the $n$-dimensional quotient space $N/\breve{E}v$ and observe that $\tau$ commutes with the quotient map as $v$ is $\tau$-stable. It follows that the image of $\mathcal{L}$ in this quotient is again a $\tau$-stable lattice and hence by induction it has a basis consisting of $\tau$-fixed elements. Lift this basis to $\tau$-stable elements $\{e_1, \dots, e_n\}$ of $N$, which is possible since $N$ has a $\tau$-stable basis. Then we have that $\mathcal{L}$ has a basis of the form $\{a_0v, e_1 - a_1v, \dots e_n - a_nv \}$ for suitable $a_i \in \breve{E}$. By the choice of $v$ it immediately follows that we can assume $a_0 = 1$. If $a_i \in \mathcal{O}_{\breve{E}}$, then the corresponding $\tau$-stable vector $e_i$ is already in $\mathcal{L}$, and we can substitute it to $e_i -a_iv$ in the basis of $\mathcal{L}$. Assume that for some $i$ the coefficient $a_i \in \breve{E}$ is not an integer. Observe that since $\mathcal{L}$ is $\tau$-stable we have that $\mathcal{L}$ contains the element $(e_i - a_iv) - \tau(e_i-a_iv) = (\sigma(a_i) - a_i)v$ for each $i$. By definition of $v$ it follows that $(\sigma(a_i) - a_i) \in \mathcal{O}_{\breve{E}}$. We can then write $a_i = b_i + c_i$ with $c_i \in \mathcal{O}_{\breve{E}}$ and $b_i = \sigma(b_i) \in \breve{E}$. We can substitute $e_i - a_iv$ in the basis of $\mathcal{L}$ by the $\tau$-stable element $e_i - b_iv$, which concludes the proof.
\end{proof}

\begin{prop}\label{prop:1}
  For any lattice $M$ in $\V(\F)$ there is a unique minimal $\mathcal{O}_E$-lattice $\Lambda(M) \subset C$ such that $M \subset \Lambda(M) \otimes_{\mathcal{O}_E} \mathcal{O}_{\breve{E}}$. Moreover, $\Lambda(M)$ is a $2$-vertex lattice.
\end{prop}

\begin{proof}
  Consider $M$ in $\V(\F)$ and the corresponding lattice $T_d$ as above. As in \cite{rtw}*{Prop.\ 4.1}  we define $\Lambda(M) \vcentcolon = T_d^{\tau} = T_d \cap C$.  Since $T_d$ is $\tau$-stable, by Lemma \ref{rem:basis} it has a basis consisting of $\tau$-stable elements. It follows that $\Lambda(M)$ is an $\mathcal{O}_E$-lattice in $C$ and that $T_d = \Lambda(M) \otimes_{\mathcal{O}_E} \mathcal{O}_{\breve{E}}$. By definition of $T_d$, it follows that $\Lambda(M)$ is the minimal lattice in $C$ which, after tensoring with $\mathcal{O}_{\breve{E}}$, contains $M$. 
  
  If $d = 0$ or $d = 1$ it directly follows from the definition of $M$ that $\pi T_d \subset M \cap \tau(M)$. If $2 \le d$, by Remark \ref{lem:5} we know that 
  \begin{equation*}
    \pi^2 T_d \subset \bigcap_{l \in \mathbb{Z}} \tau^l(M) \subset M \cap \tau(M) \cap \dotsm \cap \tau^d(M) = T_d^\vee,
  \end{equation*} where the last equality follows from the fact that $M$ is self-dual and that $\tau$ commutes with taking duals, as it has slopes zero.  This proves that $\Lambda(M)$ is a $2$-vertex lattice.
\end{proof}

\begin{rem}\label{rem:vl2}
  Observe that $\Lambda(M)$ is a vertex lattice if and only if $\pi T_d \subset M$. Indeed, if this is the case then arguing as above we obtain $\pi T_d \subset M \cap \tau(M) \dotsm \cap \tau^d(M) = T_d^\vee$. Conversely, if $\pi T_d \subset T_d^{\vee}$, since $T_d^\vee$ is contained in $M$ we have that $\pi T_d \subset M$. Note that if $\Lambda$ is a vertex lattice and $\Lambda(M) \subset \Lambda$, it follows that $\Lambda(M)$ is a vertex lattice as well. Indeed, if $\Lambda(M) \subset \Lambda$ by taking duals and by definition of ($2$-) vertex lattice, we have that
  \begin{equation*}
    \pi \Lambda(M) \subset \pi \Lambda \subset \Lambda^\vee \subset \Lambda(M)^\vee \subset \Lambda(M) \subset \Lambda.
  \end{equation*}
\end{rem}

Let $\Lambda$ be a $2$-vertex lattice, we denote 
\begin{align*}
  \V_{\Lambda}(\F) &= \{ M \in \V(\F) \mid \Lambda(M) \subset \Lambda\} \\
  \V_{\Lambda}^{\circ}(\F) &= \{ M \in \V(\F) \mid \Lambda(M) = \Lambda\}.
\end{align*}

We recall some results from \cite{rtw}*{Sec.\ 3} about the set of vertex lattices in order to compare them to the behavior of $2$-vertex lattices. For $n$ even and non-split form, and for odd $n$ let  $\mathscr{L}$ denote the set of vertex lattices. If $n$ is even and the form is split, which is the only case where vertex lattices of type $n$ exist, we let $\mathscr{L}$ be the set of vertex lattices of type different from $n-2$. In both cases, we give $\mathscr{L}$ the structure of a simplicial complex as follows. We say that two vertex lattices $\Lambda_1$ and $\Lambda_2$, at least one of which is of type $\le n-2$, are neighbors if $\Lambda_ 1 \subset \Lambda_2$ or vice versa. For two vertex lattices both of type $n$, we say that they are neighbors if their intersection is a vertex lattice of type $n-2$. Then an $r$-simplex of $\mathscr{L}$ is a subset of $r$ vertex lattices which are pairwise neighbors. Let $\mathrm{SU}(C)$ be the special group of unitary similitudes of $(C, h)$, \textit{i.e.}\ the subgroup of linear transformations of $C$ preserving the Hermitian form $h$ and having determinant one. As remarked in \cite{rtw}*{Sec.\ 3} there is an action of $\mathrm{SU}(C)(\Q_p)$ on $\mathscr{L}$ which preserves the simplicial complex structure we just defined.

\begin{prop}\label{prop:vlproperties} Keep notation as above.
  \begin{enumerate}
    \item \cite{rtw}*{Prop.\ 3.4} There is a $\rm SU(C)(\Q_p)$-equivariant isomorphism between $\mathscr{L}$ and the Bruhat-Tits simplicial complex of $\rm SU(C)$ over $\Q_p$. Moreover, $\mathscr{L}$ is connected.
    \item \cite{rtw}*{Prop.\ 4.3, 6.7} Let $\Lambda_1$ and $\Lambda_2$ be two vertex lattices in C. Then $\V_{\Lambda_1} \subset \V_{\Lambda_2}$ if and only if $\Lambda_1 \subset \Lambda_2$ and equality holds if and only if the two lattices are also equal. It follows
    \begin{equation*} 
      \V_{\Lambda} = \bigsqcup_{\Lambda' \subset \Lambda} \V^\circ_{\Lambda'} 
    \end{equation*}
    and every summand is non-empty.
    \item \cite{rtw}*{Prop.\ 4.2} The intersection $\V_{\Lambda_1} \cap \V_{\Lambda_2}$ is non-empty if and only if $\Lambda_1 \cap \Lambda_2$ is a vertex lattice, in which case it coincides with $\V_{\Lambda_1 \cap \Lambda_2}$.
  \end{enumerate}
\end{prop}

For $2$-vertex lattices the situation is more complicated, and there is not a full analogue of the results above, compare Remark \ref{rem:2vl} below. First, we need to recall the \emph{Jordan splitting} for lattices in the Hermitian space $C$. It is proved in \cite{jac}*{Prop.\ 4.3} that any $\mathcal{O}_E$-lattice in $C$ has a canonical decomposition as a direct sum of modular lattices in possibly smaller-dimensional Hermitian subspaces. Moreover, this decomposition is compatible with taking duals, \textit{i.e.}\ if $L$ is a lattice in $C$ with Jordan splitting
\begin{equation*}
  L = \bigoplus_{1 \le \lambda \le t} L_\lambda,
\end{equation*}
with each $L_{\lambda}$ modular, then its dual $L^\vee$ has Jordan splitting $L^\vee = \bigoplus_{1 \le \lambda \le t} (L_\lambda)^\vee$. Indeed, observe that the dual of an $m$-modular lattice is by definition $(-m)$-modular.
\begin{prop}\label{prop:2vl}
  Consider the set of $2$-vertex lattices, \textit{i.e.}\ the set of $\mathcal{O}_E$-lattices $\Lambda$ in $C$ such that $\pi^2\Lambda \subset \Lambda^\vee \subset \Lambda$.
  \begin{enumerate}
    \item The set of $2$-modular lattices is in bijection with the set of vertex lattices of type $0$, hence with the $0$-simplices of the Bruhat-Tits building of $\rm SU(C)$ over $\Q_p$. 
    \item Every $2$-vertex lattice is contained in some, possibly non-unique, $2$-modular lattice. Hence,
    \begin{equation*}
      \V(\F) = \bigcup_{\Lambda  \in \{2\text{-modular}\}} \V_{\Lambda}(\F),
    \end{equation*}
    and for every $2$-modular lattice $\Lambda$ already the set $\V_{\Lambda}^{\circ}(\F)$ is non-empty.
  \end{enumerate}
\end{prop}
\begin{proof}
  Let $\Lambda$ be a $2$-modular lattice. Thus, we have $\pi^2\Lambda = \Lambda^\vee \subset \pi \Lambda \subset \Lambda$. Observe that  $(\pi\Lambda)^\vee = \pi^{-1}(\Lambda^\vee) = \pi^{-1} (\pi^2\Lambda) = \pi \Lambda$, which means that $\pi \Lambda$ is a self-dual vertex lattice, that is a vertex lattice of type $0$. Conversely, given a vertex lattice $L$ of type $0$, the lattice $\pi^{-1}L$ satisfies $(\pi^{-1}L)^\vee = \pi L^\vee = \pi L = \pi^2(\pi^{-1}L)$, hence it is a $2$-modular lattice. 
  
  If $L$ is a $2$-vertex lattice, which means $\pi^2 L \subset L^\vee$, the summands appearing in its Jordan decomposition can only be $0,1$ or  $2$-modular lattices. Therefore, it is enough to prove that every $0$ or $1$-modular lattice is contained in a $2$-modular lattice.  If $L$ is $0$-modular, then consider $\pi^{-1} L$, which we have already seen is a $2$-modular lattice, and it contains $L$. If $\pi L = L^\vee$, by the connectedness of the simplicial complex $\mathscr{L}$ and its bijection with the Bruhat-Tits building for $\mathrm{SU}(C)(\Q_p)$ as recalled in the previous proposition, we know that $L$ contains a self-dual lattice $ \pi L \subset L_0^\vee=L_0 \subset L$. Then the $2$-modular lattice $\pi^{-1}L_0$ contains $L$.

  The non-emptiness of the set $\V_{\Lambda}^{\circ}$ will actually follow from the results of Section \ref{sec:closedpoints}, in particular from Lemma \ref{lem:spW} and \ref{lem:spW2c}.
\end{proof}

\begin{rem}\label{rem:2vl}
  We have seen that there is a bijection between the set of $2$-modular lattices and of $0$-modular lattices. One could ask if there is a bijection between the set of generic $2$-vertex lattices and vertex lattices, along the lines of the proposition above. It is true, for example, that for a vertex lattice $L$, if $L$ is not $1$-modular, one obtains a $2$-vertex lattice by taking $\pi^{-1}L^\vee$. The converse does however not work. Given a $2$-vertex lattice $\Lambda$ (that is of course not a vertex lattice), we would have to consider $L = (\pi\Lambda)^\vee = \pi^{-1}\Lambda^\vee$. Its dual, which is $\pi\Lambda$, is contained in $L$, since $\pi^2 \Lambda \subset \Lambda^\vee$ and therefore $L^\vee = \pi\Lambda \subset \pi^{-1}\Lambda^\vee = (\pi\Lambda)^\vee = L$. However, it is not true in general that $L^\vee = \pi\Lambda \supset \pi L = \pi (\pi\Lambda)^\vee = \Lambda^{\vee}$. For example, consider a $2$-vertex lattice with Jordan decomposition $\Lambda = \Lambda_1 \oplus \Lambda_2$ where $\Lambda_1$ is a $2$-modular lattice and $\Lambda_2$ is a $0$-modular (hence self-dual) lattice. Then $\Lambda^\vee = \pi^2 \Lambda_1 \oplus \Lambda_2$ and it is not contained in $\pi \Lambda$.

  This is one of the reasons why, unlike \cite{rtw}*{Sec.\ 4}, we are not going to attempt at a stratification of $\V_{\Lambda}$ in terms of sets $\V_{\Lambda'}^\circ$ for smaller $2$-vertex lattices $\Lambda'$. The other main reason is that it does not seem to be feasible to describe one single such stratum in terms of Deligne-Lusztig varieties, as we are going to note in the next section, for example in Remark \ref{rem:hereditary}.
\end{rem}

\section{Deligne-Lusztig varieties for the symplectic and orthogonal group}\label{sec:dlvs}
\noindent
In this section we recall some facts about (generalized) Deligne-Lusztig varieties and focus on three families of  varieties for the symplectic and orthogonal group. Their relevance will become clear in the next section.

\subsection{Reminder on Deligne-Lusztig Varieties}
Deligne-Lusztig varieties were first introduced in \cite{DL}. Here, as in the original paper, we give a description in terms of their $\F$-valued points. We also follow the notation of \cite{rtw}*{Sec.\ 5} and the references in there.

Let $G$ be a connected reductive group over a finite field $\F_q$. Let $T \subset B \subset G$ be respectively a maximal torus defined over $\F_q$ and a Borel subgroup over $\F_q$ containing it. Fix an algebraic closure $\F$ of $\F_q$. Let $W$ be the Weyl group $N_G(T)(\F)/T(\F)$. Denote by $\Phi$ the Frobenius on $G(\F)$. Consider the \emph{relative position map}
\begin{equation*}
  \inv \colon G/B \times G/B \rightarrow W
\end{equation*}
which sends a pair $(g_1, g_2)$ to the unique element $w \in W$ such that $g_1^{-1}g_2 \in BwB$. For $w \in W$ the corresponding \emph{Deligne-Lusztig variety}  is 
\begin{equation*}
  X_B(w) = \{ g \in G/B \mid \inv(g, \Phi(g)) = w\}.
\end{equation*}

Deligne-Lusztig varieties can be related to Schubert varieties via the local model diagram by G\"ortz and Yu \cite{GY}*{5.2}. Consider the quotient map $\pi \colon G \rightarrow G/B$ and denote by $L$ its composition with the Lang map $g \mapsto g^{-1}\Phi(g)$
\begin{equation*}
  G/B \xleftarrow{\pi} G \xrightarrow{L} G/B.
\end{equation*}
Then we have that Deligne-Lusztig varieties and Schubert cells are smoothly equivalent to each other under these maps 
\begin{equation*}
  \pi^{-1}(X_B(w)) = L^{-1}(BwB/B).
\end{equation*}
It follows that $X_B(w)$ is smooth, of pure dimension $\ell(w)$ and that the singularities of the closure $\overline{X_B(w)}$ are smoothly equivalent to the singularities of the Schubert variety $\overline{BwB}/B$, compare \cite{GY}*{5.2}. The closure $\overline{X_B(w)}$ is stratified by Deligne-Lusztig varieties for smaller elements in the Bruhat order on $W$ as follows
\begin{equation}\label{eq:btord}
  \overline{X_B(w)} = \bigsqcup_{w' \le w} X_B(w').
\end{equation}
This is a consequence of the analogue closure relations for Schubert cells and the local model diagram, compare \cite{GY}*{Sec.\ 5} for a detailed proof.

In the next sections we will also be interested in some \emph{generalized} Deligne-Lusztig varieties, which are defined as the analogue in a partial flag variety. More precisely, let $\Delta = \{\alpha_1, \dots, \alpha_n\}$ be the set of simple roots associated to the datum $(B,T)$. Recall that to each simple root $\alpha_i$ corresponds a simple reflection $s_i$ in the Weyl group. Let $I$ be a subset of the simple roots. We denote by $W_{I}$ the subgroup of $W$ generated by the simple reflections corresponding to $I$ and by $P_I$ the standard parabolic subgroup $BW_IB$. The partial flag variety $G/P_I$ parametrizes then parabolic subgroups of type $I$. Again, one can define a relative position map 
\begin{equation*}
  \inv \colon G/P_I \times G/P_I \rightarrow W_I\backslash W/ W_I,
\end{equation*}
and for a class $w \in W_I\backslash W /W_I$ the corresponding generalized Deligne-Lusztig variety
\begin{equation*}
  X_{P_I}(w) = \{ g \in G/P_I \mid \inv(g, \Phi(g)) = w\}.
\end{equation*}
We recall a result by  Bonnaf\' e and Rouquier \cite{br}*{Thm.\ 2} concerning irreducibility. 
\begin{thm}\label{thm:br}
  Let $I \subset \Delta$ and $w \in W_I \backslash W/W_I$. The corresponding generalized Deligne-Lusztig variety $X_{P_I}(w)$ is irreducible if and only if $W_I w$ is not contained in any proper $\Phi$-stable standard parabolic subgroup of $W$.
\end{thm}
Moreover, by the results of \cite{GY}*{Sec.\ 5} we can see that $X_{P_I}(w)$ is equidimensional of dimension
\begin{equation}\label{eq:dim}
  \dim(X_{P_I}(w)) = \ell_I(w) - \ell(w_I).
\end{equation} 
Here $w_I$ is the longest element in the subgroup $W_I$ and $\ell_I(w)$ denotes the maximal length of an element in the double coset  $W_IwW_I$.

We aim to give a description of the closure of a generalized Deligne-Lusztig variety analogous to that in (\ref{eq:btord}). To do so, we have to study the set $W_I\backslash W / W_I$. By \cite{bb}*{Prop.\ 2.4.4} there is a system of representatives  of $W_I\backslash W / W_I$, which we denote by ${}^IW^I$ and consists of a minimal length element in each double coset. Such a minimal length element is actually unique by \cite{schremmer}*{Prop.\ 4.22a}, and for every element $y \in W$ there is a decomposition $y = z_I x z_I'$ with $x \in {}^IW^I$ and $z_I, z_I' \in W_I$ such that $\ell(y) = \ell(z_I) + \ell(x) + \ell(z_I')$.

Before we can prove the analogue for generalized Deligne-Lusztig varieties of the closure relations (\ref{eq:btord}), we need the following combinatorial results. 
\begin{lem}\label{lem:btI}
  For $x_1, x_2$ in the system of minimal length representatives ${}^IW^I$ the following are equivalent
  \begin{itemize}
    \item[(i)] $x_1 \le x_2$ in the Bruhat order on $W$,
    \item[(ii)] there are elements $y_1 \le y_2$ such that $y_i \in W_Ix_iW_I$,
    \item[(iii)] for every $y_1 \in W_Ix_1W_I$ there exists $y_2 \in W_Ix_2W_I$ such that $y_1 \le y_2$. 
  \end{itemize}
\end{lem}
\begin{proof}
  The implications (i) $ \Rightarrow$ (ii) and (iii)$ \Rightarrow$ (ii) are clear. The implication (ii) $\Rightarrow$ (i) is proved in \cite{schremmer}*{Prop.\ 4.22c}.

  Assume (i) holds and fix $y_1 \in W_Ix_1W_I$. Consider the factorization $y_1 = z_Ix_1z_I'$ such that $\ell(y_1) = \ell(z_I) + \ell(x_1) + \ell(z_I')$, as given in \cite{schremmer}*{Prop.\ 4.22a}. Let $z_I = s_1 \cdots s_q$ be a reduced expression for $z_I$. If $\ell(s_qx_2) = \ell(x_2) +1$ since $x_1 \le x_2$ we have $s_qx_1 \le s_qx_2$. Otherwise, by the so-called \emph{lifting property} of the Bruhat order, compare \cite{bb}*{Prop.\ 2.2.7}, we have $s_qx_1 \le x_2$. By induction on the length of $z_I$, we obtain an element $y'_2 \in W_Ix_2$ such that $z_Ix_1 \le y'_2$. By repeating the same construction on the right with a reduced expression of $z_I'$ we obtain an element $y_2 \in W_Ix_2W_I$ such that $y_1 = z_Ix_1z_I' \le y_2$.
\end{proof}

The following result is proved in \cite{br} and allows us to move between generalized Deligne-Lusztig varieties for two different parabolic subgroups.
\begin{lem}\cite{br}*{Eq.\ 2}\label{lem:fIJ}
  Let $I \subset J$ be two subsets of simple reflections in $W$ and $P_I \subset P_J$ the corresponding standard parabolic subgroups. Let $f_{IJ}: G/P_{I} \rightarrow G/P_{J}$ be the morphism of varieties that sends a parabolic subgroup of type $I$ to the unique parabolic of type $J$ containing it. Let $w \in W$ and $X_{P_J}(w)$ the corresponding generalized Deligne-Lusztig variety. Then its preimage under $f_{IJ}$ is the union of Deligne-Lusztig varieties 
  \begin{equation*}
    f_{IJ}^{-1}(X_{P_J}(w)) = \bigcup_{W_IxW_{\Phi(I)} \subset W_JwW_{\Phi(J)}} X_{P_I}(x).
  \end{equation*} 
\end{lem}

We are now ready to prove the analogue of the closure relations \ref{eq:btord} for generalized Deligne-Lusztig varieties.

\begin{lem}\label{lem:closure}
  Let $P_I$ be the standard parabolic subgroup of type $I$, with $I$ a $\Phi$-stable subset of simple reflections, and $w \in {}^IW^I$. The closure in $G/P_I$ of the generalized Deligne-Lusztig variety $X_{P_I}(w)$ satisfies
  \begin{equation*}
    \overline{X_{P_I}(w)} = \bigcup_{w' \in {}^IW^I, w' \le w} X_{P_I}(w').
  \end{equation*}
\end{lem}
\begin{proof}
  We consider the morphism of projective varieties $f: G/B \rightarrow G/P_I$ which maps a Borel subgroup to the unique parabolic subgroup of type $I$ containing it. This map is surjective by definition of parabolic subgroups. Since $f$ is surjective and, as a morphism between projective varieties, it is closed, we have
  \begin{equation*}
    \overline{X_{P_I}} = \overline{f(f^{-1}(X_{P_I}(w)))} = f(\overline{f^{-1}(X_{P_I}(w))}).
  \end{equation*}
  Moreover, the preimage under $f$ of any generalized Deligne-Lusztig variety satisfies
  \begin{equation*}
    f^{-1}(X_{P_I}(w)) = \bigcup_{x \in W_IwW_I} X_B(x).
  \end{equation*} This follows by setting $I = \emptyset$ in Lemma \ref{lem:fIJ}. Since the union on the right runs over a finite set, by the closure relations (\ref{eq:btord}) for classical Deligne-Lusztig varieties, we have
  \begin{equation*}
    \overline{f^{-1}(X_{P_I}(w))} = \overline{\bigcup_{x \in W_IwW_I} X_B(x)} = \bigcup_{x \in W_IwW_I} \overline{X_B(x)} =   \bigcup_{x \in W_I w W_I } \bigcup_{x' \le x} X_B(x').
  \end{equation*}
  By Lemma \ref{lem:btI} there is a bijection of sets
  \begin{equation*}
    \{x' \in W \mid x' \le x, \text{for some } x \in W_IwW_I \} \longleftrightarrow \{x' \in W_IyW_I \mid y \in {}^IW^I,  y \le w\}.
  \end{equation*}
  Putting these observations together, we conclude that the closure of $X_{P_I}(w)$ is
  \begin{align*}
    \overline{X_{P_I}(w)} &= f(\overline{f^{-1}(X_{P_I}(w))}) = f \big(\bigcup_{x \in W_I w W_I } \bigcup_{x' \le x} X_B(x') \big) = f\big (\bigcup_{\substack{y \in {}^IW^I \\ y \le w}}  \bigcup_{y' \in W_IyW_I} X_B(y') \big )  \\
    &= f \big (\bigcup_{\substack{y \in {}^IW^I \\ y \le w}} f^{-1}(X_{P_I}(y)) \big) = \bigcup_{\substack{y \in {}^IW^I \\ y \le w}} X_{P_I}(y).
  \end{align*}
\end{proof}

The remainder of this chapter is dedicated to the study of some families of Deligne-Lusztig varieties which will be relevant in the sequel. In particular, we are going to decompose some generalized Deligne-Lusztig varieties in terms of other such varieties for smaller parabolic subgroups. The strategy was inspired to us by reading the proofs of \cite{lus}*{Sec.\ 4} and \cite{rtw}*{Sec.\ 5}, and it is based on the morphism introduced in Lemma \ref{lem:fIJ} and the following observation. 

\begin{lem}\label{lem:iso}
  With notation as in Lemma \ref{lem:fIJ} above, suppose the morphism $f_{IJ}: G/P_I \rightarrow G/P_J$ induces a bijection between the closed points of $f_{IJ}^{-1}(X_{P_J}(w)) = \bigcup_{W_IxW_{\Phi(I)} \subset W_JwW_{\Phi(J)}} X_{P_I}(x) $ and $X_{P_J}(w)$. Then $f_{IJ}$ induces an isomorphism between these two varieties.
\end{lem}
\begin{proof}
  First observe that $f_{IJ}: G/P_I \rightarrow G/P_J$ is a smooth morphism. Indeed, both flag varieties are smooth, as they are homogeneous spaces for $G$. The fibers of this morphism are all isomorphic to $P_J/P_I$, hence they are again smooth, as homogeneous spaces for $P_J$ and all have the same dimension. By so-called \textit{miracle flatness}, see for example \cite{gw}*{B.9}, this map is flat with smooth fibers, hence smooth. Recall that the base-change of a smooth map is smooth, therefore, the morphism $f_X$ defined as the base change of $f_{IJ}$ along the following diagram is smooth. 
  \[\begin{tikzcd}
    \bigcup_{W_IxW_{\Phi(I)} \subset W_JwW_{\Phi(J)}} X_{P_I}(x) \arrow{r}{f_X} \arrow[hookrightarrow]{d} & X_{P_J}(w) \arrow[hookrightarrow]{d}\\%
    {G/P_I} \arrow{r}{f_{IJ}}& G/P_J   
  \end{tikzcd}\]
 Here the vertical arrows are the just immersions of the generalized Deligne Lusztig varieties in the corresponding flag varieties. 
 
 By hypothesis, we know that $f_X$ gives a bijection between the sets of closed points. By \cite{gw}*{Rem.\ 12.16}, to prove that $f_X$ is quasi-finite, it is enough to prove that it has finite fibers on $k$-valued points, for any algebraically closed field $k$. Since it is injective on closed points, this is clearly the case, hence the morphism $f_{X}$ is quasi-finite. Recall that a smooth morphism of finite type is étale if and only if it is smooth and quasi-finite. It is then enough to prove that $f_{X}$ is surjective and universally injective. Indeed, since $f_{X}$ is étale, universally injective implies that it is an open immersion. Since an open immersion is an isomorphism onto its image, if $f_{X}$ is surjective we are done. Recall that universally injective is equivalent to the diagonal morphism being bijective on $k$-valued points for any field $k$. Since $f_X$ is a morphism between projective schemes, it is proper, hence the diagonal morphism is a closed immersion and therefore it is already injective on $k$-valued points. Moreover, for a scheme of finite type over an algebraically closed field, as in our case, the set of closed points is very dense, see \cite{gw}*{Prop.\ 3.35}. Therefore, there is no proper closed subscheme containing all closed points. It follows that we can test if the diagonal morphism is surjective on closed points, which is equivalent to the map being injective. Last, by the same argument, $f_{X}$ is surjective since it is surjective on closed points. 
\end{proof}

\subsection{Some Deligne-Lusztig varieties for the symplectic group}\label{sec:dlvsp} 
In this section we study a family of  Deligne-Lusztig varieties that is the analogue of the one analyzed in  \cite{rtw}*{Sec.\ 5}, and contains it as a proper subset. We follow here their notation. In particular, we aim to find a stratification as in \textit{loc.cit.}\ that will be related to the decomposition over the admissible set studied in the last section of this paper. 

Let $V$ be a vector space of dimension $2m$ over $\F_p$, endowed with a skew-symmetric form $\langle ~,~\rangle$. We fix a basis $e_1, \dots, e_{2m}$ of $V$ such that 
\begin{equation*}
  \langle e_i , e_{2m+1-j} \rangle = \delta_{i,j} ~~~~ i,j = 1,\dots, m.
\end{equation*} 
Let $T \subset B \subset \Sp(V) = G$ be the torus of diagonal matrices in $\mathrm{Sp}_{2m}$ and the Borel subgroup of upper triangular matrices. Then the simple reflections generating the Weyl group $W$ can be enumerated as follows
\begin{itemize}[nolistsep] 
  \item for $1 \le i \le m-1$ the reflection $s_i$ exchanges $e_i$ with $e_{i+1}$ and $e_{2m -i}$ with $e_{2m + 1-i}$,
  \item the reflection $s_m$ exchanges $e_m$ with $e_{m+1}$.
\end{itemize}

We say that a subspace $U$ of $V$ is isotropic if it is contained in its orthogonal space with respect to the symplectic form. The maximal dimensional isotropic subspaces are called Lagrangian subspaces and have dimension $m$. As remarked in \cite{rtw}*{Sec.\ 5.2}, if $P$ is the Siegel parabolic, \textit{i.e.}\ the standard parabolic corresponding to the  reflections $\{s_1, \dots, s_{m-1}\}$, then the flag variety $G/P$ parametrizes the Lagrangian subspaces of $V$. In particular, we are interested in the subvariety $S_V$ of $G/P$  given by 
\begin{equation*}
   S_V = \{U \subset V, ~\text{Lagrangian} \mid \dim(U \cap \Phi(U)) \ge m -2\}.
\end{equation*} 
Observe that this can be considered as the analogue in signature $(2,n-2)$ of the variety defined in \textit{loc.cit.}, and contains it as a proper closed subvariety.
\begin{lem}\label{lem:normal}
  $S_V$ can be identified with the closure of the generalized Deligne-Lusztig variety $X_{P}(s_ms_{m-1}s_m)$ in $G/P$. In particular, $S_V$ is normal with isolated singularities.
\end{lem} 
\begin{proof} 
  If $U \in S_V$ then the relative position $\inv(U, \Phi(U))$ is either the identity, the class of $s_m$ or of $s_ms_{m-1}s_m$ in $W_0\backslash W/W_0$, where $W_0$ denotes the subgroup of the Weyl group corresponding to $P$, thus generated by $ \{s_1, \dots, s_{m-1}\}$. It follows that $S_V$ is the disjoint union 
  \begin{equation} 
    S_V = X_P(1) \sqcup X_P(s_m) \sqcup X_P(s_ms_{m-1}s_m). 
  \end{equation} 
  Observe now that the identity and $s_m$ are the only minimal length representatives in $W_0\backslash W/W_0$ smaller than $s_ms_{m-1}s_m$ in the Bruhat order. By Lemma \ref{lem:closure} this proves the first claim. As in \textit{loc.cit.}, the second statement follows from G\"ortz and Yu's local model diagram and the fact that generalized Schubert varieties are normal with isolated singularities.
\end{proof}
By the discussion in \cite{rtw}*{Prop.\ 5.5} we also know that the union $X_P(1) \sqcup X_P(s_m)$ corresponds to the Lagrangian subspaces $U$ in $S_V$ such that $\dim (U \cap \Phi (U)) \ge m -1$.

\subsubsection{The six-dimensional case}\label{sec:dlvsp6}
We construct a stratification of $S_V$ which will be relevant especially in the study of the admissible set in Section \ref{sec:adlvs}. In this paper we restrict to the case $n = 6$, but a similar stratification can be defined for any dimension. We consider the following parabolic subgroups
\begin{itemize}
  \item $P_3 = P$, the Siegel parabolic, it corresponds to the reflections $\{s_1, s_2\}$.
  \item $P_2$, the standard parabolic corresponding to the reflection $\{s_1\}$. It is the stabilizer of the partial isotropic flag $\langle e_1, e_2 \rangle \subset \langle e_1, e_2, e_3 \rangle$.
  \item $P_2'$, the standard parabolic corresponding to $\{s_2\}$. It is the stabilizer of $\langle e_1 \rangle \subset \langle e_1, e_2, e_3 \rangle$.
  \item $B$ the Borel subgroup, it can be identified with the stabilizer of the complete isotropic flag $\langle e_1 \rangle \subset \langle e_1, e_2 \rangle \subset \langle e_1, e_2, e_3 \rangle$.
\end{itemize}
In order to give a stratification of $S_V$ we follow the approach of \cite{rtw}*{Sec.\ 5}. In particular, we recursively show that the restriction of the quotient maps $G/P_i \rightarrow G/P_{i-1}$ for $P_{i} \subset P_{i-1}$ gives a bijection on closed points. By Lemma \ref{lem:iso}, this will produce isomorphisms $X_{P_{i-1}}(w) \cong X_{P_i}(w_1) \sqcup X_{P_i}(w_2)$ for suitable $w_1, w_2$ depending on $w$.
\begin{lem}\label{lem:SVdim}
  There is a decomposition of $S_V$ as disjoint union of locally closed subvarieties
  \begin{equation}\label{eq:strC}
    \begin{aligned}
     S_V = &X_P(1) \sqcup X_{P_2}(s_3) \sqcup X_{B}(s_3s_2)  \sqcup X_{P_2'}(s_3s_2s_3) ~\sqcup  \\ 
    & X_B(s_3s_2s_1)  \sqcup X_B(s_3s_2s_3s_1) \sqcup X_B(s_3s_2s_3s_1s_2),
    \end{aligned}
  \end{equation}
  and this decomposition is a stratification such that the closure of each stratum is given by the union of the strata with smaller dimension. The variety $S_V$ is irreducible of dimension $5$.
  \end{lem}
\begin{proof}
In \cite{rtw}*{Prop.\ 5.5} a stratification of $X_P(1) \sqcup X_P(s_3)$ is already given, namely as the union of locally closed subvarieties 
\begin{equation}\label{eq:decrtw}
  X_P(1) \sqcup X_P(s_3) \cong X_P(1) \sqcup X_{P_2}(s_3) \sqcup X_{B}(s_3s_2) \sqcup X_B(s_3s_2s_1).
\end{equation}
Each of the four generalized Deligne-Lusztig varieties appearing on the right-hand side parametrizes isotropic flags of the form 
\begin{equation*}
  U \cap \Phi(U) \cap \dotsm \cap \Phi^i(U) \subset \dotsm \subset U \cap \Phi(U)\subset U
\end{equation*}
for $i = 0,1,2,3$, respectively, and such that the $(3-i)$-dimensional subspace $U \cap \Phi(U) \cap \dotsm \cap \Phi^i(U)$ is $\Phi$-stable.  It follows that the irreducible components of $X_P(1)$, $X_{P_2}(s_3)$ and $X_{B}(s_3s_2)$ are indexed over the $\Phi$-stable subspaces $W \subset V$ of dimension $3, 2$ and $1$, respectively. 

Similarly, we want to construct a stratification of the remaining subvariety $X_P(s_3s_2s_3)$ as disjoint union of locally closed subspaces. First, we want to prove that there is a decomposition
\begin{equation*}
  X_{P}(s_3s_2s_3) \cong  X_{P_2'}(s_3s_2s_3) \sqcup X_{P_2'}(s_3s_2s_3s_1),
\end{equation*} 
and that $X_{P_2'}(s_3s_2s_3s_1)$ is open and dense in this union. By Lemma \ref{lem:fIJ} we know that $X_{P_2'}(s_3s_2s_3) \sqcup X_{P_2'}(s_3s_2s_3s_1)$ is the preimage of $X_{P}(s_3s_2s_3)$ under the morphism $G/P_2' \rightarrow G/P$. Therefore, by Lemma \ref{lem:iso}, it is enough to show that this morphism induces a bijection on closed points. Let $k$ be an algebraically closed field. We know that the $k$-points of $X_P(s_3s_2s_3)$ are Lagrangian subspaces $U \subset V_k$ such that $\dim(U \cap \Phi(U)) = 3-2 = 1$. Therefore, we can consider the partial isotropic flag $U \cap \Phi(U) \subset^2 U$, which is a $k$-point of $G/P_2'$. It belongs to either $X_{P_2'}(s_3s_2s_3)(k)$ or $ X_{P_2'}(s_3s_2s_3s_1)(k)$ depending on whether $U \cap \Phi(U)$ is stable under the Frobenius or not. This defines a map between the  $k$-points of $X_{P}(s_3s_2s_3)$ and $X_{P_2'}(s_3s_2s_3) \sqcup X_{P_2'}(s_3s_2s_3s_1)$. This map is the inverse on closed points of the map $G/P_2'(k) \rightarrow G/P(k)$ which sends a flag $U_1 \subset U$ to its second subspace. By Lemma \ref{lem:iso}, it follows that the restriction of the quotient map gives the desired isomorphism. The subvariety $X_{P_2'}(s_3s_2s_3)$ is open and dense in the union above by Lemma \ref{lem:closure}.

Our goal is to obtain a decomposition of $S_V$ which we can later relate to the simplicial complex $\mathscr{L}$ of the previous section and to the admissible set of Section \ref{sec:adlvs}. To do so, we need to further decompose the open subvariety $X_{P'_2}(s_3s_2s_3s_1)$. Consider again the map $G/B \rightarrow G/P_2'$ which on $k$-points sends a complete flag $U_1 \subset U_2 \subset U_3$ to the partial flag $U_1 \subset U_3$ obtained by forgetting its middle term. By Lemma \ref{lem:fIJ} we know that the preimage of $X_{P'_2}(s_3s_2s_3s_1)$ under this map is $X_B(s_3s_2s_3s_1) \sqcup X_B(s_3s_2s_3s_1s_2)$. Again, by Lemma \ref{lem:iso}, it is enough to show that this map induces a bijection between the sets of closed points. To do so, we construct its inverse (as a map of sets). We claim that the desired map is obtained by sending a partial isotropic flag $U \cap \Phi(U) \subset U$ in $X_{P_2'}(s_3s_2s_3s_1)(k)$ to the complete flag 
\begin{equation*}
  U \cap \Phi(U) \subset U \cap (\Phi(U) \cap \Phi^2(U))^{\vee} \subset U.
\end{equation*}
Indeed, let $U \cap \Phi(U) \subset^2 U$ be a partial isotropic flag in $X_{P_2'}(s_3s_2s_3s_1)(k)$, we can assume that it has this form by the previous construction on closed points. We have already observed that partial flags $U \cap \Phi(U) \subset^2 U$ in $X_{P_2'}(s_3s_2s_3s_1)(k)$ satisfy $U \cap \Phi(U) \cap \Phi^2(U) = 0$. This means that the one-dimensional subspace $\Phi(U) \cap \Phi^2(U)$ is not contained in $U$. Consider the subspace $U \cap (\Phi(U) \cap \Phi^2(U))^{\vee}$, where the exponent denotes the orthogonal subspace with respect to the alternating form on $V$. We can compute its dimension as follows
\begin{align*}
  \dim (U \cap (\Phi(U) \cap \Phi^2(U))^{\vee}) &= 6 - \dim((U \cap (\Phi(U) \cap \Phi^2(U))^{\vee})^{\vee}) \\ &= 6 - \dim (U^{\vee} + (\Phi(U) \cap \Phi^2(U))) \\ &= 6 - \dim (U + (\Phi(U) \cap \Phi^2(U))) = 2, 
\end{align*}
where we use the fact that $U$ is Lagrangian, hence it coincides with its orthogonal, and has dimension $3$, and the fact that the $1$-dimensional space $(\Phi(U) \cap \Phi^2(U))$ is not contained in $U$.  Therefore, the flag above is actually complete. It follows from Lemma \ref{lem:iso} that the base change to $X_{P'_2}(s_3s_2s_3s_1)$ of the quotient morphism $G/B \rightarrow G/P_2'$ is an isomorphism 
\begin{equation*}
  X_{P'_2}(s_3s_2s_3s_1) \cong X_B(s_3s_2s_3s_1) \sqcup X_B(s_3s_2s_3s_1s_2).
\end{equation*}

Since $S_V$ is the closure in $G/P$ of $X_P(s_3s_2s_3)$ and the latter contains $X_B(s_3s_2s_3s_1s_2)$ as an open and dense subset (by the previous decomposition and the closure relations \ref{eq:btord}), we deduce that $S_V$ is irreducible and of dimension $5$.
\end{proof}

We show that the stratification of $S_V$ given above has good \textit{hereditary} properties in the sense of \cite{rtw}*{Prop.\ 5.7}. Roughly speaking, this means that the strata of $S_V$ that are not irreducible, can be identified with a union of varieties of the form $S_{V'}$ for suitable smaller-dimensional, symplectic spaces $V'$. The next proposition is proved in the same way as \cite{rtw}*{Prop.\ 5.7}. For completeness, we recall here the main ideas of the proof.

\begin{lem}\label{lem:SVstr} 
  Denote $S_0 = X_P(1)$, $S_1 = X_{P_2}(s_3)$ and $S_2 = X_B(s_3s_2) \sqcup X_{P_2'}(s_3s_2s_3)$.
  \begin{enumerate}[nolistsep]
    \item[(i)] The irreducible components of $S_i$ are in bijection with the $\Phi$-stable isotropic subspaces of $V$ of dimension $3 -i$.
    \item[(ii)] Let $W$ be such an isotropic subspace. The irreducible component $X_W$ of $S_i$ corresponding to $W$ by (i) is a Deligne-Lusztig variety for the symplectic group $\Sp(W^\vee/W)$ of rank $3-i$. The closure of $X_W$ in $S_i$ is isomorphic to $S_{W^\vee/W}$, the variety defined in the same way as $S_V$ but for the symplectic vector space $W^\vee/W$.
  \end{enumerate}
\end{lem}
\begin{proof}
  As in \cite{rtw}*{Prop.\ 5.7} we observe that the generalized Deligne-Lusztig varieties appearing in $S_i$ parametrize isotropic flags
  \begin{equation*}
    U_{3-i} \subset U_{3-i+1} \subset \dotsm \subset U_3,
  \end{equation*} 
  where $U_{3-i}$ is $\Phi$-stable. Then $U_{3-i}$ is a $\F_p$-rational isotropic subspace of $V$ of dimension $3-i$. For $i = 0,1$, we already know by \cite{rtw}*{Prop.\ 5.7} that the subvariety $X_W$ of points of $S_i$ such that $U_{3-i}$ above is equal to a fixed subspace $W$ can be identified with the Deligne-Lusztig variety for $\Sp(W^{\vee}/W)$ and a Coxeter element. In case $i = 2$ the subvariety $X_W$ can be identified with the union of two Deligne-Lusztig varieties for $\Sp(W^{\vee}/W)$, one for the Coxeter element $s_3s_2$ and one for the element $s_3s_2s_3$ in the Weyl subgroup of type $C_2$ generated by $s_2, s_3$. 
  
  In all three cases, such elements have full support in the corresponding Weyl groups, hence by Theorem \ref{thm:br} the subvarieties $X_W$ are irreducible.
  Last, as remarked in \textit{loc.cit.}, for a $\Phi$-stable subspace $W$ of dimension $i$, the closure of $X_W$ in $S_V$ is 
  \begin{equation*}
    \overline{X_W} = \{U \in S_V, W \subset U\},
  \end{equation*} which can be identified with the closed variety $S_{W^\vee/W}$ by sending $U$ to its image in the quotient $W^\vee/W$.
\end{proof}

\subsection{Some Deligne-Lusztig varieties for the orthogonal group}\label{sec:dlvso} 
In this section, following the notation of \cite{lus}*{Sec.\ 2} we introduce two other families of Deligne-Lusztig varieties that will be relevant in the next sections. Let $V$ be an $n$-dimensional $\F$-vector space with a fixed $\F_p$-structure. Denote with $\Phi$ again its Frobenius morphism. Let $(~ ,~) : V\times V \rightarrow \F$ be a non-degenerate symmetric bilinear form on $V$, such that $(\Phi(x), \Phi(y)) = (x, y)^p$. We say that a subspace $U$ of $V$ is isotropic if it is contained in its orthogonal with respect to the symmetric form. A maximal isotropic subspace of $V$ has dimension $\lfloor \tfrac{n}{2}\rfloor$. If the dimension of $V$ is even, we say that the form is \emph{split} if there exists a maximal $\Phi$-stable isotropic subspace, which has then dimension $\tfrac{n}{2}$, otherwise the form is called \emph{non-split} and a maximal $\Phi$-stable isotropic subspace has dimension $\tfrac{n}{2} -1$. 

As in \textit{loc.cit.}\ we fix a Borel subgroup of $\SO(V)$ corresponding to an isotropic flag of length $\lfloor \frac{n-1}{2}\rfloor$. Recall that if the dimension of $V$ is even, the correspondence between parabolic subgroups of $\SO(V)$ and isotropic flags in $V$ is slightly more involved than, for example, for the symplectic group, compare \cite{conrad_notes}*{App.\ T} and the references there. Roughly speaking, the usual map which sends a flag to its stabilizer is a bijection onto the set of parabolic subgroups of $\SO(V)$ if and only if we restrict to isotropic flags where subspaces of dimension $\frac{n}{2}$ and $\frac{n}{2}-1$ do not appear together.

In the next sections we will be interested in the following family of generalized Deligne-Lusztig varieties for the special orthogonal group $\SO(V)$. 
\begin{defn}\label{def:ya}\cite{lus}*{Def.\ 2}
  Given an integer $a \ge 1$ consider the locally closed subscheme $Y_a$ of the projective space  $\mathbb{P}(V)$ defined by the homogeneous equations
  \begin{equation*}
    (x, \Phi^i(x)) = 0 \text{~for~} 0 \le i \le a-1, \text{~and~} (x, \Phi^a(x)) \neq 0.
  \end{equation*} 
  We also consider the variety $Y_{\infty}$ defined by the equations $(x, \Phi^i(x)) = 0 \text{~for all~} i \ge 0$.
\end{defn}

\begin{lem}\label{lem:lus3}\cite{lus}*{Lemma 3}
  Let
  \begin{equation*}
   a_0 =  \begin{cases}
      \frac{n}{2} -1 &\text{ if } \dim(V) \text{ is even and the form is split}\\
      \frac{n}{2} &\text{ if } \dim(V) \text{ is even and the form is non-split} \\
      \frac{n-1}{2} &\text{ if } \dim(V) \text{ is odd}.
    \end{cases}\,
\end{equation*}
Then $Y_a = \emptyset $ for any $a > a_0$. Moreover, $Y_{a_0}$ can be identified with the Deligne-Lusztig variety $X_B(w)$ for some $\Phi$-Coxeter element, respectively in the non-split case with the union $X_B(w) \cup X_B(\Phi(w))$. Here a $\Phi$-Coxeter element is an element of $W$ that is obtained as the product of one reflection for each $\Phi$-orbit in $W$.
\end{lem}
We fix some more notation. Assume first that $V$ has even dimension $n = 2d$. We fix a basis $e_1, \dots, e_d, f_1, \dots, f_d$ such that 
\begin{equation*}
  (e_i, e_j) = (f_i, f_j) = 0, ~~~ (e_i,f_j) = \delta_{i,j}.
\end{equation*} 
Moreover, if the form is split we can assume that all the basis vectors are fixed by $\Phi$, otherwise we can assume that $\Phi$ exchanges $e_d$ with $f_d$ and fixes the other vectors, compare \cite{conrad_notes}*{App.\ T}. Let $T \subset B \subset G = \SO(V)$ denote the diagonal torus and the Borel of upper triangular matrices in the orthogonal group. Then the simple reflections generating the Weyl group can be enumerated as follows

\begin{itemize}[nolistsep]
  \item For $1 \le i \le d-1$ the reflection  $t_i$ exchanges $e_i$ with $e_{i+1}$ and $f_i$ with $f_{i+1}$. 
  \item The reflection $t_d$ exchanges $e_{d-1}$ with $f_d$ and $e_d$ with $f_{d-1}$.
\end{itemize}
If the form is split, the action of $\Phi$ on the Weyl group is trivial, otherwise, $\Phi$ exchanges the reflection $t_{d-1}$ with $t_d$.

Suppose now that $V$ has odd dimension $n = 2d +1$, then there is a basis $e_0, e_1, \dots, e_d, f_1, \dots, f_d$ of $V$ such that 
\begin{equation*}
  (e_i, e_j) = (f_i, f_j) = 0, ~~~ (e_i,f_j) = \delta_{i,j},~~~ (e_0, e_0) = 1.
\end{equation*} The Weyl group is generated in this case by the reflections $t_1, \dotsm t_{d-1}$ defined as in the case $n = 2d$ while the reflection $t_d$ only exchanges $e_d$ with $f_d$. The action of the Frobenius on $W$ is trivial.

We study the variety $R_V$ in the projective space $\mathbb{P}(V)$ given by 
\begin{equation}\label{eq:RV}
  R_V = \{ x \in \mathbb{P}(V) \mid (x, x) = (x, \Phi(x)) = 0  \} = Y_{\infty} \sqcup \bigsqcup_{a \ge 2}^{a_0} Y_a,
\end{equation} 
where the varieties $Y_a$ are those of Definition \ref{def:ya}. As in the previous section, we want to show that the decomposition above is actually a stratification. To do so we need first to fix some notation. If the dimension of $V$ is $2d$ or $2d +1$, consider for $1 \le i \le d-2 $ the standard parabolic subgroup $P_i$ of $\SO(V)$ corresponding to the subset of simple reflections  $I_i = \{t_{i+1}, \dots, t_d\}$ of $W$. Observe that each subset $I_i$ is $\Phi$-stable. We also set $P_{d-1} = P_d = B$. In other words, for $i \le d-1$ the parabolic $P_i$ is the stabilizer of the standard partial isotropic flag of length $i$ 
\begin{equation*}
  \langle e_1 \rangle  \subset \langle e_1, e_2 \rangle \subset \dotsm \subset \langle e_1, e_2, \dots, e_i \rangle. 
\end{equation*}
We consider the following elements in the Weyl group.
\begin{itemize}[nolistsep]
  \item For $2 \le a \le a_0$ we set $w_a = t_1t_2\dotsm t_{d-1}t_dt_{d-2}t_{d-3}\dotsm t_{a}$, with the convention that $w_{a_0}$ is the $\Phi$-Coxeter element of Lemma \ref{lem:lus3}.
  \item If the dimension is even and the form is split, we set $w_{\infty} = t_1 \dotsm t_{d-1}$, otherwise we let $w_{\infty} = t_1\dotsm t_{d-2}$.
\end{itemize}

\begin{lem}\label{lem:RVnormal}
  The variety $R_V$ can be identified with the closure of the generalized Deligne-Lusztig variety $X_{P_1}(t_1)$ in $G/P_1$. In particular, it is normal with isolated singularities.
\end{lem}
\begin{proof}
  By definition $R_V$ parametrizes isotropic lines $l$ in $V$ such that $l + \Phi(l)$ is an isotropic subspace. Therefore, the relative position $\inv(l,\Phi(l))$ is either the identity or the class of $t_1$ in $W_{I_1}\backslash W/W_{I_1}$. Hence, we obtain a decomposition as union of an open and closed subset
  \begin{equation}\label{eq:rv}
    R_V = X_{P_1}(1) \sqcup X_{P_1}(t_1),
  \end{equation} 
  and we can conclude with Lemma \ref{lem:closure}. The second statement follows again from G\"ortz and Yu's local model diagram.
\end{proof}

\begin{lem}\label{lem:RVdim}
  For $1 \le i \le a_0 $ the subset $R_i = Y_{\infty} \sqcup \bigsqcup_{a > i}^{a_0} Y_a,$ is closed in $R_V$ and can be identified with the closure of the generalized Deligne-Lusztig variety $X_{P_{i+1}}(w_{i+1})$, which is isomorphic to $Y_{i+1}$. In particular, $Y_2$ is open and dense in $R_V$, hence $R_V$ is irreducible of dimension $2d -3$.
\end{lem}
\begin{proof}
  By the decomposition (\ref{eq:rv}) of $R_V$ and the generalized closure relations of Lemma \ref{lem:closure} we have that $X_{P_1}(t_1)$ is open and dense in $R_V$. Let $l$ be a closed point of $X_{P_1}(t_1) \subset R_V$, that is an isotropic line in $V$. We can consider the isotropic flag $l \subset l + \Phi(l)$. This defines a closed point in $X_{P_2}(t_1) $ if $l + \Phi(l)$ is $\Phi$-stable, otherwise in $ X_{P_2}(t_1t_2) $ if $l + \Phi(l) + \Phi^2(l)$ is isotropic, or in $X_{P_2}(w_2)$, if $\Phi^2(l)$ is not orthogonal to $l$. Again this map is the inverse on closed points of the base change to $X_{P_1}(t_1)$ of the projection $G/P_1 \rightarrow G/P_2$. By Lemma \ref{lem:iso} it follows that we have decomposed $R_V$ as
  \begin{equation*}
    R_1 = R_V = X_{P_1}(1) \sqcup X_{P_1}(t_1) \cong X_{P_1}(1) \sqcup X_{P_2}(t_1) \sqcup X_{P_2}(t_1t_2) \sqcup X_{P_2}(w_2).
  \end{equation*}
  By the generalized closure relations of Lemma \ref{lem:closure} $X_{P_2}(w_2)$ is then open and dense in $X_{P_1}(t_1)$ and therefore in $R_V$. Observe that the image of $X_{P_2}(w_2)$ under the quotient map $G/P_2 \rightarrow G/P_1$ is $Y_2$. This can again be tested on $k$-valued points, for any algebraically closed field $k$. It follows that we have an isomorphism $X_{P_2}(w_2) \cong Y_2$.
  We can conclude that 
  \begin{equation*}
    R_2 = R_1 \setminus Y_2 \cong X_{P_1}(1) \sqcup X_{P_2}(t_1) \sqcup X_{P_2}(t_1t_2),
  \end{equation*}
  which is closed in $R_1$, and it contains $X_{P_2}(t_1t_2)$ as an open subset. 

  Assume that we have a decomposition $R_i = X_{P_1}(1) \sqcup \bigsqcup_{j = 2}^{i} X_{P_j}(t_1\dotsm t_{j-1}) \sqcup X_{P_i}(t_1\dotsm t_{i})$ with $X_{P_i}(t_1\dotsm t_i)$ open in $R_i$. Observe that the closed points of $X_{P_i}(t_1 \dotsm t_i)$ correspond to isotropic flags of the form
  \begin{equation*}
    l \subset l + \Phi(l) \subset \dotsm \subset l + \Phi(l) + \dotsm + \Phi^{i-1}(l)
  \end{equation*}
  such that $\Phi^{i}(l)$ is orthogonal to $l$. Again we consider the base change to $R_i$ of the quotient map $G/P_{i+1} \rightarrow G/P_i$. By Lemma \ref{lem:iso} we only have to show that this gives a bijection between the sets of closed points. We can construct its inverse (as map of sets) by sending a flag of length $i$ in $X_{P_i}(t_1\dotsm t_i)(k)$, for an algebraically closed field $k$ to the isotropic flag of length $i +1$ obtained by appending the isotropic subspace $l + \Phi(l)+ \dotsm + \Phi^{i}(l)$. This defines a closed point in $X_{P_{i+1}}(t_1 \dotsm t_i)$ if this subspace of dimension $i+1$ is $\Phi$-stable, a point in $X_{P_{i+1}}(t_1\dotsm t_{i+1})$ if $\Phi^{i+2}(l)$ is orthogonal to $l$, or otherwise in $X_{P_{i+1}}(w_{i+1})$. The latter is open in $R_i$ by Lemma \ref{lem:closure}. Observe that its image under the composition $G/P_{i+1} \rightarrow G/P{i} \rightarrow G/P_1$ is the subvariety $Y_i$ defined above. Again this can be checked on closed points. Last, $R_{i+1} = R_i \smallsetminus Y_i$ is the union
  \begin{equation*}
    R_{i+1} = R_i \smallsetminus Y_i =  X_{P_1}(1) \sqcup \bigsqcup_{j = 2}^{i} X_{P_j}(t_1\dotsm t_{j-1}) \sqcup X_{P_{i+1}}(t_1\dotsm t_i) \sqcup X_{P_{i+1}}(t_1\dotsm t_{i+1}),
  \end{equation*}
  and by Lemma \ref{lem:closure} $X_{P_{i+1}}(t_1\dotsm t_{i+1})$ is open in it, and we can conclude by induction. 

  With Theorem \ref{thm:br} we can compute the dimension of $X_{P_2}(w_2) = Y_2$, from which the last statement follows. 
\end{proof}

\begin{rem}\label{rem:infty}
  Observe that by the previous lemma, for $i = a_0$ we obtain that $Y_{\infty}$ is isomorphic to the closure of $X_{B}(w_{\infty})$. Moreover, from the proof it follows that if the dimension is odd or the form is split, the variety $Y_{\infty}$ is isomorphic to the union of generalized Deligne-Lusztig varieties $\bigsqcup_{i=1}^{d-1}X_{P_i}(t_1\dotsm t_{i-1})$, otherwise to the union $\bigsqcup_{i=1}^{d-2}X_{P_i}(t_1\dotsm t_{i-1})$. The different index appearing in these unions is due to the fact that if the form is non-split there are no isotropic subspaces of dimension $d$.
\end{rem}

\begin{rem}\label{rem:hereditary} Observe that each closed stratum $R_i \subset R_V$ is irreducible as it is the closure of the generalized Deligne-Lusztig variety $X_{P_a}(w_a)$, which is irreducible by Theorem \ref{thm:br}. It follows that the stratification of $R_V$ we have just found does not have as good hereditary properties as that of $S_V$. In other words, unlike the stratification of $S_V$, see Lemma \ref{lem:SVstr}, the strata $R_i$ of $R_V$ cannot be interpreted as a variety of the form $R_{V'}$ for some smaller vector space $V'$. Moreover, given a line $l$ in $R_V$ denote by $T_{l}$ the minimal $\Phi$-stable subspace of $V$ containing $l$. Then $l$ belongs to the stratum $Y_a$ of $R_V$ if the maximal length of an isotropic chain in $T_l$ is at least $a$, which however carries little information on $T_l$ or its dimension. One can only say that the set of lines $l \in Y_a$ such that $T_l = V$ defines an open and therefore dense subscheme in $Y_a$, as also remarked in \cite{lus}*{Lem.\ 6}.
\end{rem}

We study one last family of Deligne-Lusztig varieties for the orthogonal group, which will be relevant for the analysis of the non-split case in the next sections. Roughly speaking, these new varieties are a \textit{dual version} of the varieties $Y_a$, as instead of isotropic lines, we consider isotropic subspaces of dimension $d-1$. For all $0 \le i \le d-1$ we consider
\begin{itemize}[nolistsep]
  \item $\mathtt{P}_i$ the parabolic subgroup of $G = \SO(V)$ corresponding to the subset of simple reflections $\mathtt{I}_i = \{t_1, \dots , t_{d-2-i} \}$. In other words,  $\mathtt{P}_i$ is the stabilizer of the standard isotropic flag of length $i + 1$: $\langle e_1, \dots, e_{d-1-i} \rangle \subset \dotsm \subset \langle e_1, \dots, e_{d-2}\rangle \subset \langle e_1, \dots, e_{d-1} \rangle$. In particular $\mathtt{P}_{d-2} = \mathtt{P}_{d-1} = B$.
  \item $u_i = t_{d-1}\dotsm t_{d-i}$ with the convention that $u_0 = 1$ and $u_1 = t_{d-1}$. In particular, if the form is non-split $u_{d-1} = t_{d-1}\dotsm t_1$ is a $\Phi$-Coxeter element, and it is the inverse of $w_{a_0}$, the $\Phi$-Coxeter element of Lemma \ref{lem:lus3}.
\end{itemize}

Consider the subvariety $Q_V$ of the partial flag variety $\SO(V)/\mathtt{P}_0$ parametrizing the isotropic subspaces $U$ of $V$ of dimension $d-1$ such that $U + \Phi(U)$ is isotropic
\begin{equation*}
  Q_V = \{ U \subset V \mid \dim(U) = d-1, U + \Phi(U) \subset U^{\perp} \cap \Phi(U)^\perp\}.
\end{equation*}
We can give an analogous stratification for $Q_V$ as we did for $R_V$ above or $S_V$ in the previous section.
\begin{lem}\label{lem:QVstrata}
  The variety $Q_V$ is the closure in $\SO(V)/\mathtt{P}_0$ of the generalized Deligne-Lusztig variety $X_{\mathtt{P}_0}(t_{d-1})$. There is a stratification
  \begin{equation*}
    Q_V = \bigsqcup_{i = 0}^{d-1} Z_i
  \end{equation*}
  where each stratum $Z_i$ parametrizes $(d-1)$-dimensional isotropic subspaces $U$ of $V$ such that $U + \Phi(U)$ is isotropic and $i$ is the smallest index such that $U \cap \Phi(U) \cap \dotsm \cap \Phi^i(U)$ is $\Phi$-stable. 
  
  Moreover, each subvariety $Z_i$ can be identified with the (generalized) Deligne-Lusztig variety $X_{\mathtt{P}_i}(u_i)$. In particular, $Z_{d-1} \cong  X_B(u_{d-1})$ or in the non-split case $Z_{d-1} \cong  X_B(u_{d-1}) \cup X_{B}(\Phi(u_{d-1}))$, and it is open and dense in $Q_V$, from which it follows that $Q_V$ is pure of dimension $d-1$. In particular, in the non-split case $Q_V$ has exactly two irreducible components.
\end{lem}
\begin{proof}
  The strategy of the proof is the same as in the proof of Lemma \ref{lem:SVdim} and \ref{lem:RVdim}. First, we observe that if $U$ is a $(d-1)$-dimensional isotropic subspace of $V$ such that $U + \Phi(U)$ is again isotropic, then either $U$ is $\Phi$-stable (observe that this is possible also when the form is non-split), or $U + \Phi(U)$ has dimension $d$. In other words the relative position $\mathrm{inv}(U, \Phi(U))$ is either the identity or the class of $t_{d-1}$ in $W_{\mathtt{I}_0}\backslash W / W_{\mathtt{I}_0}$. Hence, we have
  \begin{equation*}
    Q_V = X_{\mathtt{P}_0}(1) \sqcup X_{\mathtt{P}_0}(t_{d-1}),
  \end{equation*}
  and by Lemma \ref{lem:closure} $X_{\mathtt{P}_0}(t_{d-1})$ is open and dense in $Q_V$. It is clear that $Z_0 = X_{\mathtt{P}_0}(1)$. Consider an isotropic subspace $U$ that is a closed point of $X_{\mathtt{P}_0}(t_{d-1})$. Since $U +\Phi(U)$ has dimension $d$, the intersection $U \cap \Phi(U)$ is an isotropic subspace of $U$ of dimension $d-2$. Again we obtain a map on closed points $X_{\mathtt{P}_0}(t_{d-1}) \rightarrow G/\mathtt{P}_1$ by sending $U$ to the partial isotropic flag $U \cap \Phi(U) \subset U$. This flag defines a closed point of $X_{\mathtt{P}_1}(t_{d-1}) $ if $U \cap \Phi(U)$ is $\Phi$-stable, otherwise a point of $X_{\mathtt{P}_1}(t_{d-1}t_{d-2})$. Again, this is the inverse on closed points of the map given by the base change to $X_{\mathtt{P}_0}(t_{d-1})$ of the quotient map $G/\mathtt{P}_1 \rightarrow G/\mathtt{P}_0$. Therefore, by Lemma \ref{lem:iso} there is an isomorphism
  \begin{equation*}
    X_{\mathtt{P}_0}(t_{d-1}) \cong X_{\mathtt{P}_1}(t_{d-1}) \sqcup X_{\mathtt{P}_1}(t_{d-1}t_{t-2}).
  \end{equation*}
  In particular, we can check on closed points that the image of $X_{\mathtt{P}_1}(t_{d-1})$ under this isomorphism is $Z_{1}$. We know by Lemma \ref{lem:closure} that the subvariety $X_{\mathtt{P}_1}(t_{d-1}t_{d-2})$ is open and dense in $X_{\mathtt{P}_0}(t_{d-1})$ and therefore in $Q_V$. One can then use induction as in the proof of Lemma \ref{lem:RVdim}.

  Observe that by Theorem \ref{thm:br} the Deligne-Lusztig variety $X_B(t_{d-1}\dots t_{1})$ is irreducible if and only if the action of the Frobenius map $\Phi$ on the Weyl group is non-trivial, that is only if the dimension of $V$ is even and the form is non-split. Otherwise, $u_{d-1}$ is contained in the non-trivial $\Phi$-stable subgroup of $W$ generated by $ t_{1},\dots,t_{d-1}$. It follows that if the form is non-split $Z_{d-1}$, and consequently $Q_V$, has two irreducible components. 
\end{proof}

\begin{rem}\label{rem:flag}
  A key observation which we will need in the proof of Theorem \ref{thm:intro} is the existence of a morphism from $Z_{d-1}$ into a flag variety. Assume the form is non-split and consider the isomorphism given in the previous proposition $Z_{d-1} \cong X_{B}(u_{d-1}) \cup X_B(\Phi(u_{d-1}))$. Then there is an immersion $X_{B}(u_{d-1}) \cup X_{B}(\Phi(u_{d-1}))\rightarrow G/B$, where $G$ is the orthogonal group. As we have recalled above, $G = \SO(V)$ acts on $V$ and $B$ is the stabilizer of the standard isotropic flag. It follows that $G/B$ is locally the orbit of the standard isotropic flag under the action of $G$. As in the construction of the Grassmannian, compare \cite{gw}*{Sec.\ 8.4}, one sees that $G/B$ is actually isomorphic to the orbit space of the isotropic flag. Let $\mathcal{F}l(V)$ be then the projective variety parametrizing flags of subspaces of $V$ of the form $U_1 \subset U_2 \subset \dotsm U_{d-1}$ where the dimension of each subspace $U_{i}$ is $i$. By sending an isotropic flag to itself as a point of $\mathcal{F}l(V)$ we obtain an immersion $G/B \rightarrow \mathcal{F}l$. By precomposing with the immersion $X_B(u_{d-1}) \cup X_B(\phi(u_{d-1})) \rightarrow G/B$ we obtain the desired morphism.
\end{rem}

\section{Pointwise decomposition of $\bar{\N}_{2,6}^0$}\label{sec:closedpoints}

In this section we study the $k$-valued points of $\bar{\N}_{2,6}^0$ for any algebraically closed field $k$ containing $\F$. This serves as preparation for the description of the irreducible components of the reduced scheme underlying $\bar{\N}_{2,6}^0$. From now on, we fix $n = 6$ and $s=2$ and drop the subscript from the notation $\bar{\N}_{2,6}^0$.

We extend the Hermitian form $h$ on $C$ to a sesquilinear form on $C \otimes_{\Q_p} W(k)_{\Q}$ by setting $h(v\otimes a, w \otimes b) = a \sigma(b) h( v,w )$. Similarly, using the relation (\ref{eq:forms}) between the Hermitian and alternating form on $C$, we can extend the alternating form on $C$ to an alternating form on $C \otimes_{\Q_p} W(k)_{\Q}$, which we denote again with angled brackets. By the same arguments as in (\ref{eq:V(k)}) we have a bijection between the $k$-valued points of $\bar{\N}^0$ and the set of $\mathcal{O}_E \otimes_{\Z_p} W(k)$-lattices 
\begin{equation}\label{eq:V(k)2}
  \V(k) = \{M \subset C\otimes_{\Q_p} W(k)_{\Q} \mid M^\vee = M, \pi \tau(M) + \pi M \subset M \cap \tau(M), M \subset^{\le 2} (M + \tau(M))\}.
\end{equation}
Observe that we have reformulated  the condition $\pi \tau(M) \subset M \subset \pi^{-1} \tau(M)$ of (\ref{eq:V(k)}) in an equivalent way, which will be useful in the sequel.

\subsection{The set $\V_{\Lambda}(k)$ for a vertex lattice $\Lambda$}
Let $\Lambda$ be a vertex lattice of type $2m$ in $C$, recall that $m \le 3$ if the form is split, otherwise $m \le 2$. The strategy is the same as \cite{rtw}*{Sec.\ 6}, with a few modifications due to the different signature, \textit{i.e.}\ to the different index in (\ref{eq:V(k)2}). For an algebraically closed field $k$ containing $\F$ we denote by $\Lambda_k$ the $\mathcal{O}_E \otimes_{\Z_p} W(k)$-lattice $\Lambda \otimes_{\Z_p} W(k)$. Since $\Lambda$ is a vertex lattice, if a self-dual lattice $M$ is contained in $ \Lambda_k$, then $\pi \Lambda_k \subset \Lambda_k^\vee \subset M \subset \Lambda_k$. Moreover, by $\tau$-stability of $\Lambda_k$, and consequently of $\pi\Lambda_k$, we have that 
\begin{equation*}
  \pi M + \pi \tau(M) \subset \pi \Lambda_k \subset M \cap \tau(M).
\end{equation*} 
Therefore, if $M \subset \Lambda_k$, the inclusion in the middle of definition (\ref{eq:V(k)2}) of $\V(k)$ is always satisfied and we can omit it, compare also \cite{rtw}*{Cor.\ 6.3}. It follows that for a vertex lattice $\Lambda$
\begin{equation*}
  \V_{\Lambda}(k) = \{ M \in \V(k) \mid M \subset \Lambda_k\} = \{M \subset \Lambda_k \mid M = M^\vee, M \subset^{\le 2} (M + \tau(M))\}.
\end{equation*}

As in \textit{loc.cit.}\ we consider the $2m$-dimensional $\F_p$-vector space $V = \Lambda/\Lambda^\vee = \Lambda/\Lambda^{\sharp}$ and the corresponding $k$-vector space $V_k = V \otimes_{\F_p} k = \Lambda_k/\Lambda_k^\vee$. One can define an alternating form on $V$ as follows. For $x, y \in V$ with lifts $x', y'$ in $\Lambda$, we let $\langle x, y \rangle_V$ be the image of $p\langle x', y' \rangle$ in $\F_p$. This form can then be extended $k$-linearly to $V_k$. Since $\Lambda^\vee$ is the dual of $\Lambda$ with respect to the alternating form, the form just defined on $V_k$ is a well-defined, non-degenerate and alternating bilinear form, see \cite{rtw}*{Lem.\ 6.4} for a detailed proof. Moreover, as remarked in \textit{loc.cit.}, by the isomorphism $C \otimes_{\Q_p} W(\F)_\Q \cong N$ given in Section \ref{sec:lattices}, the map $\tau$ on $\Lambda$ induces the identity on $V$ and the Frobenius on $V_k$. The following result is proved in the same way as \cite{rtw}*{Lem.\ 6.5}. For completeness, we recall here the main ideas of the proof.

\begin{lem}\label{lem:spV}
  The map $M \mapsto M/\Lambda_k^\vee$ induces a bijection between $\V_{\Lambda}(k)$ and the set  of $k$-valued points of the generalized Deligne-Lusztig variety $S_V$ defined in Section \ref{sec:dlvsp}.
\end{lem}
\begin{proof}
  The fact that $M$ is self-dual is equivalent to its image $U$ under the quotient map being a Lagrangian subspace of the sympletic space $V_k$.  Similarly, $M$ having index at most $2$ in $M + \tau(M)$ is equivalent to its image $U$ satisfying $U \cap \Phi(U) \subset^{\le 2} U$, from which it follows that $U$ is a point of $S_V$. Conversely, consider a Lagrangian subspace $U$ in $S_V$. Its preimage under the quotient map $\Lambda_k \rightarrow V_k$ is a self-dual lattice $M$ contained in $\Lambda_k$, such that $M \subset^{\le 2} M + \tau(M)$.
\end{proof}

\subsection{The set $\V_{\Lambda}(k)$ for a $2$-modular lattice $\Lambda$}
Fix  a $2$-modular lattice $\Lambda$ in $C$, that is an $\mathcal{O}_E$-lattice satisfying $\pi^2 \Lambda = \Lambda^\vee \subset \Lambda$. Recall that in this case $\pi\Lambda$ is self-dual. As in the previous case, for an algebraically closed field $k$ containing $\F$ we consider the lattice $\Lambda_k = \Lambda \otimes_{\Z_p} W(k)$ in $C \otimes_{\Q_p} W(k)_{\Q}$ and the set of $\mathcal{O}_E \otimes_{\mathbb{Z}_p}W(k)$-lattices 
\begin{equation*}
  \V_{\Lambda}(k) = \{ M \subset \Lambda_k \mid M = M^\vee, \pi M + \pi \tau(M) \subset M \cap \tau(M), M \subset^{\le 2} (M + \tau(M))\}.
\end{equation*}
Observe that for $M \in \V_{\Lambda}(k)$, if $\pi \Lambda_k \subset M$  the two lattices coincide by self-duality. Therefore, in general $\pi \Lambda_k \not \subset M$. It follows that, unlike in the previous case, the inclusion $\pi M + \pi\tau(M) \subset M \cap \tau(M)$ in the definition of $\V_{\Lambda}(k)$ above does not follow from $M \subset \Lambda_k$, and is therefore not redundant. As a first consequence, we are going to see that in the analogue of Lemma \ref{lem:spV} for $2$-modular lattices we loose surjectiveness.

As above, we consider the $\F_p$-vector space $V = \Lambda/\Lambda^\vee$ and its base change $V_k$. Observe that since $\Lambda$ is $2$-modular $V$ has dimension $2n = 12$. Again, the alternating form on $\Lambda$ induces an alternating form on $V$ that can be extended $k$-linearly to $V_k$. 
\begin{lem}\label{lem:spV2}
  For a $2$-modular lattice $\Lambda$, the map $\V_{\Lambda}(k) \rightarrow S_V(k)$ sending $M$ to $ M/\Lambda_k^\vee$ is injective but not surjective.
\end{lem}
\begin{proof}
  The first claim is proved as in Lemma \ref{lem:spV}. If $M \in \V_{\Lambda}(k)$, we have $\Lambda_k^\vee \subset M^{\vee} = M \subset \Lambda_k$, therefore, the map is clearly injective. By definition of the form on $V_k$, if $M$ is a self-dual lattice, then its image is a Lagrangian subspace of $V_k$. Similarly, since, as we have remarked, the map $\tau$ induces the Frobenius $\Phi$ on $V_k$, the index of $M$ in $M + \tau(M)$ is equal to the codimension of its image $U$ in $U + \Phi(U)$. Therefore, $M$ is sent to a point of $S_V(k)$.

  Observe that the action of $\pi$ on $\Lambda$ induces a linear map  $\bar{\pi}: V_k \rightarrow V_k$ of rank $6$. Indeed, since $\Lambda^\vee = \pi^2 \Lambda$, the image of the map $\bar{\pi}$ is the six-dimensional subspace $\overline{\pi\Lambda}_k = \pi \Lambda_k/ \pi^2 \Lambda_k \subset V_k$. Moreover, $\overline{\pi \Lambda}_k$ is also the kernel of $\bar{\pi}$. As we have already observed, $\pi \Lambda_k$ is a self-dual, $\tau$-stable lattice, hence $\overline{\pi\Lambda}_k$ is a $\Phi$-stable Lagrangian subspace of $V_k$. Consider now $\overline{\mathcal{L}}$, a $\Phi$-stable Lagrangian complement of $\overline{\pi \Lambda}_k$ in $V_k$. For example, one can take the base change to $k$ of any Lagrangian complement of the image of $\pi \Lambda$ in $V$. Clearly, $\overline{\mathcal{L}}$ belongs to $S_V(k)$. Since $\overline{\mathcal{L}} \cap \overline{\pi\Lambda}_k = 0$, when we lift it to a $W(k)$-lattice $\mathcal{L} \subset \Lambda_k$, we have that $\mathcal{L} \cap \pi \Lambda_k = \pi^2 \Lambda_k$. Moreover, since $\overline{\pi\Lambda}_k$ is both the kernel and image of $\bar{\pi}$, we have that $\bar{\pi}(\overline{\mathcal{L}}) = \overline{\pi\Lambda}_k$. It follows that $\pi \mathcal{L} = \pi \Lambda_{k}$, which is not contained in $\mathcal{L}$, so $\mathcal{L}$ is not an $\mathcal{O}_E \otimes_{\Z_p}W(k)$--lattice, hence it does not belong to $\V_{\Lambda}(k)$.
\end{proof}

Our goal is now to find a description in terms of Deligne-Lusztig varieties of the image of the map $\V_{\Lambda}(k) \rightarrow S_V(k)$ above. Recall that the vector space $C$ carries also a symmetric form, which is related to the alternating form by the formula $(x, y) = \langle \pi x, y \rangle$. As we have seen in Section \ref{sec:lattices}, the duals of an $\mathcal{O}_E\otimes W(k)$-lattice $M$ with respect to the two forms satisfy $M^{\perp} = \pi^{-1}M^{\vee}$. In particular, if $M$ is self-dual with respect to the alternating form, we have that $M^{\perp} = \pi^{-1}M$. Hence, any lattice $M \in \V_{\Lambda}(k)$ is contained in its dual with respect to the symmetric form. Similarly, observe that the condition $\pi M + \pi \tau(M) \subset M \cap \tau(M)$ is equivalent to
\begin{equation*}
  M + \tau(M) \subset \pi^{-1}(M \cap \tau(M)) = \pi^{-1}(M + \tau(M))^\vee = (M + \tau(M))^{\perp},
\end{equation*}
and we can reformulate the definition of $\V_{\Lambda}(k)$ as
\begin{equation}\label{eq:defVk}
  \V_{\Lambda}(k) = \{ M \subset \Lambda_k \mid M = M^\vee,  M + \tau(M) \subset (M + \tau(M))^{\perp}, M \subset^{\le 2} (M + \tau(M))\}.
\end{equation}
This reformulation turns out to be particularly useful for describing the image of the map of Lemma \ref{lem:spV2}. Consider the six-dimensional $\F_p$-vector space $W = \Lambda/\pi\Lambda$ and its base change $W_k =W \otimes_{\F_p} k =  \Lambda_k/\pi\Lambda_k$. We endow $W$ with a symmetric bilinear form by setting $(x, y)$ as the image in $\F_p$ of $p(x', y')$ for two lifts $x', y'$ in $\Lambda$. We also extend this form $k$-linearly to $W_k$. 
\begin{lem}
  The bilinear form on $W_k$ defined above is well-defined, symmetric and non-degenerate.
\end{lem}
\begin{proof}
  First, observe that for two elements $x, y \in \Lambda$ the value of the bilinear form $p(x, y) = p \langle \pi x, y \rangle = \langle  \pi x, \eta^{-1}\pi ^2 y\rangle$ is in $\Z_p$, since $\eta^{-1}\pi ^2 y \in \pi^2 \Lambda = \Lambda^{\vee}$, hence it makes sense to consider its image in $\F_p$.
  Since $\Lambda$ is a $2$-modular lattice we have $ \Lambda^{\perp} = \pi^{-1}\Lambda^\vee= \pi^{-1}(\pi^2 \Lambda) = \pi \Lambda $. Hence, if $x' \in \pi \Lambda$, we have $(x', y') \in \Z_p$ for every $y' \in \Lambda$, and therefore the image of $p(x',y')$ in $\F_p$ is $0$. This proves that the form is well-defined on the quotient $W = \Lambda/\pi\Lambda$ and therefore on $W_k$. It is also clear that it is symmetric. Assume there is an element $x' \in \Lambda$ such that for all $y' \in \Lambda$ the image of $p(x', y')$ is zero in $\F_p$. This means that $(x',y') \in \Z_p$ for all $y' \in \Lambda$, and therefore, $x' \in \Lambda^{\perp} = \pi \Lambda$. This proves that the form on $W$, and consequently on $W_k$ is non-degenerate.
\end{proof}

As we have already observed, the image of $\pi\Lambda_k$ in $V_k$ is a $\Phi$-stable Lagrangian. Therefore, the quotient map $V_k \rightarrow V_k/\overline{\pi\Lambda} = W_k$ commutes with the Frobenii on $V_k$ and $W_k$. It follows that $\tau$ induces again the identity on $W$ and the Frobenius $\Phi$ on $W_k$. Since $W_k$ is a six-dimensional $k$-vector space endowed with a symmetric form, it is a natural question to ask whether it is split, \textit{i.e.}\ whether there is a $\Phi$-stable maximal isotropic subspace. 
\begin{lem}
  The symmetric form on $W_k$ is split if and only if the Hermitian form $h$ on $C$ is split.
\end{lem}
\begin{proof}
  In \cite{rtw}*{Lemma 3.3} it is proved that the Hermitian form on the $n$-dimensional space $C$ is split if and only if $C$ contains a vertex lattice of type $n$, that is, if and only if there is an $\mathcal{O}_E$-lattice $\mathcal{L} \subset C$ such that $\mathcal{L}^\vee = \pi\mathcal{L}$ or equivalently, such that $ \mathcal{L}^\perp = \mathcal{L} $. Since $\pi \Lambda$ is self-dual, it is itself a vertex lattice of type $0$. By the correspondence of  \cite{rtw}*{Prop.\ 3.4} between vertex lattices and the Bruhat-Tits simplicial complex of $\rm{SU}(C)(\Q_p)$, if the form is split there exists a vertex lattice $\mathcal{L}$ of maximal type containing $\pi \Lambda$. Therefore, the Hermitian form $h$ on $C$ is split if and only if there is a vertex lattice of type $n = 6$ containing $\pi \Lambda$. 

  If such a vertex lattice $\mathcal{L}$ exists, then from the fact that $\mathcal{L} = \mathcal{L}^\perp$ and the definition of the orthogonal form on $W_k$ it follows that the image of $\mathcal{L}_k$ in $W_k$ is a $\tau$-stable, isotropic subspace. Moreover, from the inclusions $\pi\mathcal{L} = \mathcal{L}^\vee \subset \pi\Lambda \subset \mathcal{L}$ it follows that $\pi \Lambda$ has index $n/2 = 3$ in $\mathcal{L}$. Therefore, the $\Phi$-stable isotropic subspace given by the image of $\mathcal{L}_k$ in $W_k$ has maximal dimension $3$, and hence the form on $W_k$ is split. 
  
  On the other hand, if there is a $\Phi$-stable maximal isotropic subspace $L$ in $W_k$, we can lift it to a $\tau$-stable $\mathcal{O}_E \otimes W(k)$-lattice $\pi \Lambda_k \subset^3 \mathcal{L} \subset \Lambda_k$. Moreover, since $L = L^\perp$, by the same argument as in the proof of Lemma \ref{lem:5.5} below we have that $\mathcal{L}^{\perp} = \mathcal{L}$. By Lemma \ref{rem:basis}, since $\mathcal{L} = \tau(\mathcal{L})$, it has a $\tau$-stable basis. Hence, we can consider the set of its $\tau$-fixed points $\mathcal{L}^\tau$ and obtain a vertex lattice of type $n = 6$ in $C$, from which it follows that the Hermitian form on $C$ is split.
\end{proof}

Our goal now is to describe the points in $\V_{\Lambda}(k)$ in terms of points of a Deligne-Lusztig variety for the orthogonal group $\SO(W_k)$. The first step in this direction is the following observation.
\begin{lem}\label{lem:5.5}
  The map $M \mapsto (M + \pi\Lambda_k) /\pi \Lambda_k$ induces a map from $\V_{\Lambda}(k)$ to the set $$\{ U \subset W_k \mid U + \Phi(U) \subset (U +\Phi(U))^{\perp}, U \subset^{\le 2} U + \Phi (U)\}.$$
\end{lem}
\begin{proof}
  First observe that the quotient map $q: \Lambda_k \rightarrow W_k$ is compatible with taking the dual (respectively the orthogonal) with respect to the symmetric forms on both sides. Indeed, if $M \subset \Lambda_k$ is a lattice in $\V_{\Lambda}(k)$ with image $U \subset W_k$, then by definition of the form on $W_k$, the orthogonal space of $U$ satisfies
  \begin{equation*}
    U^{\perp} = \{ x \in \Lambda_k \mid p(x, y) \in pW(k), \text{for all } y \in q^{-1}(U)\} /\pi \Lambda_k.
  \end{equation*}
  This means that $U^{\perp}$ is the image in $W_k$ of the lattice $(M + \pi \Lambda_k)^{\perp} = M^\perp \cap \Lambda_k$ with respect to the symmetric form on $C \otimes W(k)_{\Q}$. It follows
  \begin{align*}
    (U + \Phi(U))^\perp &= U^\perp \cap \Phi(U)^\perp = q(M^\perp \cap \Lambda_k) \cap q(\tau(M)^\perp \cap \Lambda_k)  \\
    &= q(\pi^{-1} M \cap \Lambda_k) \cap q(\pi^{-1} \tau(M) \cap \Lambda_k) \\
    & \supset q(\pi^{-1}M \cap \pi^{-1}\tau(M) \cap \Lambda_k) \\
    &\supset q(M + \tau(M) + \pi\Lambda_k) = U + \Phi(U),
  \end{align*}
  where the second inclusion follows from (\ref{eq:V(k)2}).
\end{proof}

Observe that the set appearing in Lemma \ref{lem:5.5} above as the image of the quotient map resembles now the description of the $k$-valued points of some Deligne-Lusztig variety for the orthogonal group. What is still missing is the information on the dimension of the image $U$ of $M$ in $W_k$. For example, if we restrict to $\dim(U) = 1$ we obtain the points of the generalized Deligne-Lusztig variety $R_W$ introduced in \ref{sec:dlvso}, while for $\dim(U) = 2$ we recover the points of the variety $Q_W$ of \ref{sec:dlvso}. We let $\V_\Lambda^{(i)}(k)$ denote the subset of lattices $M \in \V_\Lambda(k)$ such that $\pi\Lambda_k \subset^i M + \pi\Lambda_k$.

\begin{lem}\label{lem:spW}
   The restriction of the map $M \mapsto (M + \pi \Lambda_k) /\pi\Lambda_k$ induces a surjective map 
  \begin{equation*}
    \V^{(1)}_{\Lambda}(k) = \{ M \in \V_{\Lambda}(k) \mid \pi \Lambda_k \subset^1 M + \pi \Lambda_k \} \longrightarrow R_W(k)
  \end{equation*}
  onto the $k$-valued points of the generalized Deligne-Lusztig variety $R_W$ of Section \ref{sec:dlvso} with fibers equal to $\mathbb{A}^1(k)$. 
\end{lem}
\begin{proof}
  By Lemma \ref{lem:5.5} above, if $M \in \V_{\Lambda}^{(1)}(k)$ is mapped to a line $l$ in $W_k$, then $l$ and $l + \Phi(l)$ are both isotropic and therefore $l$ is a point in the variety $R_W(k)$ defined in the previous section. Observe that the map $M \mapsto M+\pi\Lambda_k/\pi\Lambda_k$ factors through the map $\V_{\Lambda}(k) \rightarrow S_V(k), M \mapsto M/\pi^2\Lambda_k$. As we have seen in Lemma \ref{lem:spV2} this latter map is injective but not surjective. In particular, its image is the proper subset $S_{V\pi}$ of $S_V(k)$ consisting of Lagrangian subspaces $U \subset V_k$ such that $\overline{\pi}(U) + \overline{\pi}(\Phi(U)) \subset U \cap \Phi(U)$, where $\overline{\pi}$ denotes again the rank-$6$ linear map on $V_k$ induced by the action of $\pi$ in $\Lambda$. It is then enough to prove the statement for the map $S_{V \pi} \rightarrow R_W$ induced by the quotient map $q: V_k \rightarrow W_k$.

  Fix a Lagrangian complement $\mathcal{L}$ of $\overline{\pi\Lambda}$ in $V$, that is a Lagrangian subspace of $V$ such that $V = \mathcal{L} \oplus \overline{\pi\Lambda}$. Then we can identify $W$ with $\mathcal{L}$ and a line $l\in R_W(k)$ with a line $l$ in $\mathcal{L}_k$. Via the isomorphism $W_k \cong \mathcal{L}_k$ we can define a symmetric form on $\mathcal{L}_k$. By definition it satisfies $(v_1, v_2) = \langle v_1, \overline{\pi}(v_2) \rangle = - \langle \overline{\pi}(v_1),  v_2 \rangle$ for all $v_1, v_2 \in \mathcal{L}_k$. Recall that the restriction of $\overline{\pi}$ induces a linear isomorphism between $\mathcal{L}_k$ and $\overline{\pi\Lambda}_k$. Consider a line $l \in R_W$ and its preimage $N = q^{-1}(l) = l \oplus \overline{\pi\Lambda}_k \subset V_k$ with orthogonal $N^\vee$ with respect to the alternating form. Observe that since $\overline{\pi\Lambda}_k$ is Lagrangian we have $N^\vee = l^\vee \cap \overline{\pi\Lambda}_k \subset^1 \overline{\pi\Lambda}_k \subset^1 N = l \oplus \overline{\pi\Lambda}_k$. Let $L \neq \overline{\pi\Lambda}_k$ be a six-dimensional subspace of $V_k$ such that
  \begin{equation}\label{eq:U}
    N^\vee \subset^1 L \subset^1 N.
  \end{equation} Clearly $L + \overline{\pi\Lambda}_k = N$ is mapped by $q$ to $l$. We show that $L$ is in $S_{V \pi}$.  By definition of $R_W$ we have that $l + \Phi(l)$ is isotropic with respect to the symmetric form. In other words we have that $\langle l, \overline{\pi}(l) \rangle  = \langle l, \overline{\pi}(\Phi(l)) \rangle = \langle \Phi(l), \overline{\pi}(\Phi(l)) \rangle = 0$. This means that $\overline{\pi}(l) + \overline{\pi}(\Phi(l)) \subset l^\vee \cap \overline{\pi\Lambda}_k = N^\vee \subset L$. Similarly, $\overline{\pi}(l) + \overline{\pi}(\Phi(l)) \subset \Phi(l)^\vee \cap \overline{\pi\Lambda}_k = \Phi(N)^\vee = \Phi(N^\vee) \subset \Phi(L)$. Here the Frobenius commutes with taking the orthogonal, that is we have the equality $\Phi(N)^\vee = \Phi(N^\vee)$, because $k$ is algebraically closed, hence $\Phi$ preserves dimensions, and clearly $\Phi(N^\vee) \subset \Phi(N)^\vee$. We can conclude  that $\overline{\pi}(L) +\overline{\pi}(\Phi(L)) = \overline{\pi}(l) + \overline{\pi}(\Phi(l)) \subset N^\vee \cap \Phi(N^\vee) \subset L \cap \Phi(L)$. 
  
  It remains then to prove that $L \in S_V$. Complete a basis of $N^\vee \subset^1 L$ to a basis of $L$, in other words find an element $x \in L$ such that $L = \langle x \rangle \oplus N^\vee$. We already know that $N^\vee$ is contained in its orthogonal $N$ with respect to the alternating form. Since $x \in L \subset N$ then $\langle x, N^\vee \rangle = 0$, hence $L$ is isotropic and has dimension $6$, from which it follows that it is Lagrangian. 
  Consider $L + \Phi(L) = \langle x, \Phi(x) \rangle + N^\vee + \Phi(N)^\vee$. Since $N^\vee \subset^1 \overline{\pi\Lambda}_k$ and $\overline{\pi\Lambda}_k$ is $\Phi$-stable, we have that $L + \Phi(L) \subset \langle x, \Phi(x) \rangle + \overline{\pi\Lambda}_k$ which has dimension at most eight, from which we can conclude that $L \in S_{V \pi}$. 
  
  We have proved that every subspace $L \neq \overline{\pi\Lambda}_k$ such that $N^\vee \subset^1 L \subset^1 N = l \oplus \overline{\pi\Lambda}_k$ is a preimage of $l$ in $S_{V \pi}$. It follows that the map $\V_{\Lambda}^{(1)}(k) \rightarrow R_W(k)$ is surjective, and its fibers are in bijection with the $k$-points of $\mathbb{P}(N/N^\vee) \small{\setminus} \{\overline{\pi\Lambda}_k \}$ which we can identify with $\mathbb{A}^1(k)$. 
\end{proof}

\begin{lem}\label{lem:splitV1}
  If the Hermitian form on $C$ is split, the subset of lattices $M \in \V_{\Lambda}(k)$ whose associated lattice $\Lambda(M)$ is \emph{not} a vertex lattice is contained in $\V_{\Lambda}^{(1)}(k)$. In particular, it is the preimage of $R_W \small \setminus Y_{\infty}$ under the map of Lemma \ref{lem:spW}.
\end{lem}
\begin{proof}
  Fix $M \in \V_{\Lambda}(k)$ and let $U$ be the image in $W_k$ of $M + \pi \Lambda_k$. We argue by cases on the possible dimension of $U$. By the previous lemma we know that $U$ and $U + \Phi(U)$ are isotropic, hence have dimension at most $3$. Therefore, if $U$ has dimension $3$, it is $\Phi$-stable. This is possible as the form is split. It follows that $M + \pi \Lambda_k$ is $\tau$-stable and contains $M$. Moreover, it satisfies $(M + \pi \Lambda_k)^\vee = M \cap \pi \Lambda_k \supset \pi M + \pi^2 \Lambda_k$, which means that its intersection with $C$ is a vertex lattice of type $6$. By minimality, it contains $\Lambda(M)$, which is then a vertex lattice itself, by Remark \ref{rem:vl2}.

  Suppose that $U$ has dimension $2$. If $U$ is $\Phi$-stable, then arguing as in the previous case, $\Lambda(M)$ is a vertex lattice. Suppose that $U$ is not $\Phi$-stable. Then since $U + \Phi(U)$ is isotropic and properly contains $U$, it has dimension $3$. Consider the inclusion $U \cap \Phi(U) \subset^1 U$. If the intersection $U \cap \Phi(U)$ is $\Phi$-stable, then we can consider the $4$-dimensional space given by the quotient $W_k'=(U \cap \Phi(U))^{\perp}/U\cap \Phi(U)$. The symmetric form on $W_k$ induces a well-defined, non-degenerate symmetric form on this space, and $\Phi$ induces again the Frobenius. In particular, the symmetric form on the quotient space $W_k'$ is still split (consider for example the image of a maximal $\Phi$-stable isotropic subspace of $W_k$ containing $U \cap \Phi(U)$). The image of $U$ in the quotient $W_k'$ is then an isotropic line $l$ such that $l + \Phi(l)$ is isotropic. Since for a split, $4$-dimensional symmetric space the parameter $a_0$ defined in Lemma \ref{lem:lus3} is $1$, it follows that $l + \Phi(l)$ is $\Phi$-stable. Therefore, $U + \Phi(U)$ is an isotropic $\Phi$-stable subspace of $W_k$. Its preimage $\mathcal{L}$ in $\Lambda_k$ is then a $\tau$-stable lattice containing $M$ and such that $\mathcal{L} \subset \mathcal{L}^\perp = \pi^{-1} \mathcal{L}^{\vee}$, hence $\mathcal{L}^{\tau}$ and consequently $\Lambda(M)$ are vertex lattices.

  Suppose now that the image $U$ of $M$ in $W_k$ has dimension $2$ and $U \cap \Phi(U)$ is not $\Phi$-stable. Since the latter is one-dimensional, there is a vector $v \in U$ such that $U \cap \Phi(U) = \langle \Phi(v) \rangle$. Since $U \cap \Phi(U)$ is not $\Phi$-stable, we have that $\Phi(v)$ and $\Phi^2(v)$ are linearly independent. The same holds then for $v$ and $\Phi(v)$, so we have that $U = \langle v, \Phi(v) \rangle$ and $U + \Phi(U) = \langle v, \Phi(v), \Phi^2(v) \rangle$. Since the latter is isotropic, we have that $v$ is orthogonal to $\Phi(v)$ as well as to $\Phi^2(v)$. Again by Lemma \ref{lem:lus3} we know that $a_0 = 2$ for a split six-dimensional symmetric space, so it follows that $U + \Phi(U)$ is $\Phi$-stable and isotropic. We can deduce as above that $\Lambda(M)$ is a vertex lattice. This proves that if $\Lambda(M)$ is not a vertex lattice, the image of $M$ in $W_k$ has dimension one, hence $M \in \V_{\Lambda}^{(1)}(k)$.

  Last, since the largest isotropic subspace of $W_k$ has dimension $3$, observe that if $M \in \V^{(1)}_{\Lambda}(k)$ is sent to a line $l \in Y_{\infty}$, it means that $l + \Phi(l) + \Phi^2(l)$ is $\Phi$-stable. Then we can argue as in the beginning of this proof to see that $\langle l + \Phi(l) + \Phi^2(l) \rangle + \overline{\pi\Lambda}_k$ lifts to a vertex lattice containing $M$. Conversely, if $l \in R_W(k) \setminus Y_{\infty}$ then $l + \Phi(l) + \Phi^2(l)$ is not isotropic. On the other hand, the image of a vertex lattice in $W_k$ is isotropic, hence it cannot contain $l + \Phi(l) + \Phi^2(l)$ and hence it cannot contain $M$.
\end{proof}

In the non-split case, as we are going to see, there are lattices $M \in \V_{\Lambda}^{(2)}(k)$ whose associated lattice $\Lambda(M)$ is not a vertex lattice. This is essentially a consequence of the different possible values of the parameter $a_0$ introduced in Lemma \ref{lem:lus3}.

\begin{lem}\label{lem:spW2}
  If the Hermitian form on $C$ is non-split, for every $2$-modular lattice $\Lambda$
  \begin{equation*}
    \V_{\Lambda}(k) = \{\pi\Lambda_k\} \sqcup \V_{\Lambda}^{(1)}(k) \sqcup \V_{\Lambda}^{(2)}(k),
  \end{equation*} 
  and the restriction of the map $M \mapsto (M + \pi \Lambda_k) /\pi\Lambda_k$ induces a surjective map 
  \begin{equation*}
    \V^{(2)}_{\Lambda}(k) \longrightarrow Q_W,
  \end{equation*}
 to the generalized Deligne-Lusztig variety $Q_W$ of Section \ref{sec:dlvso}.
\end{lem}

\begin{proof}
  Fix $M \in \V_{\Lambda}(k)$ and let $U$ denote its image in $W_k$.  Again we argue by cases on the dimension of $U$. By Lemma \ref{lem:5.5} we know that $U$ and $U + \Phi(U)$ are isotropic subspaces of $W_k$. This already excludes the case $\dim(U) = 3$ as this would imply $U = \Phi(U)$, a contradiction to the fact that the symmetric form on $W_k$ is non-split. 

  Suppose now that $U$ has dimension $2$. Recall that the $k$-valued points of the variety $Q_W$ are isotropic subspaces $U$ of $W_k$ of dimension $2$ and such that $ U + \Phi(U)$ is isotropic, too. Then by Lemma \ref{lem:5.5} it is clear that $\V_{\Lambda}^{(2)}(k)$ is mapped to a point in $Q_W(k)$. We show that this map is surjective. Consider the subset $S_{V \pi}$ of $S_V(k)$ as in the proof of Lemma \ref{lem:spW}.  Fix again a Lagrangian complement $\mathcal{L}$ of $\overline{\pi\Lambda}$ in $V$, which we identify with $W$. Let $U \subset \mathcal{L}_k$ be a $2$-dimensional subspace in $Q_W(k)$, we show how to construct a preimage of $U$ in $S_{V \pi}(k)$, which means a preimage in $\V_{\Lambda}^{(2)}(k)$. Consider the subspace $N = U \oplus \overline{\pi\Lambda}_k$ and its orthogonal $N^\vee$ with respect to the alternating form.  Let $L$ be the six-dimensional subspace $ L = U \oplus N^\vee$, then clearly $L$ is sent to $U$ by the quotient map $V_k \rightarrow W_k$. We prove that $L \in S_{V \pi}$. Since $U $ is contained in the Lagrangian subspace $\mathcal{L}$, it is an isotropic subspace, that is $U \subset U^\vee$. Moreover, $U \subset N$, from which it follows that $\langle U, N^\vee \rangle = 0$, and we can conclude that $L = U \oplus N^\vee$ is Lagrangian. We need to prove that $L + \Phi(L)$ has dimension at most eight. Observe that since $U + \Phi(U)$ has dimension at most $3$, we have $\dim(N + \Phi(N)) = \dim((U + \Phi(U)) \oplus \overline{\pi\Lambda}_k) \le 9$ from which it follows, by $\dim(N) = 8$, that $\dim(N \cap \Phi(N)) \ge 7$. By taking duals and observing, as in the proof of Lemma \ref{lem:spW}, that $\Phi(N^\vee) = \Phi(N)^\vee$ we obtain $\dim(N^\vee + \Phi(N^\vee)) = 12 - \dim(N \cap \Phi(N))\le 5$. Hence, we conclude that $L + \Phi(L) = (U + \Phi(U)) \oplus (N^\vee + \Phi(N^\vee))$ has dimension at most $3 + 5 = 8$, hence $L$ belongs to $S_{V}(k)$. The fact that $L$ belongs to $S_{V
 \pi}$, that is $\overline{\pi}(L) + \overline{\pi}(\Phi(L)) \subset L \cap \Phi(L)$ follows from the fact that $U + \Phi(U)$ is isotropic with respect to the symmetric form on $\mathcal{L}_k$ and by the same argument as in the proof of Lemma \ref{lem:spW}. This proves that the map $\V_{\Lambda}^{(2)}(k) \rightarrow Q_W(k)$ is surjective.
\end{proof}

\begin{lem}\label{lem:spW2b}
  Recall the stratification $Q_W = \bigsqcup_{i = 0}^2 Z_i$ of Lemma \ref{lem:QVstrata}. The map of Lemma \ref{lem:spW2} above sends a lattice $M \in \V_{\Lambda}^{(2)}(k)$ to
  \begin{itemize}
    \item[(i)] a point in $Z_0(k)$ if and only if $\Lambda(M)$ is a vertex lattice, moreover, in this case there is another $2$-modular lattice $\Lambda'$ such that $M \in \V_{\Lambda'}^{(1)}(k)$,
    \item[(ii)] a point in $Z_1(k)$ if and only if $\Lambda(M)$ is not a vertex lattice and there is another $2$-modular lattice $\Lambda'$ such that $M \in \V_{\Lambda'}^{(1)}(k)$,
    \item[(iii)] a point in $Z_2(k)$ if and only if $\Lambda(M)$ is not a vertex lattice and for every $2$-modular lattice $\Lambda'$ containing $\Lambda(M)$, we have that $M \in \V_{\Lambda'}^{(2)}(k)$. 
  \end{itemize} 
  In particular, this means that there exist lattices $M \in \V_{\Lambda}^{(2)}(k)$ such that $\Lambda(M)$ is not a vertex lattice.
\end{lem}
\begin{proof}
  \par{(i)} Recall that in the definition of the stratification of $Q_W = \bigsqcup_{i= 0}^2 Z_i$ given in Lemma \ref{lem:QVstrata} the closed points of $Z_0$ correspond to isotropic, $2$-dimensional $\Phi$-stable subspaces $U$ of $W_k$. Then we can argue as in the proof of Lemma \ref{lem:splitV1} to see that if $M$ is sent to $Z_0$, then $\Lambda(M)$ is a vertex lattice. Conversely, as we have seen in the proof of Lemma \ref{lem:splitV1} the image in $W_k$ of a vertex lattice $L \subset \Lambda$ is always an isotropic subspace with respect to the symmetric form.  Indeed, observe that $(L_k + \pi\Lambda_k)^\vee = (L_k)^\vee \cap \pi\Lambda_k \supset \pi L_k$, which means that the image of $L_k + \pi\Lambda_k$ in $W_k$ is an isotropic subspace. Let $U$ be a point in $Z_1(k) \sqcup Z_2(k)$, then $U + \Phi(U)$ is not $\Phi$-stable, as it has dimension $3$ and the form is non-split. Hence, $U + \Phi(U) + \Phi^2(U)$ has dimension at least $4$ so cannot be isotropic. Since the image of $\Lambda(M)$ in $W_k$ is $\Phi$-stable and hence contains $U + \Phi(U) + \Phi^2(U)$, it cannot be isotropic. Therefore, in this case $\Lambda(M)$ cannot be a vertex lattice. This proves the first statement. 
  
  Let $M$ be in the preimage of $Z_0$ and let $L_2 = (M + \pi\Lambda)^\tau$, which by the discussion above we know is a vertex lattice and since $\pi\Lambda_k \subset ^2 M + \pi\Lambda_k$ it has type $4$. Clearly, $\pi\Lambda$, which is a vertex lattice of type $0$, is contained in $L_2$. Recall the simplicial complex $\mathscr{L}$ of vertex lattices introduced in Proposition \ref{prop:vlproperties}. In the non split case, $\pi\Lambda$ and $L_2$ are both vertices of this complex, which we know is connected and isomorphic to the Bruhat-Tits building for $\mathrm{SU}(C)(\Q_p)$. It follows that we can find a vertex lattice $L_1$ of type $2$ such that $\pi\Lambda \subset^1 L_1 \subset^1 L_2$. Consider the $\F_p$-vector space $L_1 / L_1^\vee$,  then the image of $\pi\Lambda$ in this quotient is a Lagrangian subspace of dimension $1$. Consider a Lagrangian complement of $\pi\Lambda$ in $L_1 / L_1^\vee$. Its preimage in $C$ is again a self-dual $\mathcal{O}_E$-lattice contained in $L_1$, hence by Proposition \ref{prop:2vl} we can identify it with a lattice of the form $\pi\Lambda_1$ for some $2$-modular lattice $\Lambda_1$. Moreover, since its image modulo $L_1^\vee$ is a Lagrangian complement of $\pi\Lambda$ we have $L_1 = \pi\Lambda + \pi\Lambda_1$.

  We show that if $M \neq \pi\Lambda_1 \otimes W(k)$ then $M \in \V_{\Lambda_1}^{(1)}(k)$. Suppose this is not the case and $M \in \V_{\Lambda_1}^{(2)}(k)$, which means $\pi\Lambda_1 \subset^2 M + \pi\Lambda_1$, here we omit the subscript $k$ for better readability. Since $\pi\Lambda_1 \subset L_2 = M + \pi\Lambda$  it follows that $M + \pi\Lambda_1 \subset M + \pi\Lambda$, and since both contain $M$ with index two, this inclusion is actually an equality. 
  Let $U$ be the image of $M$ in $V_k$ and consider  the chain of subspaces in  $V_k$
  \begin{equation*}
    \overline{\pi\Lambda}_1 \subsetneq \overline{\pi\Lambda_1} + \overline{\pi\Lambda} \subsetneq U + \overline{\pi\Lambda}_1 = U + \overline{\pi\Lambda},
  \end{equation*}
  obtained as the image in $V$ of the chain of inclusions of lattices 
  $\pi\Lambda_1 \subset^1 L_1 = \pi\Lambda + \pi\Lambda_1 \subset^1 L_2 = M + \pi\Lambda_1 = M + \pi\Lambda$. Observe that the inclusions remain proper in $V_k$ as $\pi\Lambda_1 \subset  \Lambda$ and by duality $\pi^2\Lambda \subset \pi\Lambda_1$. 
  Since the image $\overline{\pi\Lambda_1}$ of $\pi\Lambda_1$ in $V_k$ is contained with codimension $2$ in $U + \overline{\pi\Lambda_1}$ we can find two vectors $u_1, u_2 \in U$ such that $U + \overline{\pi\Lambda_1} = \langle u_1, u_2 \rangle \oplus \overline{\pi\Lambda_1}$. Moreover, by the inclusions above, we can actually choose these two vectors such that $u_2  \in \overline{\pi\Lambda} \cap U$. However, since $U + \overline{\pi\Lambda_1} = U + \overline{\pi\Lambda}$ and all these spaces are Lagrangian, by taking the orthogonal on both sides $U \cap \overline{\pi\Lambda_1} = U \cap \overline{\pi\Lambda}$. This means that $u_2 \in \overline{\pi\Lambda_1}$ which leads to a contradiction.

  \par{(ii)} Consider now $M \in \V_{\Lambda}^{(2)}(k)$ with image $U$ in $Z_1(k)$. Denote by $L$ the image of $M$ in $V$. By the definition of $Z_1$ in Lemma \ref{lem:QVstrata}, we know that $U \cap \Phi(U)$ is a one-dimensional, $\Phi$-stable subspace of $W_k$. Let $x \in \mathcal{L}_k \cong W_k$ be a $\Phi$-stable element such that $U \cap \Phi(U) = \langle x \rangle$ and let $x_L$ be a lift in $L$. Since $x_L$ is $\Phi$-stable modulo $\overline{\pi\Lambda}_k$, there is an element $\pi\lambda \in \overline{\pi\Lambda}_k$ such that $\Phi(x_L) = x_L + \pi\lambda$. Consider the seven-dimensional subspace $N_x = \langle x \rangle  \oplus \overline{\pi\Lambda}_k$ and its orthogonal $N_x^\vee \subset \overline{\pi\Lambda}_k$. Observe that $N_x$ is $\Phi$-stable. As we have observed several times in this section, $\Phi(N_x^\vee) = \Phi(N_x)^\vee$, and therefore $N_x^\vee$ is $\Phi$-stable as well.
  
  We show that $X = \langle x_L \rangle \oplus N_x^\vee$ is the space we are looking for, that is, it corresponds to another $2$-modular lattice $\Lambda'$ such that $M \in \V^{(1)}_{\Lambda'}$. First, observe that $X$ is Lagrangian, since $N_x^\vee \subset^1 X \subset^1 N_x$ and hence we can argue as in the proof of Lemma \ref{lem:spW}. We need to prove that $X$ is $\Phi$-stable. First, we note that since $N_x = U \cap \Phi(U) \oplus \overline{\pi\Lambda}_k = (L + \overline{\pi\Lambda}_k) \cap (\Phi(L) + \overline{\pi\Lambda}_k)$ it follows that $N_x^\vee = (L \cap \overline{\pi\Lambda}_k) + (\Phi(L) \cap \overline{\pi\Lambda}_k)$ as all summands appearing are Lagrangian. Complete $x$ to a basis $\{x, y\}$ of $U \subset  \mathcal{L}_k$ and denote by $y_L$ an element in $L$ such that $L = \langle x_L, y_L \rangle \oplus (L \cap \overline{\pi\Lambda}_k)$. Then we have
  \begin{equation*}
    L + \Phi(L) = \langle x_L, y_L, \Phi(y_L) \rangle \oplus (\langle \pi\lambda \rangle + (L \cap \overline{\pi\Lambda}_k) + (\Phi(L) \cap \overline{\pi\Lambda}_k)) = \langle x_L, y_L, \Phi(y_L) \rangle \oplus (\langle \pi\lambda \rangle + N_x^\vee).
  \end{equation*}
  Since the image of $\langle x_L, y_L, \Phi(y_L) \rangle$ in $W_k$ is  $U + \Phi(U)$ and has dimension $3$, these elements are linearly independent. Moreover, we already know that $N_x^\vee$ has dimension $5$. Since $L \in S_V$, the dimension of $L + \Phi(L)$ cannot be larger than eight, so we have that $\pi\lambda \in N_x^\vee$. We conclude that 
  \begin{equation*}
    \Phi(X) = \langle \Phi(x_L) \rangle \oplus \Phi(N_x^\vee) = \langle x_L + \pi\lambda \rangle \oplus N_x^\vee = \langle x_L \rangle \oplus N_x^\vee = X.
  \end{equation*}

  It remains to prove that $X$ can be lifted to a lattice in $C \otimes W(k)_\Q$. Since $U = \langle x, y \rangle $ is isotropic with respect to the symmetric form on $W_k \cong \mathcal{L}_k$ we have that $\overline{\pi}(U) \subset U^\vee \cap \overline{\pi\Lambda}_k \subset \langle x \rangle^\vee \cap \overline{\pi\Lambda}_k = N_x^\vee$. It follows that $\overline{\pi}(X) = \overline{\pi}(x) \subset \overline{\pi}(U) \subset N_x^\vee \subset X$. In particular $X \in S_{V \pi}$, hence we can lift it to a $\tau$-stable, self-dual $W(k) \otimes \mathcal{O}_E$-lattice $X$. Since $\overline{\pi}(L) \subset X$ we have $\pi M \subset X$. By Proposition \ref{prop:2vl} we know that $\Lambda' = \pi^{-1} X^{\tau}$ is a $2$-modular lattice, which then contains $\Lambda(M)$. Moreover, since $X \cap L \subset^1 L$, we have that $\pi\Lambda'_K \subset^1 M + \pi\Lambda'_k$, in other words $M \in \V_{\Lambda'}^{(1)}(k)$.

  It remains to prove the ``if'' part of the second statement. Suppose $M$ is a lattice in $\V(k)$ such that there are two $2$-modular lattices $\Lambda_{1,2} \subset C$ such that $M \subset \Lambda_1 \cap \Lambda_2$ and $M \subset^i M + \pi\Lambda_i$, (here we omit the subscript $k$ to ease the notation). Moreover, $\Lambda(M)$ is not a vertex lattice. We want to prove that $M$ is mapped to a point in $Z_1(k)$ by the map $M \mapsto M + \pi\Lambda_2$ associated to $\Lambda_2$. By definition of $Z_1$ this is equivalent to showing that $(M + \pi\Lambda_2) \cap (\tau(M) + \pi\Lambda_2)$ is $\tau$-stable. It suffices to prove that $\pi\Lambda_1 + \pi\Lambda_2 \subset M + \pi\Lambda_2$. Indeed, if this is the case by $\tau$-stability $\pi\Lambda_1 + \pi\Lambda_2 \subset (M + \pi\Lambda_2) \cap (\Phi(M) + \pi\Lambda_2)$. Since $M + \pi\Lambda_2$ is not $tau$-stable, otherwise $\Lambda(M)$ would be a vertex lattice, it follows that the inclusion above is an equality. Consider the inclusions
  \begin{equation*}
    M \subset M + (\pi\Lambda_1 \cap \Lambda_2) \subset M + \pi\Lambda_1.
  \end{equation*} 
  Since the index of $M$ in $M + \pi\Lambda_1$ is $1$ we have that one of the inclusions above is actually an equality. If $M = M + (\pi\Lambda_1 \cap \Lambda_2)$ or equivalently $\pi\Lambda_1 \cap \Lambda_2 \subset M$, by taking duals on both sides we have $M \subset \pi\Lambda_1 + \pi^2\Lambda_2$. The latter is a $\tau$-stable lattice such that $\pi(\pi\Lambda_1 + \pi^2\Lambda_2) = \pi^2\Lambda_1 + \pi^3\Lambda_2 \subset \pi\Lambda_1 \cap \Lambda_2 = (\pi\Lambda_1 + \pi^2\Lambda_2)^\vee \subset \pi\Lambda_1 + \pi^2\Lambda_2$, where the first inclusion follows from the fact that $\pi^2\Lambda_{1,2} \subset M \subset \Lambda_{1,2}$. Therefore, in this case $M$ is contained in a $\tau$-stable vertex lattice, which contradicts the assumption on $\Lambda(M)$. It follows that the second inclusion above is an equality, that is $M + \pi\Lambda_1 = M + (\pi\Lambda_1 \cap \Lambda_2)\subset \Lambda_2$ from which it follows that $\pi\Lambda_1 \subset \Lambda_2$. Observe that by taking duals we also have $\pi^2 \Lambda_2 \subset \pi\Lambda_1$ from which we conclude that $\pi\Lambda_1 + \pi\Lambda_2 \subset \Lambda_1 \cap \Lambda_2$. Observe that $\pi\Lambda_1 + \pi\Lambda_2$ is a $\tau$-stable vertex lattice. Finally, consider the inclusions
  \begin{equation*}
    M \cap \pi\Lambda_1 \subset (M \cap \pi\Lambda_1) + (\pi\Lambda_1 \cap \pi\Lambda_2) \subset \pi\Lambda_1.
  \end{equation*}
  Again by the fact that $M \cap \pi\Lambda_1$ has index one in $\pi\Lambda_1$, one of the inclusions above is an equality. If the first inclusion is an equality, which is equivalent to $\pi\Lambda_1 \cap \pi\Lambda_2 \subset M \cap \pi\Lambda_1 \subset M$ by taking duals we have $M \subset \pi\Lambda_1 + \pi\Lambda_2$, and we have just observed that the latter is a $\tau$-stable vertex lattice, which contradicts the assumption on $\Lambda(M)$. It follows that the second inclusion is an equality, hence $\pi\Lambda_1 \subset (M \cap \pi\Lambda_1) + (\pi\Lambda_1 \cap \pi\Lambda_2) \subset M + \pi\Lambda_2$ and therefore $\pi\Lambda_2 \subsetneq \pi\Lambda_1 + \pi\Lambda_2 \subsetneq M + \pi\Lambda_2$. Here the first inclusion is proper as $\Lambda_{1,2}$ are distinct, while the second is because $M + \pi\Lambda_2$ is not $\tau$-stable. Since $\pi\Lambda_1 + \pi\Lambda_2$ is a $\tau$-stable lattice contained in $M + \pi\Lambda_2$, it follows that its image modulo $\pi\Lambda_2$ is a $\Phi$-stable subspace contained in the image $U$ of $M$. By $\Phi$-stablity it is contained in the intersection $U \cap \Phi(U)$ which has dimension $1$. It follows that $U \cap \Phi(U)$ is $\Phi$-stable and therefore $M$ is sent to a point of $Z_1(k)$.

  \par{(iii)} The last statement follows directly from the previous two.
\end{proof}

\begin{lem}\label{lem:spW2c}
  We denote by $\V^{(2)^\circ}_{\Lambda}(k)$ the preimage of $Z_2(k)$, that is the set of lattices $M \in \V_{\Lambda}^{(2)}(k)$ such that $\Lambda(M)$ is not a vertex lattice and $M \in \V_{\Lambda'}^{(2)}$ for every $2$-modular lattice $\Lambda(M) \subset \Lambda'$. Then the restriction of the map of Lemma \ref{lem:spW2} induces a surjective map 
  \begin{equation*}
    \V^{(2)^\circ}_{\Lambda}(k) \longrightarrow Z_2,
  \end{equation*}
 with fibers equal to $\mathbb{A}^2(k)$.
\end{lem}
\begin{proof}
  It remains to study the fibers of the map  $\V^{(2)^\circ}_{\Lambda}(k) \longrightarrow Z_2(k)$. Recall the isomorphism of Lemma \ref{lem:QVstrata} between $Z_2$ and the union of Deligne-Lusztig variety $X_B(t_2t_1) \cup X_B(t_3t_1)$. On closed points it gives a bijection $U \mapsto U \cap \Phi(U) \subset U$, compare Lemma \ref{lem:QVstrata}. Fix $U \in Z_2$, and let $l = U \cap \Phi(U) \subset \mathcal{L}_k$. We have already seen that $L = U \oplus N^\vee$ is a preimage in $S_{V\pi}$ (or equivalently it produces a preimage in $\V_{\Lambda}^{(2)}$) of $U$. Fix a basis $l = \langle u \rangle$ and let $l' =  \langle u + \pi\lambda \rangle$ be another lift of $l$ in $V_k$. If $L' \in S_{V \pi}$ is another preimage of $U$ and contains $l'$, then it is of the form
  \begin{equation*}
    L' = \langle u + \pi\lambda, \Phi^{-1}(u) + \pi\lambda_2 \rangle \oplus N^\vee,
  \end{equation*} for some $\pi\lambda_2$ in the $2$-dimensional space $\overline{\pi\Lambda}/N^\vee$. Consider
  \begin{align*}
    L' + \Phi(L') &= \langle u + \pi\lambda, \Phi^{-1}(u) + \pi\lambda_2, \Phi(u) + \Phi(\pi\lambda), u + \Phi(\pi\lambda_2) \rangle \oplus ( N^\vee + \Phi(N^\vee)) \\
    &= \langle u + \pi\lambda, \Phi^{-1}(u) + \pi\lambda_2, \Phi(u) + \Phi(\pi\lambda) \rangle \oplus ( \langle \pi\lambda - \Phi(\pi\lambda_2) \rangle + N^\vee + \Phi(N^\vee)).
  \end{align*}
  Since $L' + \Phi(L')$ has to have dimension eight, and its subspace
  \begin{equation*} 
    \langle u + \pi\lambda, \Phi^{-1}(u) + \pi\lambda_2, \Phi(u) + \Phi(\pi\lambda) \rangle \oplus (N^\vee + \Phi(N^\vee)) 
  \end{equation*}
  already has dimension $3 + 5 = 8$ we have that $\Phi((\pi\lambda_2)) \in \pi\lambda + (N^\vee + \Phi(N^\vee))$, which means that $\pi\lambda_2$ belongs to the one-dimensional (affine) subspace $\Phi^{-1}(\pi\lambda) + (N^\vee + \Phi^{-1}(N^\vee))/N^\vee \subset \overline{\pi\Lambda}_k/N^\vee$. Moreover, since $L'$ has to be Lagrangian, we have to impose another linear condition
  \begin{equation*}
    \langle \pi\lambda, \Phi^{-1}(u)\rangle = \langle \pi\lambda_2, u \rangle.
  \end{equation*} Observe that $N^\vee + \Phi(N^\vee) \subset \langle u \rangle^\vee \cap \overline{\pi\Lambda}_k$ as $u$ is contained in both the Lagrangian spaces $L$ and $\Phi(L)$. By comparing dimension we have equality and since $\Phi^{-1}(N^\vee)$ is not contained in $N^\vee + \Phi(N^\vee)$ it follows that $\Phi^{-1}(N^\vee)$ is not orthogonal to $u$. Therefore, the linear condition on $\pi\lambda_2$ above is non-trivial and determines a unique point in $\Phi^{-1}(\pi\lambda) + (N^\vee + \Phi^{-1}(N^\vee))/N^\vee \subset \overline{\pi\Lambda}_k/N^\vee$. It follows that a preimage $L'$ of $U$ is uniquely determined by how it lifts the subspace $U \cap \Phi(U)$, that is by a unique element in the $2$-dimensional space $\overline{\pi\Lambda}_k/N^\vee \cong \mathbb{A}^2(k)$. 
\end{proof}

\section{Geometry of $\bar{\N}^0$}\label{sec:geometry}
\noindent
In this section we study the irreducible components of the reduced scheme underlying $\bar{\N}^0$. Recall that the irreducible components of the analogous scheme for $\GU(1, n-1)$ are indexed over the set of vertex lattices of maximal type, as proved in \cite{rtw}, and are isomorphic to generalized Deligne-Lusztig varieties. In our case, if the form is split, we are going to see that in addition to components analogous to those of \textit{loc.cit.}, a second type of irreducible components appears. These components originate from $2$-modular lattices, and are universally homeomorphic to line bundles over a Deligne-Lusztig variety. In the non-split case, we are going to prove that there are again two types of irreducible components, which both originate from $2$-modular lattices and are universally homeomorphic to line bundles over a generalized Deligne-Lusztig variety, respectively to the closure of vector bundles of rank $2$ over a classical Deligne-Lusztig variety of Coxeter type. 

\subsection{The subscheme $\N_{\Lambda}$}
Let $k$ be a perfect field containing $\F$. Let $\Lambda$ be a vertex lattice or a $2$-modular lattice in $C$. We write again $\Lambda_k$ for $\Lambda \otimes_{\mathbb{Z}_p}W(k)$. We first define the closed subfunctor $\N_{\Lambda}$ of $\bar{\N}^{0}$ associated to  $\Lambda$, whose $\F$-points are in bijection with $\V_{\Lambda}(\F) = \{M \in \V(\F) \mid \Lambda(M) \subset \Lambda\}$. The construction is similar to that of \cite{rtw}*{Sec.\ 6} and \cite{vw}*{Sec.\ 4}, we recall here the main ideas and point out the differences, that are due to the fact that now we have to consider $2$-modular lattices as well. 

\begin{lem}\label{lem:XLambdaplus}
Let $\Lambda^{+} = \Lambda_{\F}$ and $\Lambda^{-} = \Lambda^{\vee}_{\F}$. They correspond to two $p$-divisible groups $X_{\Lambda^{\pm}}$ with  quasi-isogenies $\rho_{\Lambda^{\pm}}: X_{\Lambda^{\pm}} \rightarrow \mathbb{X}$.
\end{lem}
\begin{proof}
  If $\Lambda$ is a vertex lattice, it is proved in \cite{rtw}*{Lem.\ 6.1} that both $\Lambda^{+}$ and $\Lambda^{-}$ are stable under $\pi, F, V$ from which the claim follows by Dieudonn\'e theory. If $\Lambda$ is a $2$-modular lattice, we know that $\Lambda$ is $\pi$- and $\tau$-stable, and consequently so is $\Lambda^{+}$. We then have \begin{equation*}
    F\Lambda^{+} = pV^{-1}\Lambda^{+} = \pi^2 V^{-1}\Lambda^{+} = \pi \tau \Lambda^{+} = \pi \Lambda^{+} \subset \Lambda^{+}.
  \end{equation*}
  Similarly, $V\Lambda^{+} = \pi \tau^{-1} \Lambda^{+} = \pi \Lambda^{+} \subset \Lambda^{+}$. As in the proof of \cite{rtw}*{Lem.\ 6.1}, we observe that for $x \in \Lambda^{-} = (\Lambda^{+})^{\vee}$ and $y \in \Lambda$ we have $ \langle Fx, y \rangle = \langle x, Vy \rangle^\sigma$ which is integral, since $Vy$ is again in $\Lambda^{+}$, as we have just seen. Therefore, $Fx \in \Lambda^{-}$. In the same way one proves that $\Lambda^{-}$ is stable under $V$. To prove that it is stable under $\pi$ it is enough to recall that $\Lambda^{-} = \pi^2\Lambda^{+}$, as $\Lambda$ is a $2$-modular lattice, and the statement follows from $\pi$-stability of $\Lambda^+$.

  Recall $N$, the rational Dieudonn\'e module of $\mathbb{X}$. From the inclusions $\Lambda^{\pm} \subset C\otimes_{\Q_p} W(\F)_{\Q} \cong N$,  again by means of Dieudonn\'e theory, we obtain the quasi-isogenies to $\mathbb{X}$.
\end{proof}

As in \cite{rtw}*{Sec.\ 6} we define the subfunctor $\widetilde{\N}_{\Lambda}$ of $\bar{\N}^0$ consisting of the tuples $(X, \rho, \lambda, \iota)$ over an $\F$-scheme $S$, such that $\rho_{X, \Lambda^+} \vcentcolon = (\rho_{\Lambda^{+}})^{-1}_S \circ \rho $ or equivalently $\rho_{X, \Lambda^-} \vcentcolon = \rho^{-1} \circ (\rho_{\Lambda^{-}})_S$ is an isogeny. Observe that as in \textit{loc.cit.}\ there is a commutative diagram
\[\begin{tikzcd}
  X_{\Lambda^{+}} \arrow{r}{\lambda} \arrow[swap]{d}{\rho_{\Lambda^+}} & X_{\Lambda^{-}}  \\%
  \mathbb{X} \arrow{r}{\lambda_{\mathbb{X}}}& \mathbb{X}^\vee \arrow[swap]{u}{{}^t\rho_{\Lambda^-}}
\end{tikzcd}\]
where $\lambda$ is the isogeny induced by the duality of lattices $\Lambda^{-} = (\Lambda^{+})^{\vee} = (\Lambda^{+})^{\sharp}$ and ${}^t\rho_{\Lambda^-}$ is the dual of the quasi-isogeny $\rho_{\Lambda^{+}}$. Since, by definition of $\bar{\N}^0$, the height of $\rho_X$ is zero, it follows that the height of the isogenies $\rho_{X, \Lambda^{\pm}}$, is half of the type $t(\Lambda)$ of $\Lambda$, \textit{i.e.}\ of the index of $\Lambda^\vee \subset \Lambda$. The next lemma  is proved in the same way as \cite{rtw}*{Lem.\ 6.2} and \cite{vw}*{Lem.\ 4.2}, as the arguments there do not make use of the fact that $\Lambda$ is a vertex lattice, or that the extension $\Q_p \subset E$ is (un)ramified. 

\begin{lem} The functor $\widetilde{\N}_{\Lambda}$ is representable by a projective $\F$-scheme, and it is a closed subscheme of $\bar{\N}^0$. \qed
\end{lem}

Denote by $\N_{\Lambda}$ the reduced scheme underlying $\widetilde{\N}_{\Lambda}$. Our goal is to extend to a morphism of schemes the bijection, respectively surjection, we have described in the previous section between the $k$-valued points of $\N_{\Lambda}$ and the Deligne-Lusztig varieties $S_V$, respectively $R_W$ or $Q_W$. The first step in this direction is given by defining a morphism of $\widetilde{\N}_{\Lambda}$ into a Grassmannian variety. 

As in the previous section we let $V$ denote the $t(\Lambda)$-dimensional vector space $V = \Lambda^{+}/\Lambda^{-}$. Consider the Grassmannian functor $\rm Grass(V)$ parametrizing subspaces of dimension $t(\Lambda)/2$ in $V$. Then $\rm Grass(V)$ is represented by a projective $\F$-scheme. The construction of a morphism $\widetilde{\N}_{\Lambda} \rightarrow \rm Grass(V)$ follows the strategy of \cite{vw}*{4.6}. First, we consider the Lie algebra of the universal vector extension of a $p$-divisible groups. It defines a functor $X \mapsto D(X)$. As remarked in \textit{loc.cit.}, if $X$ is a $p$-divisible group over a perfect field $k$ with Dieudonn\'e module $M$, this is just $D(X) \cong M/pM$ via a functorial isomorphism. The following is the key result to construct the desired map. 

\begin{prop}\cite{vw}*{Cor.\ 4.7}
  Let $S$ be an $\F$-scheme and let $\rho_i: X\rightarrow Y_i$ for $i =1,2$ be two isogenies of $p$-divisible groups over $S$ such that $\mathrm{Ker}(\rho_1) \subset  \mathrm{Ker}(\rho_2) \subset X[p]$. Then $\mathrm{Ker}(D(\rho_1))$ is locally a direct summand of the locally free $\mathcal{O}_S$-module $\mathrm{Ker}(D(\rho_2))$ of rank $\mathrm{ht}(\rho_1)$. The formation of $\mathrm{Ker}(D(\rho_i))$ commutes with base change in $S$. \qed
\end{prop}

For an $\F$-algebra $R$ consider $X \in \widetilde{\mathcal{N}}_{\Lambda}(R)$ and the isogenies $\rho_{X, \Lambda^{\pm}}$ defined above. Then by definition, the composition
\begin{equation*}
  \rho_{\Lambda}: X_{\Lambda^{-}} \xrightarrow{\rho_{X, \Lambda^{-}}} X \xrightarrow{\rho_{X, \Lambda^{+}}} X_{\Lambda^{+}}
\end{equation*}
is the isogeny induced by the inclusion of Dieudonn\'e modules $\Lambda^{-} \subset \Lambda^{+}$. Since $\Lambda$ is either a vertex lattice or a $2$-modular lattice we know that in any case $p\Lambda = \pi^2 \Lambda \subset \Lambda^\vee$, which means that $p \Lambda^{+} \subset \Lambda^{-}$. It follows that $\mathrm{Ker}(\rho_{X, \Lambda^-}) \subset \mathrm{Ker}(\rho_{\Lambda}) \subset X_{\Lambda^-}[p]$, so we can apply the previous proposition to conclude that $\mathrm{Ker}(D(\rho_{X,\Lambda^{-}}))$ is a direct summand of $\mathrm{Ker}(D(\rho_{\Lambda}))$ of rank $\mathrm{ht}(\rho_{X, \Lambda^-}) = t(\Lambda)/2$. Moreover, this construction commutes with base change. As in \textit{loc.cit.}, we observe that $\mathrm{Ker}(D(\rho_{\Lambda}))$ is just the vector space $V$. Similarly, if $R$ is a perfect field and $M$ is the Dieudonn\'e module of $X$ we have, by the isomorphism $D(X) \cong M/pM$ that $\mathrm{Ker}(D(\rho_{X, \Lambda^-})) \cong M/\Lambda^-$.

Recall the variety $S_V$ defined in Section \ref{sec:dlvsp}. It parametrizes Lagrangian subspaces $U$ in a symplectic vector space $V$ such that the codimension of $U$ in $U + \Phi(U)$ is at most $2$. Let $\Lambda$ be a vertex lattice and let $V$ be the $\F$-vector space $\Lambda/\Lambda^\vee$, which we have introduced in the previous section. The following result is analogous to \cite{rtw}*{Prop.\ 6.7} and \cite{vw}*{Thm.\ 4.8}, and it is proved in the same way, see also Remark \ref{rem:notiso} below for a comparison.

\begin{prop}\label{prop:vl}
  Let $\Lambda$ be a vertex lattice in $C$ and denote by $\N_{\Lambda}$ the reduced scheme underlying $\widetilde{\N}_{\Lambda}$. The map $f: \N_{\Lambda} \rightarrow \Grass(V)$ that sends $(X, \lambda, \rho, \iota)$ to $E(X) \vcentcolon = \ker(D(\rho_{X, \Lambda^-}))$ is a morphism of projective schemes, and it induces a  universal homeomorphism $\N_{\Lambda} \longrightarrow S_V$.
\end{prop}
\begin{proof}
  By the previous proposition, $f$ commutes with base change, hence it is a morphism of projective schemes. Since we are working with projective schemes over the algebraically closed field $\F$ it is enough to check that the map induced on $k$-valued points is a bijection for any algebraically closed field $k$. Indeed, $f$ is a universal homeomorphism if and only if it is universally injective, finite and surjective, compare for example \cite{gw}*{Ex.\ 12.32}. Since we are considering reduced schemes, surjectiveness can be tested on $k$-valued points for $k$ algebraically closed. Universally injective is equivalent to the diagonal morphism being a bijection on $k$-valued points for any field $k$. Since a morphism of projective schemes is proper, hence separated, the diagonal morphism is already injective as it is a closed immersion. Moreover, for a scheme $X$ of finite type over an algebraically closed field $k$, the set of $k$-valued points is very dense in $X$, see \cite{gw}*{Prop.\ 3.35}. Therefore, a closed subscheme $Y \subset X$ coincides with $X$ if and only if it contains all $k$-valued points. This means that surjectiveness of the diagonal $\Delta_f$, which is equivalent to injectiveness of $f$, can be tested on $k$-points for $k$ algebraically closed. Last, a morphism is finite if and only if it is proper and quasi-finite. By \cite{gw}*{Rem.\ 12.16} it is sufficient to see if the map has finite fibers on $k$-valued points, with $k$ algebraically closed, which is already implied by being injective. Then we can conclude with Lemma \ref{lem:spV}.
\end{proof}

Recall that if $\Lambda$ is a $2$-modular lattice, the action of $\pi$ induces a linear map $\overline{\pi}$ of rank $6$ on the $12$-dimensional vector space $V = \Lambda^+/\Lambda^- = \Lambda^+/\pi^2\Lambda^+$. The image and kernel of $\overline{\pi}$ both coincide with the $6$-dimensional subspace $\overline{\pi\Lambda}$ given by the image of $\pi\Lambda$ in $V$. Consider then the closed subscheme $S_{V\pi}$ of the variety $S_V$, given by the Lagrangian subspaces $U \in S_V$ such that $\overline{\pi}(U) + \overline{\pi}(\Phi(U)) \subset U$ (observe that this is equivalent to the condition $\overline{\pi}(U) + \overline{\pi}(\Phi(U)) \subset U$ originally given in the definition of $S_{V\pi}$). Recall $S_{V\pi}$ has already been introduced in the proof of Lemma \ref{lem:spW}, where we have proven that it is the image of the map $\V_{\Lambda}(k) \rightarrow S_V(k)$ for $k$ an algebraically closed field. 

\begin{prop}\label{lem:closedimm}
  Let $\Lambda$ be a $2$-modular lattice in $C$ and denote by $\N_{\Lambda}$ the reduced scheme underlying $\widetilde{\N}_{\Lambda}$. The map $f : {\N}_{\Lambda} \rightarrow S_{V\pi}$ that sends $(X, \lambda, \rho, \iota)$ to $E(X) \coloneqq \mathrm{Ker}(D(\rho_{X,\Lambda^{-}}))$ is a universal homeomorphism of projective schemes.
\end{prop}
\begin{proof}
  As in the proof of Proposition \ref{prop:vl}, since we are working with reduced projective schemes over the algebraically closed field $\F$ it is enough to check that the map on $k$-valued points is a bijection, for any algebraically closed field $k$. Then we can conclude with Lemma \ref{lem:spW}.
\end{proof}

The remainder of this section is dedicated to combining the results on the geometric points of $\widetilde{\N}_{\Lambda}$ proved in the previous section, see Lemmas \ref{lem:spV} and \ref{lem:spW}, with the construction of Lemma \ref{lem:closedimm} of the universal homeomorphism $f$ onto the variety $S_{V\pi}$. Our goal is to obtain a description of the irreducible components of $\bar{\N}_{\mathrm{red}}^{0}$ in terms of Deligne-Lusztig varieties. Again the split and non-split cases are rather different and deserve to be treated separately.

\subsection{Irreducible components in the split case}
Assume that the Hermitian form on $C$ is split. As we have seen in Lemma \ref{lem:splitV1}, if a lattice $M \in \bar{\N}^0(k)$ is not contained in a vertex lattice, then it is contained in a $2$-modular lattice $\Lambda$ such that $\pi \Lambda_k \subset^1 M + \pi\Lambda_k$. These two cases will correspond to two types of irreducible components of $\bar{\N}^0_{\rm red}$.

\begin{rem}\label{rem:stratavl}
  We have already seen that if $\mathcal{L}$ is a vertex lattice, then $\N_{\mathcal{L}}$ is universally homeomorphic to the generalized Deligne-Lusztig variety $S_V$. Let $\mathcal{L}$ be a vertex lattice of type $6$ in $C$, which exists as we are considering the split case. Then $V = \mathcal{L}/\mathcal{L}^\vee$ is a symplectic $\F$-vector space of dimension $6$. It follows from Lemma \ref{lem:SVdim} that $\N_{\mathcal{L}}$ is irreducible and has dimension $5$. Moreover, it contains the open and dense subscheme $\N_{\mathcal{L}}^{\circ}$ corresponding to the open stratum $X_{B}(s_3s_2s_1) \sqcup X_B(s_3s_2s_3s_1) \sqcup X_B(s_3s_2s_3s_1s_2)$ in the stratification  of $S_V$ given in (\ref{eq:strC}). In terms of lattices, $\N_{\mathcal{L}}^\circ$ corresponds to those lattices $M$ in $N$ such that $\Lambda(M) = \mathcal{L}$.

  A similar stratification holds for lattices of smaller type, too. In particular, by (\ref{eq:strC}) for any vertex lattice $\mathcal{L}$, the corresponding scheme $\N_{\mathcal{L}}$ contains $\N_{\mathcal{L}}^{\circ}$ as an open and dense subscheme. Moreover, by Lemma \ref{lem:SVstr} and Proposition \ref{prop:vl} the closure of $\N_{\mathcal{L}}^{\circ}$ is the union of the subschemes $\N_{\mathcal{L'}}^\circ$ for all vertex lattices $\mathcal{L'} \subset \mathcal{L}$.
\end{rem}

Let now $\Lambda$ be a $2$-modular lattice in $C$ and denote again by $\N_{\Lambda}$ the reduced scheme underlying $\widetilde{\N}_{\Lambda}$. In Lemma \ref{lem:splitV1} we have seen that if $M$ is a lattice in $\V_{\Lambda}(k)$ such that the corresponding minimal $\tau$-stable lattice $\Lambda(M)$ is not a vertex lattice, then the index of $\pi\Lambda_k$ in $M + \pi\Lambda_k$ is $1$. Therefore, we are interested in the closed subscheme of $S_{V \pi}$ given by
\begin{equation}
  S_{V \pi}^{\small{\le 1}}(R) \coloneqq \{ U \in S_{V \pi}(R) \mid U + \overline{\pi\Lambda}_R \text{ is a direct summand of } V_R \text{ with } \mathrm{rk}(\overline{\pi\Lambda}+ U ) \le 7\},
\end{equation}
where $\overline{\pi\Lambda}_R$ is the image of $\pi\Lambda_R$ in $V_R = \Lambda_R/\pi^2\Lambda_R$. Consider the open subscheme $S_{V\pi}^{(1)}$ defined by the condition on the rank being an equality. We denote by $\N_{\Lambda}^{\le 1}$ and $\N_{\Lambda}^{(1)}$ their schematic preimage in $\N_{\Lambda}$ under the morphism $f$. Since $f$ is a universal homeomorphism, $\N_{\Lambda}^{\le 1}$ is closed in $\N_{\Lambda}$ and contains $\N_{\Lambda}^{(1)}$ as an open subscheme. 
\begin{lem}\label{lem:dense}
  The subscheme $\N_{\Lambda}^{(1)}$ is open and dense in $\N_{\Lambda}^{\le 1}$.
\end{lem}
\begin{proof}
  Observe that the complement of $S^{(1)}_{V\pi}$ in $S^{\le 1}_{V\pi}$ consists only of the point $\overline{\pi\Lambda}$. It follows that the complement of $\N_{\Lambda}^{(1)}$ in $\N_{\Lambda}^{\le 1}$ consists only of the $p$-divisible group $X_{\pi\Lambda}$ corresponding via Dieudonné theory and Lemma \ref{lem:XLambdaplus} to the lattice $\pi\Lambda_\F$. Our goal is to show that it belongs to the closure of $\N_{\Lambda}^{(1)}$.

  Let $\mathcal{L} \subset C$ be a vertex lattice of type $2$ containing $\pi\Lambda$. Such a lattice exists since $\pi\Lambda$ is a vertex lattice of type $0$ and the simplicial complex $\mathscr{L}$ of vertex lattices is connected, compare Proposition \ref{prop:vlproperties}. Then, using the fact that $\pi\Lambda$ is self-dual and by definition of vertex lattices we have $\pi\mathcal{L} \subset \mathcal{L}^\vee \subset \pi\Lambda \subset^1 \mathcal{L}$, from which follows that $\mathcal{L} \subset \Lambda$. 

  Observe that if $M \in \V_{\mathcal{L}}^\circ(k)$, where $k$ is an algebraically closed field, we have that $M \in \N_{\Lambda}^{(1)}(k)$. Indeed, given such a lattice $M$, we know that $\Lambda(M) = \mathcal{L}$ and therefore $M$ is not $\tau$-stable. Since both $M$ and $\pi\Lambda_k$ are self-dual lattices, if $M$ is contained in $\pi\Lambda$ then it is equal to it. Since $M$ is not $\tau$-stable, this is not possible. Hence, we have inclusions $\pi\Lambda_k \subsetneq \pi\Lambda_k + M \subset \mathcal{L}_k$ and since $\pi\Lambda_k$ has index $1$ in $\mathcal{L}$ the claim follows. It follows that there is an inclusion of reduced schemes $\N_{\mathcal{L}}^\circ \subset \N_{\Lambda}^{(1)}$. By Remark \ref{rem:stratavl} we know that $\N_{\pi\Lambda}$ is contained in the closure of $\N_{\mathcal{L}}^\circ$, hence its only element $X_{\pi\Lambda}$ belongs to the closure of $\N_{\Lambda}^{(1)}$. 
\end{proof}

We shortly recall the \emph{universal vector bundle on the Grassmannian}, for more details we refer to \cite{gw}*{Ex.\ 11.9}. Let $\Grass_n(W)$ be the Grassmannian variety parametrizing subspaces of dimension $m$ in a given vector space $W$. Then the universal vector bundle over $\Grass_m(W)$ is a locally trivial vector bundle of rank $m$. Its $k$-valued points, for any field $k$, consist of pairs $(U, v)$, where $U$ is a subspace belonging to $\Grass_m(W)$ and $v$ is a vector in $U$. Roughly speaking, one identifies the fiber of the universal vector bundle over a subspace  $U$ with $U$ itself. 

In this section we are in particular interested in the universal line bundle $\mathcal{O}(1)$ over the projective space $\mathbb{P}(W)$, where $W$ denotes again the six-dimensional $\F$-vector space $\Lambda /\pi\Lambda$. We also consider
\begin{equation}
  \mathcal{H}^1 \coloneqq \Hom(\mathcal{O}(1), \mathcal{O}(-1)) = \mathcal{O}(-2).
\end{equation}

In order to study $S_{V\pi}^{(1)}$ we first need to study the   subscheme $\mathcal{S}^{(1)}$ of $\Grass(V)$ parametrizing the Lagrangian subspaces $U \subset V$ such that the dimension of the subspace $\overline{\pi\Lambda} + U$ is equal to $7$. In other words, $\mathcal{S}^{(1)}$ is the intersection of the Lagrangian Grassmannian $\mathcal{L}\Grass(V)$ with a Schubert cell. It is also clear that $S_{V\pi}^{(1)}$ is a closed subscheme of $\mathcal{S}^{(1)}$. 

\begin{lem}\label{lem:S1P1}
  The quotient map $q: V= \Lambda/\pi^2\Lambda \longrightarrow W = \Lambda/\pi\Lambda$ induces a morphism from $\mathcal{S}^{(1)}$ to the line bundle $\mathcal{H}^1$ over $\mathbb{P}(W)$.
\end{lem}
\begin{proof}
  We first observe that the quotient map $q$ induces a map $\mathcal{S}^{(1)} \rightarrow \mathbb{P}(W)$. This follows directly from the definition of $\mathcal{S}^{(1)}$ as intersection of the Lagrangian Grassmannian and a Schubert cell. An $R$-point $U \in \mathcal{S}^{(1)}(R)$ is sent by $q$ to the direct summand $(U + \overline{\pi\Lambda}_R) /\overline{\pi\Lambda}_R$ of $W_R$. If $R \rightarrow R'$ is a morphism of $\F$-algebras, since $\overline{\pi\Lambda}_{R}$ is a free submodule of $V_{R}$ the quotient by $\overline{\pi\Lambda}_{R}$ commutes with the tensor product $\cdot \otimes_R R'$. In other words we have
  \begin{equation*}
    (U\otimes_R R' + \overline{\pi\Lambda}_{R'}) /\overline{\pi\Lambda}_{R'} = ((U + \overline{\pi\Lambda})\otimes_R R' )/\overline{\pi\Lambda}_{R'} = \left ( (U + \overline{\pi\Lambda}_R) /\overline{\pi\Lambda}_R \right ) \otimes_R R',
  \end{equation*}  
  which proves that the map induced by $q$ commutes with base change by  $\F$-algebras. It follows that $q$ induces a morphism of projective $\F$-schemes $\mathcal{S}^{(1)} \rightarrow \mathbb{P}(W)$.
  
  Our aim is to construct a morphism of $\mathbb{P}(W)$-schemes $\mathcal{S}^{(1)} \rightarrow \mathcal{H}^1$. By Yoneda's Lemma, it is enough to give a map $\mathcal{S}^{(1)}(R) \rightarrow \mathcal{H}^1(R)$ for every $\F$-algebra $R$, and then prove that this map commutes with tensor products, compare \cite{gw}*{Cor.\ 4.7}. In other words, our goal is to associate to any Lagrangian $U$ in $\mathcal{S}^{(1)}(R)$ that is sent by $q$ to $l \in \mathbb{P}(W)$ an $R$-linear map $\psi_U: l \rightarrow l^*$. We have already seen this in the proof of Lemma \ref{lem:splitV1} for $R = k$, however, the proof there requires fixing a basis for $l$, which may not exist in general, for example when $R$ is not a local ring. We give here another construction that is independent of the choice of a basis. 
  
  Fix a Lagrangian complement $\mathcal{L}$ of $\overline{\pi\Lambda}$, \textit{i.e.}\ a Lagrangian subspace of $V$ such that $V = \mathcal{L} \oplus \overline{\pi\Lambda}$. Observe that for any $\F$-algebra $R$, since the form on $V_R$ is just the $R$-linear extension of that of $V$, the tensor product $\mathcal{L}_R$ remains a Lagrangian complement of $\overline{\pi\Lambda}_R$ in $V_R$. We identify $\mathcal{L} \cong W$. Let $l\in \mathbb{P}(\mathcal{L})(R)$ for an $\F$-algebra $R$. As in the previous section we denote by $N = l \oplus \overline{\pi\Lambda}_R$ its preimage in $V_R$ under $q$ and by $N^\vee = l^\vee \cap \overline{\pi\Lambda}_R$ its orthogonal. Since $N^\vee \subset \overline{\pi\Lambda}_R \subset N$ the subspace
  \begin{equation*}
    U_0 = l \oplus N^\vee
  \end{equation*} is a submodule of $V_R$ that is sent to $l$ by the quotient map $q$. Observe that for any $x \in l$ and any $v + \pi\lambda \in l \oplus \overline{\pi\Lambda}_R$, since $l$ is contained in the Lagrangian $\mathcal{L}$, we have $\langle x, v + \pi\lambda \rangle = \langle x, \pi\lambda \rangle$. It follows that the orthogonal of $U_0$ satisfies
  \begin{equation*}
    U_0^\vee = l^\vee \cap N = l^\vee \cap (l \oplus \overline{\pi\Lambda}_R) = l \oplus (l^\vee \cap \overline{\pi\Lambda}_R) = l \oplus N^\vee = U_0,
  \end{equation*}
  from which it follows that $U_0$ is Lagrangian. 

  We claim that $U_0$ is a direct summand of $V_R$. Indeed, since $l$ is a direct summand of $\mathcal{L}_R$ it is also a direct summand of $V_R = \mathcal{L}_R \oplus \pi\Lambda_R$. It follows that $N = l \oplus \overline{\pi\Lambda}_R$ is a direct summand of $V_R$, for example one can take as complement the complement of $l$ in $\mathcal{L}_R$. Let $Q\subset \mathcal{L}_R$ denote such a complement. Observe that since the alternating form is non-degenerate, we have that $V_R^{\vee} = \{0\}$. Since $Q$ is a complement of $N$, we have $V_R = N + Q$ and $\{0\} = N \cap Q$. By taking the duals of these equalities we obtain $\{0\} = N^\vee \cap Q^\vee$ and $V_R = N^\vee + Q^\vee$. It follows that $N^\vee$ is a direct summand of $V_R$ and $Q^\vee$ is a complement. Observe that since $\overline{\pi\Lambda}_R \subset N$ by taking duals we have $N^\vee \subset \overline{\pi\Lambda}_R  $. We want to show that $N^\vee$ is also a direct summand of $\overline{\pi\Lambda}_R$. Let $\pi\lambda \in \overline{\pi\Lambda}_R$. By the previous observation, there exist unique $n \in N^\vee$ and $q \in Q^\vee$ such that $\pi\lambda = n + q$. Since $N^\vee \subset \overline{\pi\Lambda}_R$ it follows that $q \in Q^\vee \cap \overline{\pi\Lambda}_R$. Then $Q^\vee \cap \overline{\pi\Lambda}_R$ is the complement of $N^\vee$ in $\overline{\pi\Lambda}_R$. Then $U_0 = l \oplus N^\vee$ is a direct summand of $V_R$, for example, we can take as a complement the submodule $Q + Q^\vee\cap \overline{\pi\Lambda}_R$.
  
  Let $U$ be a Lagrangian subspace of $V$ such that $q(U) = l$ and consider the linear map $\phi_U$ obtained as the composition of the canonical isomorphisms
  \begin{equation*}
    U_0/N^\vee = U_0/(U_0 \cap \overline{\pi\Lambda}_R) \xrightarrow{\sim} (U_0 + \overline{\pi\Lambda}_R) /\overline{\pi\Lambda}_R = l = (U + \overline{\pi\Lambda}_R) /\overline{\pi\Lambda}_R \xrightarrow{\sim}U/(U \cap \overline{\pi\Lambda}_R) = U/N^\vee.
  \end{equation*}
  This induces a morphism of submodules of $N/N^\vee$
  \begin{align*}
    \psi_U: U_0/N^\vee & \longrightarrow \overline{\pi\Lambda}_R/N^\vee \\
    u &\mapsto u -\phi_U(u).
  \end{align*}
  Observe that there is an $R$-linear map from the module $\overline{\pi\Lambda}_R/N^\vee$ into the dual space $l^* = \mathrm{Hom}_R(l, R)$ given by the alternating form
  \begin{align*}
    \overline{\pi\Lambda}_R/N^\vee &\longrightarrow l^*\\
    x &\mapsto \left ( \begin{aligned}
      l &\longrightarrow R \\
      v &\mapsto \langle v, x' \rangle
    \end{aligned} \right ),
  \end{align*}
  where $x' \in \overline{\pi\Lambda}_R$ is any lift of $x$. Indeed, we have already observed that $l^\vee \cap \overline{\pi\Lambda}_R = N^\vee$, from which it follows that the map above is well-defined and injective (in particular bijective when $R$ is a field). It follows that we can identify the map $\psi_U: l \cong U_0/N^\vee \rightarrow \overline{\pi\Lambda}_R/N^\vee \hookrightarrow l^*$ with an element of $\mathrm{Hom}_R(l, l^*)$. The assignment $U \mapsto (l = q(U), \psi_U: l \rightarrow l^*)$ gives the desired map of sets
  \[\begin{tikzcd}
    \mathcal{S}^{(1)}(R) \arrow[rr] \arrow[rd] &   & \mathcal{H}^1(R)\arrow[ld] \\
                            & \mathbb{P}(W)(R). &             
  \end{tikzcd}\]

  Since we have been working exclusively with projective, hence flat,  modules (as direct summands of free modules), all quotients considered above commute with base change to another $\F$-algebra $R \rightarrow R'$. It follows that the map $U \mapsto \psi_U$ commutes with base change, too, and therefore induces a morphism of $\F$-projective schemes.
\end{proof}

Our next step is to restrict the morphism constructed in the previous lemma to the closed subscheme $S_{V\pi}^{(1)}$ of $\mathcal{S}^{(1)}$. We have already seen in Section \ref{sec:closedpoints} that there is a bijection between the closed points of $S_{V\pi}^{(1)}$ with those of the restriction of the line bundle $\mathcal{H}^1$ to the variety $R_W \subset \mathbb{P}(W)$. Recall that the latter has been defined in \ref{eq:RV} and is the closure of some generalized Deligne-Lusztig variety for the orthogonal group. 

\begin{lem}\label{lem:SVisoH1}
 The morphism of Lemma \ref{lem:S1P1} induces an isomorphism from $S_{V \pi}^{(1)}$ to the restriction of $\mathcal{H}^1$ to the variety $R_W \subset \mathbb{P}(W)$ studied in Section \ref{sec:dlvso}. It follows that $S_{V \pi}^{(1)}$ is normal, irreducible and of dimension $4$.
\end{lem}
\begin{proof}
  As $\mathcal{S}^{(1)} \rightarrow \mathcal{H}^1$ is a morphism of projective schemes it is proper. We first show it is a monomorphism. Suppose $U_1, U_2 \in \mathcal{S}^{(1)}(R)$ are both sent by $q$ to $l \in \mathbb{P}(W_R)$, and that we also have $\psi_{U_1} = \psi_{U_2} \in \mathrm{Hom}_R(l, l^*)$. Then by definition of $\psi_{U_i}$, the quotients $U_i/N^\vee$ are the image under $\psi_{U_i}$ of the fixed element $U_0$ constructed in the proof of Lemma \ref{lem:S1P1}. It follows that $U_1/N^\vee = U_2/N^\vee$ as submodules of $N/N^\vee$. Since $U_1, U_2$ are both contained in the preimage $N = l + \overline{\pi\Lambda}_R$ of $l$ and are Lagrangian, they both contain the orthogonal $N^\vee$, hence $U_1 = U_2$. 
  
  The morphism $\mathcal{S}^{(1)} \rightarrow \mathcal{H}^1$ is then injective on $R$-points for any $\F$-algebra $R$. Therefore, it is a proper monomorphism, and by Zariski's main theorem it is a closed immersion, compare \cite{gw}*{Cor.\ 12.92, Prop.\ 12.94}. We restrict this closed immersion to the reduced closed subscheme $S_{V\pi}^{(1)} \hookrightarrow \mathcal{S}^{(1)}$. In the previous section, see the proof of Lemma \ref{lem:spW}, we have seen that this morphism induces a bijection between the $k$-valued points of $S_{V\pi}^{(1)}$ and those of the restriction of $\mathcal{H}^1$ to the variety $R_W$, for any algebraically closed field $k$. Since we are working with reduced schemes it follows that the closed immersion $S_{V\pi}^{(1)} \rightarrow \mathcal{H}^1_{\mid R_W}$ is actually an isomorphism. 
  
  The remaining properties follow from the corresponding statement for $R_W$, see Lemma \ref{lem:RVnormal}, and the fact that line bundles preserve normality and irreducibility, while they increase the dimension by one.
\end{proof}

The following result is an immediate consequence of the previous lemma, the definition of $\N_{\Lambda}^{(1)}$ and the fact that it is dense in $\N_{\Lambda}^{\le 1}$ by Lemma \ref{lem:dense}.

\begin{cor}
  $\N_{\Lambda}^{(1)}$ is universally homeomorphic to a locally trivial line bundle over $R_W$. It follows that its closure $\N_{\Lambda}^{\le 1}$ is irreducible and has dimension $4$.
\end{cor}

We are now ready to prove the first part of Theorem \ref{thm:intro}. 

\begin{prop}\label{prop:irred}
  Assume that the Hermitian form over $C$ is split. Then $\bar{\N}_{\rm red}^0$ has irreducible components of two types.
  \begin{itemize}
    \item[(i)]  For every maximal vertex lattice $\mathcal{L}$, there is an irreducible component $\N_{\mathcal{L}}$, which is universally homeomorphic to the generalized Deligne-Lusztig variety $S_V$ for the symplectic group $\Sp_6$ and has dimension $5$.
    \item[(ii)] For every $2$-modular lattice $\Lambda$, there is an irreducible component $\N^{\le 1}_{\Lambda}$. It contains the dense subscheme $\N^{(1)}_{\Lambda}$, which is universally homeomorphic to a locally trivial line bundle over the generalized Deligne-Lusztig variety $R_W$. These components have dimension $4$.
  \end{itemize}
\end{prop}
\begin{proof}
  We have seen in the previous section that for $k$ algebraically closed, if a lattice $M \in \V(k)$ is not contained in a vertex lattice, then $\pi \Lambda_k \subset^1 M + \pi\Lambda_k$ for some $2$-modular lattice $\Lambda$. Therefore, the union of the subsets $\N_{\mathcal{L}}$ for $\mathcal{L}$ running over the set of vertex lattices of maximal type, together with $\N^{\le 1}_{\Lambda}$ for $\Lambda$ running over the set of $2$-modular lattices, contains $\bar{\N}_{\rm red}^0$. Again we are using the fact that a reduced scheme over $\F$ is determined by its closed points.

  For a maximal vertex lattice $\mathcal{L}$, we have seen that the irreducible scheme $\N_{\mathcal{L}}$ contains the open and dense subscheme corresponding to the stratum $X_B(s_3s_2s_1) \sqcup X_B(s_3s_2s_3s_1) \sqcup X_B(s_3s_2s_3s_1s_2)$ in the decomposition (\ref{eq:strC}) of $S_V$. Its $k$-points correspond to those lattices $M \in \V(k)$ such that $\Lambda(M) = \mathcal{L}$ and which are therefore not contained in $\N_{\mathcal{L}'}$ for any other maximal vertex lattice $\mathcal{L}'$. Similarly, observe that the irreducible subscheme $\N^{(1)}_{\Lambda}$ contains an open and dense subscheme whose $k$-points are the lattices $M$ such that $\Lambda(M) = \Lambda$. This subscheme corresponds to the dense subvariety of the Deligne-Lusztig variety $Y_{a_0}$ introduced in the discussion of Remark \ref{rem:hereditary}. Its $k$-valued points are therefore not contained in any $\N^{\le 1}_{\Lambda'}$ for any other $2$-modular lattice $\Lambda'$ nor in $\N_{\mathcal{L}}$ for a maximal vertex lattice $\mathcal{L}$. 

  Since for any vertex lattice there is a $2$-modular lattice containing it, we need to check if for some vertex lattice $\mathcal{L}$ the corresponding component $\N_{\mathcal{L}}$ is contained in the union of the components $\N_{\Lambda}^{\le 1}$.  Since the dimension of any $\N_{\mathcal{L}}$ is $5$ and that of any $\N^{\le 1}_{\Lambda}$ is $4$, this is not possible, hence we can conclude that these are exactly the irreducible components of $\bar{\N}_{\rm red}^0$.
\end{proof}

The following result will be relevant in the next section for a comparison with the decomposition given by the set of admissible elements on the generalized affine Deligne-Lusztig variety $X(\mu, b)$ associated to our problem. Recall that we have proven in Lemma \ref{lem:SVstr} that the variety $S_V$ has a stratification in terms of varieties $S_{V'}$ for smaller dimensional symplectic vector spaces $V'$. Similarly, we have seen in Lemma \ref{lem:RVdim} that $R_W$ has a stratification in terms of the generalized Deligne-Lusztig varieties $Y_a$ of Definition \ref{def:ya}. In particular, since in this section $W$ is a split orthogonal space of dimension $6$, by Lemma \ref{lem:lus3} $R_W$ has only two strata $R_W = Y_{\infty} \sqcup Y_2$.

\begin{cor}\label{cor:str}
  The irreducible components of $\bar{\N}^0_{\mathrm{red}}$ are stratified as follows.
  \begin{itemize}
    \item[(i)] For $\mathcal{L}$ a vertex lattice in $C$ of type $6$, the corresponding irreducible component $\N_{\mathcal{L}}$ can be decomposed as 
    \begin{equation*}
      \N_{\mathcal{L}} = \bigsqcup_{\mathcal{L'} \subset \mathcal{L}} \N_{\mathcal{L'}}^{\circ},
    \end{equation*}
    where the union runs over the vertex lattices of $C$ contained in $\mathcal{L}$ and each $\N_{\mathcal{L'}}^\circ$ is universally homeomorphic to the generalized Deligne-Lusztig variety $S_{V'}$, with $V'$ the symplectic vector space $\mathcal{L}'/\mathcal{L}'^\vee$. The strata are then given by the union over the vertex lattices of a fixed type and the closure of each stratum is given by the strata corresponding to smaller type.
    \item[(ii)] For $\Lambda$ a $2$-modular lattice in $C$, the corresponding irreducible component $\N_{\Lambda}^{\le 1}$ can be decomposed as \begin{equation*}
      \N_{\Lambda}^{\le 1} = \N_{\Lambda}^{(0)} \sqcup \N_{\Lambda, \infty} \sqcup \N_{\Lambda, 2},
    \end{equation*} and the closure of each stratum is the union of the strata preceding it. Here, $\N_{\Lambda}^{(0)}$ is defined in an analogous way as $\N_{\Lambda}^{(1)}$ and its only point is the $p$-divisible group $X_{\pi\Lambda}$ associated to the lattice $\pi\Lambda \otimes_{\mathcal{O}_E} W(\F)$ contained in $N$. The other two strata are universally homeomorphic to the restriction of the line bundle over $R_W$ to its strata $R_W = Y_{\infty} \sqcup Y_2$. In particular the closed subscheme $\N_{\Lambda}^{(0)} \sqcup \N_{\Lambda, \infty}$ is contained in the union of the irreducible components $\N_{\mathcal{L}}$, for all vertex lattices $\mathcal{L} \subset \Lambda$.
  \end{itemize}
\end{cor}
\begin{proof}
  The first statement follows from the universal homeomorphism $f: \N_{\mathcal{L}} \longrightarrow S_{V}$ and the stratification of $S_V$ proved in Lemma \ref{lem:SVstr}. Recall that the irreducible components of the smaller dimensional strata of $S_V$ are indexed over the isotropic subspaces $U$ of $V = \mathcal{L}/\mathcal{L}^\vee$, and that the component corresponding to $U$ is again a generalized Deligne-Lusztig variety $S_{V'}$ for $V' = U^\vee/U$. One can then observe that the isotropic subspaces of $V$ are in bijection with the vertex lattices $\mathcal{L'}$ of $C$ that are contained in $\mathcal{L}$.

  For a $2$-modular lattice $\Lambda$, it is clear that $\N_{\Lambda}^{\le 1} = \N_{\Lambda}^{(0)} \sqcup \N_{\Lambda}^{(1)}$. In particular, $\N_{\Lambda}^{(0)}$ is the preimage under $f$ of $S_{V\pi}^{(0)}$, the closed subscheme of $S_{V\pi}^{\le 1}$ consisting of Lagrangian submodules $U$ such that the rank of $U + \overline{\pi\Lambda}_R$ is $6$, which is equivalent to $U = \overline{\pi\Lambda}_R$. Observe that $\pi\Lambda \otimes W(\F)$ is a $\tau$-stable, self-dual lattice, hence it belongs to $\N_{\mathcal{L}}(\F)$ for some vertex lattice of type $6$ contained in $\Lambda$, and it corresponds to a $p$-divisible group $X_{\pi\Lambda} \in \N_{\Lambda}(\F)$. 
  
  The open and dense subscheme $\N_{\Lambda}^{(1)}$ is universally homeomorphic by the previous proposition to a line bundle over $R_W$. Then the stratification follows from the decomposition of $R_W$ given in Lemma \ref{lem:RVdim}. We have seen in Lemma \ref{lem:splitV1} that the closed points of $\N_{\Lambda}^{(1)}$ that are mapped by $q$ into $Y_{\infty}$ correspond to lattices $M$ such that $\Lambda(M)$ is a vertex lattice, from which the last statement follows.
\end{proof}

\begin{rem}
  If we compare the previous proposition with the analogous results \cite{rtw}*{Prop.\ 6.6} for signature $(1, n-1)$, we see that irreducible components homeomorphic to Deligne-Lusztig varieties for the symplectic group appear in both cases. However, the existence of a second family of irreducible components, those homeomorphic to the line bundle, is a new specific feature of signature $(2,4)$. 
\end{rem}
\begin{rem}
  Another difference to signature $(1, n-1)$ is that the intersection pattern is now quite hard to describe in terms of the Bruhat-Tits building. For example, even if we know there is a point $M \in \N^{(1)}_{\Lambda_1}(k) \cap \N^{(1)}_{\Lambda_2}(k)$, for two distinct $2$-modular lattices, it is in general not true that the whole fiber over the image of $M$ in $R_W$ is contained in the intersection. Indeed, this would be the case if and only if $\pi \Lambda_1 \subset \Lambda_2$, which is not true in general. On the other hand, for vertex lattices, the intersection pattern can be described also in our case in terms of the Bruhat-Tits building for $\mathrm{SU}(C)(\Q_p)$ by Proposition \ref{prop:vlproperties} and the previous corollary. 
  
  The intersection of two components corresponding to different types of lattices is also not easy to describe. Recall the decomposition of the closed subvariety $Y_{\infty}$ of $R_W$ given in Remark \ref{rem:infty}. For dimension $6$ and split symmetric form, $Y_{\infty}$ can be decomposed as a union of three strata
  \begin{equation*}
    Y_{\infty} = X_{P_1}(1) \sqcup X_{P_2}(t_1) \sqcup X_B(t_1t_2).
  \end{equation*}
  As we have seen in Lemma \ref{lem:RVdim} the closed points of $X_B(t_1t_2)$ are those $l \in \mathbb{P}(W)(k)$ such that $l + \Phi(l) + \Phi^2(l)$ is $\Phi$-stable and has dimension $3$. By the discussions in the previous chapter, see in particular the proof of Lemma \ref{lem:splitV1}, these points correspond to lattices $M \in \N_{\Lambda}^{(1)}(k)$ such that $\Lambda(M)$ is a vertex lattice of type $6$. This means that for some vertex lattice $\mathcal{L}$ of type $6$, the intersection $\N_{\Lambda}^{\le 1} \cap \N_{\mathcal{L}}^\circ$ is non-empty, but by dimension reasons it does not contain the whole stratum $\N_{\mathcal{L}}^\circ$. Similarly, the subvariety of $Y_{\infty}$ that in Lemma \ref{lem:RVdim} is identified with $X_{P_2}(t_1)$ corresponds to those lattices $M$ such that $\Lambda(M)$ is a vertex lattice of type $4$. Therefore, for some vertex lattices $\mathcal{L}$ of type $4$ the intersection of the stratum $\N_{\mathcal{L}}^\circ$ with $\N_{\Lambda}^{(1)}$ is non-empty. In particular, this intersection is contained in a subscheme of $\N_{\Lambda \infty}$ universally homeomorphic to the restriction of the line bundle $\mathcal{H}^1$ to the subvariety $X_{P_2}(t_1)$ of $Y_{\infty}$. This subscheme has then dimension $2$, while the stratum $\N_{\mathcal{L}}^\circ$ has dimension $3$, compare Lemma \ref{lem:SVstr}, and therefore is not contained in $\N_{\Lambda}^{(1)}$.
  On the other hand, we have seen in the proof of Lemma \ref{lem:dense} that if $\mathcal{L}$ is a vertex lattice of type $2$ such that $\pi\Lambda \subset \mathcal{L}$ then $\N_{\mathcal{L}}$ is contained in $\N_{\Lambda}^{\le 1}$. 
\end{rem}

\begin{rem}
  As we have already observed, for a vertex lattice $\mathcal{L}$, the stratum $\N^{\circ}_{\mathcal{L}}$ corresponding to lattices $M$ such that $\Lambda(M) = \mathcal{L}$ is open and dense in $\N_{\mathcal{L}}$. It is interesting to notice that for a $2$-modular lattice $\Lambda$, the subscheme $\N^{\circ}_{\Lambda}$ is open and dense in $\N^{(1)}_{\Lambda}$, but it is not dense in the whole $\N_{\Lambda}$, as $\N_{\Lambda}^{\le 1}$ is closed in $\N_{\Lambda}$. This, together with the fact that the intersection pattern is harder to describe, and that the stratification of Corollary \ref{cor:str} of the irreducible components do not extend to a stratification of the whole $\bar{\N}^0_{\mathrm{red}}$ in terms of Deligne-Lusztig varieties, can all be seen as consequences of the fact that the underlying group-theoretical datum is not fully Hodge-Newton decomposable. 
\end{rem}

\subsection{Irreducible components in the non-split case} 
Assume now that the Hermitian form on $C$ is non-split. We have seen in the previous section, compare Lemma \ref{lem:spW2}, that any lattice $M$ in $\V(k)$ is contained in some $2$-modular lattice $\Lambda_k$ such that $\pi\Lambda_k \subset^{\le 2} M + \pi\Lambda_k$.  We are going to see in this section that there are two families of irreducible components of $\bar{\N}_{\mathrm{red}}^0$ and this time both are indexed over the set of $2$-modular lattices. Roughly speaking, these components are characterized by the index of the inclusion $\pi\Lambda_k \subset^{\le 2} M + \pi\Lambda_k$. Again our strategy is to use the universal homeomorphism $f: \N_\Lambda \longrightarrow S_{V\pi}$ and the results on closed points of the previous section to describe the irreducible components of $\bar{\N}^0_{\rm red}$ in terms of Deligne-Lusztig varieties. 

Observe that for a $2$-modular lattice $\Lambda \subset C$ the corresponding subscheme $\N_{\Lambda}$ contains the closed subscheme $\N_{\Lambda}^{\le 1}$ defined in the same way as in the split case. Observe that the proofs of Lemma \ref{lem:SVisoH1} and its corollary do not make use of the fact that the Hermitian form over $C$ is split, therefore one can show in the same way the following result.

\begin{lem}\label{lem:homeom}
 Let $\N_{\Lambda}^{(1)}$ be the preimage under the universal homeomorphism $f$ of the locally closed subscheme $S_{V\pi}^{(1)}$ defined as in the split case. Then $\N_{\Lambda}^{(1)}$ is universally homeomorphic to the restriction of the line bundle $\mathcal{H}^1$ to the variety $R_W$. It follows that its closure $\N_{\Lambda}^{\le 1}$ is irreducible and has dimension $4$.
\end{lem}

Similarly to the split case, we also consider the open subscheme $S_{V\pi}^{(2)}$ of $S_{V\pi}$ given by 
\begin{equation}
  S_{V \pi}^{(2)}(R) \coloneqq \{ U \in S_{V \pi}(R) \mid U + \overline{\pi\Lambda}_R \text{ is a direct summand of } V_R \text{ with } \mathrm{rk}(\overline{\pi\Lambda}_R + U ) = 8\}. 
\end{equation}
We denote by $\N_{\Lambda}^{(2)}$ its preimage  under the universal homeomorphism $f$. It is again an open subscheme of $\N_{\Lambda}$. Recall that by Lemma \ref{lem:spW2} all lattices in $\N_{\Lambda}(k)$ are such that $\pi\Lambda_k \subset^{\le 2} M + \pi\Lambda_k$, and since $f$ is a bijection on closed points, we have that all Lagrangian subspaces $U$ in $S_{V\pi}(k)$ satisfy $\dim(U + \overline{\pi\Lambda}_k) \le 8$. 
However, as we are going to see in Proposition \ref{prop:nonsplit}, $\N_{\Lambda}^{(2)}$ is not dense in $\N_{\Lambda}$, and its closure will be one type of irreducible components of this scheme. The other irreducible component will be $\N_{\Lambda}^{\le 1}$.

Again, let $W$ be the six-dimensional $\F$-vector space given by $\Lambda/\pi\Lambda$. Recall that it is endowed with a non-split symmetric form. In Section \ref{sec:dlvso} we have studied the variety $Q_W \subset \Grass_2(W)$. Recall that it is the closure of some generalized Deligne-Lusztig variety for the non-split orthogonal group of rank $6$. Moreover, we have proven in Lemma \ref{lem:QVstrata} that there is a stratification $Q_W = Z_0 \sqcup Z_1 \sqcup Z_2$. Since the form is non-split, the open dense subvariety $Z_2$ is isomorphic to the union $X_B(t_2t_1) \sqcup X_{B}(\Phi(t_2t_1))$ and has therefore two irreducible components.

\begin{lem}\label{lem:morphism}
  The map $S_{V\pi}^{(2)} \rightarrow \Grass_2(W)$ induced by $q: V \rightarrow W = V/\overline{\pi\Lambda}$ is a morphism of projective schemes. It sends $S_{V\pi}^{(2)}$ to the projective scheme $Q_W$ of Section \ref{sec:dlvso}.  
\end{lem}
\begin{proof}
  The fact that the map induced by $q$ is a morphism of projective schemes is proved in the same way as in Lemma \ref{lem:S1P1}. In order to find the image of $S_{V\pi}^{(2)}$ under this map, it is enough to consider its closed points. Then the statement follows from Lemma \ref{lem:spW2}. 
\end{proof}

Denote by $S_{V\pi}^{(2)^\circ}$ the preimage of the open subscheme $Z_2 \subset Q_W$ under the morphism of Lemma \ref{lem:morphism}. As in the split case, our next goal is to construct a morphism from $S_{V\pi}^{(2)^\circ}$ to a vector bundle over $Z_2$. We have seen in Remark \ref{rem:flag} that there is a morphism $Z_2 \cong X_{B}(t_2t_1) \sqcup X_B(t_3t_1)\rightarrow \mathcal{F}l(W)$, where $\mathcal{F}l(W)$ is the partial flag variety parametrizing flags of the form $U_1 \subset U_2$ with $\dim(U_i) = i$. Consider now the maps $\pi_i: \mathcal{F}l(W) \rightarrow \Grass_i(W)$ sending a flag to its term of dimension $i$. We denote by $\mathcal{U}_i$ the pullback of the universal vector bundle on $\Grass_i(W)$ along the map $\pi_i$. Then we consider $\mathcal{H}^2$ the rank-$2$, locally trivial vector bundle on $\mathcal{F}l(W)$ obtained as the homomorphism bundle 
\begin{equation}
  \mathcal{H}^2 \coloneqq \Hom(\mathcal{U}_1, \mathcal{U}_2^*).
\end{equation}

\begin{lem}\label{lem:S2homeomH2}
  The morphism $g: S_{V\pi}^{(2)^\circ} \xlongrightarrow{q} Z_2 \rightarrow  \mathcal{F}l(W)$  induces a universal homeomorphism from $S_{V\pi}^{(2)^\circ}$ to the pullback along $Z_2 \rightarrow  \mathcal{F}l(W)$ of the rank-$2$ vector bundle $\mathcal{H}^2$.
\end{lem}
\begin{proof}
  As in the proof of Lemma \ref{lem:SVisoH1}, we start with defining maps $S_{V\pi}^{(2)^\circ}(R)\rightarrow \mathcal{H}^2(R)$ for any $\F$-algebra $R$. Let $U \in S_{V\pi}^{(2)^\circ}(R)$ and denote by $l \subset T$ its image in $\mathcal{F}l(W)$. Observe that by definition of $g$ above we have that $T = q(U)$, where $q: V_R \rightarrow W_R = V_R/\overline{\pi\Lambda}_R$. We fix again a Lagrangian complement $\mathcal{L}$ of $\overline{\pi\Lambda}$ in $V$, and we identify $W$ and $\mathcal{L}$, so that we can consider the image of $U$ as a flag $l \subset T \subset \mathcal{L}$. 
  
  Consider $N = T \oplus \overline{\pi\Lambda}_R$ and its orthogonal $N^\vee = T^\vee \cap \overline{\pi\Lambda}_R$. Then, as in the proof of Lemma \ref{lem:S1P1} one shows that the submodule $U_0 = T \oplus N^\vee$ is a Lagrangian direct summand of $V_R$, and it is sent by $q$ to $T \in Z_2$ and therefore by $g$ to $l \subset T$. Observe again that since $U$ is Lagrangian and it is contained in $N$, we have $N^\vee \subset U \subset N$. It follows that the preimage under $q$ of $l$ in $U$ is a submodule of the form $l_U \oplus N^\vee$ for a rank-one submodule $l_U \subset U$ such that $l_U \subset l \oplus \overline{\pi\Lambda}_R$. Since $N = T \oplus \overline{\pi\Lambda}_R = U + \overline{\pi\Lambda}_R$ it follows that $N^\vee = U \cap \overline{\pi\Lambda}_R$, as $U$ and $\overline{\pi\Lambda}_R$ are Lagrangian direct summands. Therefore, the intersection of $U$ with the preimage in $V_R$ of $l$ satisfies $U \cap (l \oplus \overline{\pi\Lambda}_R) = l_U \oplus N^\vee$, and is then again a direct summand of $V$. 
  We have an isomorphism of submodules of $N/N^\vee$ given by the second isomorphism theorem
  \begin{equation*}
    \phi_U:  (l \oplus N^\vee)/N^\vee \cong l \cong (l_U \oplus N^\vee)/N^\vee = (U \cap (l \oplus \overline{\pi\Lambda}_R))/N^\vee
  \end{equation*} which gives again a morphism of submodules of $N/N^\vee$
  \begin{align*}
    \psi_U: l &\longrightarrow \overline{\pi\Lambda}_R/N^\vee\\
    v &\mapsto v - \phi_U(v)
  \end{align*}
  By the same arguments as in the proof of Lemma \ref{lem:S1P1}, there is an injective morphism 
  \begin{align*}
    \overline{\pi\Lambda}_R/N^\vee &\longrightarrow T^*\\
    x &\mapsto \left ( 
    \begin{aligned}
      T &\longrightarrow R \\
      v &\mapsto \langle v, x' \rangle
    \end{aligned} \right ), 
  \end{align*}
  which is an isomorphism when $R$ is a field. Again, we have been working only with direct summands of the free module $V_R$, hence with projective modules. Therefore, taking quotients and intersections commutes with tensoring with another $\F$-algebra $R \rightarrow R'$. It follows that the map $S_{V\pi}^{(2)^\circ}(R) \longrightarrow \mathcal{H}^2(R)$ sending a Lagrangian $U$ to $ (g(U) = l \subset T, \psi_U: l \rightarrow T^{*})$ commutes with tensor product by $\F$-algebras and hence it gives a morphism of projective $\F$-schemes $S_{V\pi}^{(2)^\circ} \rightarrow \mathcal{H}^2$. We have so constructed a morphism making the diagram 
  \[\begin{tikzcd}
    S_{V\pi}^{(2)^\circ} \arrow[r] \arrow[d] & \mathcal{H}^2 \arrow[d] \\
    Z_2 \arrow[r]                                  & \mathcal{F}l(W)          
  \end{tikzcd}\]
  commute. It follows that there is a morphism $S_{V\pi}^{(2)^\circ}$ to the rank-$2$ vector bundle $\mathcal{H}^2_{\mid Z_2}$ obtained as the pullback $\mathcal{H}^2 \times_{\mathcal{F}l(W)} Z_2$. Last, as we have already seen, to prove that a projective morphism is a universal homeomorphism it suffices to show that it induces a bijection on the sets of $k$-valued points, for any algebraically closed field $k$. Then we can conclude with Lemma \ref{lem:spW2}.
\end{proof} 

\begin{rem}
  Observe that the morphism above is a universal homeomorphism but not an isomorphism. Indeed, the construction of Lemma \ref{lem:spW2} giving the bijection on closed points, involves taking the Frobenius, and therefore cannot be extended to a generic $\F$-algebra $R$. 
\end{rem}

We denote by $\N_{\Lambda}^{(2)^\circ}$ the open subscheme of $\N_{\Lambda}^{(2)}$ defined as the preimage of $S_{V\pi}^{(2)^\circ}$ under the universal homeomorphism $f$. Observe that by the previous lemma, $\N_{\Lambda}^{(2)^\circ}$ has two irreducible components corresponding to the two irreducible components of $Z_2 \cong X_{B}(t_2t_1) \cup X_{B}(t_3t_1)$. We can now conclude the proof of Theorem \ref{thm:intro}.

\begin{prop}\label{prop:nonsplit}
  Assume the Hermitian form over $C$ is non-split. Then $\bar{\N}^0_{\mathrm{red}}$ has irreducible components of two types.
  \begin{itemize}
    \item[(i)] For every $2$-modular lattice $\Lambda$ there is an irreducible component $\N_{\Lambda}^{\le 1}$. It contains the dense subscheme $\N_{\Lambda}^{(1)}$, which is universally homeomorphic to a locally trivial line bundle over the generalized Deligne-Lusztig variety $R_W$. 
    \item[(ii)] For every $2$-modular lattice $\Lambda$ there are two  irreducible components contained in  $\overline{\N_{\Lambda}^{(2)^\circ}}$. Each of them is the closure of one of the irreducible components of the open subscheme $\N_{\Lambda}^{(2)^\circ}$ and is universally homeomorphic to a rank-$2$ vector bundle over the classical Deligne-Lusztig variety $X_B(t_2t_1)$, respectively $X_B(\Phi(t_2t_1)) = X_B(t_3t_1)$.
  \end{itemize}
  It follows that $\bar{\N}^0_{\mathrm{red}}$ is pure of dimension $4$.
\end{prop}
\begin{proof}
  By Lemma \ref{lem:homeom} and by Lemma \ref{lem:S2homeomH2} we know that $\N_{\Lambda}^{\le 1}$ and the two components of the closure of $\N_\Lambda^{(2)^{\circ}}$ are irreducible and have dimension $4$. From this it also follows that $\N_{\Lambda}^{\le 1}$ is not contained in the closure of $\N_\Lambda^{(2)^{\circ}}$. Moreover, we have seen in Lemma \ref{lem:spW2b} that a point $M \in \N_{\Lambda}(k)$ is either contained in $\N_{\Lambda}^{(2)^\circ}(k)$ or there exists another $\Lambda'$ such that $M \in \N_{\Lambda'}^{\le 1}(k)$. Again, since we are working with reduced schemes over the algebraically closed field $\F$, this implies that $\bar{\N}^0_{\mathrm{red}}$ is contained in the union $\bigcup_\Lambda \N_{\Lambda}^{\le 1} \sqcup \bigsqcup_{\Lambda}\N_{\Lambda}^{(2)^\circ}$ running over all $2$-modular lattices $\Lambda$.

  Last, observe that each subscheme $\N_{\Lambda}^{(1)}$, respectively $\N_\Lambda^{(2)^\circ}$, contains the preimage of the open and dense subset of $R_W$, respectively $Z_2$ mentioned in Remark \ref{rem:hereditary}. In other words, it contains the open subscheme whose $k$-valued points corresponds to lattices $M$ such that $\Lambda(M) = \Lambda$ and therefore are not contained in any other subscheme $\N_{\Lambda'}$. 
\end{proof}

\begin{rem}\label{rem:notiso}
  Observe that in the proof of Lemma \ref{lem:closedimm}, which is the key ingredient for the proof in the split case of Theorem \ref{thm:intro}, we adopt the same strategy as \cite{rtw}*{Sec.\ 6} to construct a map from the Rapoport-Zink space into the Grassmannian variety $\Grass(V)$. It is stated in \textit{loc.cit.} that this gives an isomorphism between $\bar{\N}^0_{\mathrm{red}}$ and the closed subvariety $S_V$ of the Grassmannian. The proof of \cite{rtw}*{Prop.\ 6.7} relies on the previous result \cite{vw}*{Thm.\ 4.8}, which however is only true up to a Frobenius twist, as noted by R.\ Chen\footnote[1]{Private communication with R.\ Chen, M.\ Rapoport and T.\ Wedhorn}. It seems that the Frobenius twist does not really affect the map in the ramified case, so that one still expects to have an isomorphism. This is yet still open and probably requires a careful analysis of the corresponding Zink's windows for displays. On the other hand, our construction of the homeomorphism of Lemma \ref{lem:S2homeomH2}, on which the proof of Theorem \ref{thm:intro} for the non-split case is based, involves the relative Frobenius morphism and hence is not an isomorphism. 
\end{rem}
\begin{rem}
  The following observations will be relevant in the next section for a comparison with the decomposition given by the set of admissible elements on the generalized affine Deligne-Lusztig variety $X(\mu, b)$. Recall the stratification of $R_W$ given in Lemma \ref{lem:RVdim}. In the non-split case there are three strata $R_W = Y_{\infty} \sqcup Y_3 \sqcup Y_2$. It follows from Lemma \ref{lem:homeom} that $\N_{\Lambda}^{(1)}$ has a stratification
  \begin{equation*}
    \N_{\Lambda}^{(1)} = \N_{\Lambda, \infty} \sqcup \N_{\Lambda, 3} \sqcup \N_{\Lambda, 2},
  \end{equation*}
  where each stratum is universally homeomorphic to a line bundle over the corresponding stratum of $R_W$. Moreover, the closure of each stratum is the union of the ones preceding it. 
  
  Consider a vertex lattice $\mathcal{L} \subset \Lambda$. Since the form is non-split, $\mathcal{L}$ has type at most $4$. By Proposition \ref{prop:vl} the corresponding closed subscheme $\N_\mathcal{L}$ is universally homeomorphic to the generalized Deligne-Lusztig variety $S_V$ for the symplectic group of rank $\le 4$. Moreover, it has a stratification in terms of vertex lattices of smaller type as we have already seen in Corollary \ref{cor:str}. Recall that the $k$-valued points of $\N_{\Lambda,\infty}$ correspond to lattices $M$ such that $\Lambda(M)$ is a vertex lattice and $\pi\Lambda_k \subset^{\le 1} M + \pi\Lambda_k$. 
  
  We show that for every vertex lattice $\mathcal{L}$ there is a $2$-modular lattice $\Lambda$ such that $\N_{\mathcal{L}}^\circ \subset \N_{\Lambda}^{\le 1}$. If $\Lambda(M)$ has type $0$, then by Proposition \ref{prop:2vl} $\pi^{-1}\Lambda(M)$ is a $2$-modular lattice and by Lemma \ref{lem:dense} $M$ belongs to $\N_{\pi^{-1}\Lambda(M)}^{\le 1}$. We have also seen in the proof of Lemma \ref{lem:dense}, that if $\mathcal{L}$ has type $2$ and contains $\pi\Lambda$, then $\N_\mathcal{L}^\circ \subset \N_{\Lambda}^{(1)}$. Observe that by the correspondence between the complex of vertex lattices $\mathscr{L}$ and the Bruhat-Tits building for $\mathrm{SU}(C)(\Q_p)$, for every vertex lattice of type $2$, there is self-dual lattice contained in it, which is equivalent by Proposition \ref{prop:2vl} to the existence of a $2$-modular lattice $\Lambda$ such that $\pi\Lambda \subset \mathcal{L}$. Therefore, $\N_{\mathcal{L}}^\circ \subset \N_{\Lambda,\infty} \subset \N_{\Lambda}^{(1)}$ for a suitable $2$-modular lattice $\Lambda$. Last, arguing as in the proof of the first part of Lemma \ref{lem:spW2b} we can see that if $\mathcal{L}$ is a vertex lattice of type $4$ there is a $2$-modular lattice $\Lambda$ such that $\N_{\mathcal{L}} \subset \N_{\Lambda,\infty} \subset \N_{\Lambda}^{(1)}$.

  Let $\mathtt{V}_d$ and $\mathtt{M}$ denote respectively the set of vertex lattices of type $d$ and of $2$-modular lattices in $C$. Combining the previous observations we obtain a decomposition
  \begin{equation}\label{eq:strata1}
   \bigsqcup_{\Lambda \in \mathtt{M}} \N_{\Lambda}^{\le 1} = \bigsqcup_{\mathcal{L} \in \mathtt{V_0}} \N_{\mathcal{L}}^\circ \sqcup \bigsqcup_{\mathcal{L} \in \mathtt{V_2}} \N_{\mathcal{L}}^\circ \sqcup \bigsqcup_{\mathcal{L} \in \mathtt{V_4}} \N_{\mathcal{L}}^\circ \sqcup \bigsqcup_{\Lambda \in \mathtt{M}} \N_{\Lambda,3} \sqcup \bigsqcup_{\Lambda \in \mathtt{M}} \N_{\Lambda,2}.
  \end{equation}
  Moreover, by the previous discussion, the decomposition on the right is actually a stratification where the closure of each stratum is the union of the strata preceding it.
\end{rem}  

\begin{rem}\label{rem:last}
  We also have a decomposition of $\bar{\N}^0_{\mathrm{red}}$.
  \begin{equation}\label{eq:strata}
    \bar{\N}^0_{\mathrm{red}} = \bigsqcup_{\mathcal{L} \in \mathtt{V_0}} \N_{\mathcal{L}}^\circ \sqcup \bigsqcup_{\mathcal{L} \in \mathtt{V_2}} \N_{\mathcal{L}}^\circ \sqcup \bigsqcup_{\mathcal{L} \in \mathtt{V_4}} \N_{\mathcal{L}}^\circ \sqcup \bigsqcup_{\Lambda \in \mathtt{M}} \N_{\Lambda,3} \sqcup \bigsqcup_{\Lambda \in \mathtt{M}} \N_{\Lambda,2} \sqcup \bigsqcup_{\Lambda \in \mathtt{M}} \N_{\Lambda}^{(2)^\circ}.
  \end{equation}
  It is enough to check that the $k$-valued points of $\bar{\N}^0_{\mathrm{red}}$, for $k$ an algebraically closed field, are all contained in the union on the right. We have seen in Lemma \ref{lem:spW2} that every lattice $M \in \bar{\N}^0(k) = \bar{\N}^0_{\mathrm{red}}(k)$ is contained in a $2$-modular lattice $\Lambda_k$ and $\pi\Lambda_k \subset^{\le2} M + \pi\Lambda_k$. If $M \in \N_{\Lambda}^{(2)}$, if $M$ does not belong to $\N_{\Lambda}^{(2)^\circ}$ it either belongs to $\N_{\mathcal{L}}$ for some vertex lattice $\mathcal{L}$ or to $\N_{\Lambda'}^{(1)}$ for another $2$-modular lattice $\Lambda'$, see Lemma \ref{lem:spW2b}. If $M \in \N_{\Lambda,\infty}$, then by the same argument as in Corollary \ref{cor:str} there is a vertex lattice $\mathcal{L} \subset \Lambda$ such that $M \in \N_{\mathcal{L}}$. It would be interesting to give a description of the closure of $\N_{\Lambda}^{(2)^\circ}$ and hence to prove or disprove that (\ref{eq:strata}) is a stratification. This is also tightly related to the problem of describing the intersection pattern between components of type $\N_{\Lambda}^{(2)^\circ}$ and components of type $\N_{\Lambda}^{\le 1}$.
\end{rem}

\section{Affine Deligne-Lusztig varieties}\label{sec:adlvs}
\subsection{Reminder on affine Deligne-Lusztig varieties}

Affine Deligne-Lusztig varieties were first introduced in \cite{rap}, in this section we collect some definitions and results before we present the group-theoretical datum associated to our problem. We follow the exposition in \cite{gh_cox}*{Sec.\ 2} and \cite{he_adlv}*{Sec.\ 1-3}, and refer there for further details.

Let $F$ be a non-Archimedean local field, and denote by $\breve{F}$ the completion of its maximal unramified extension in a fixed algebraic closure $\bar{F}$. The field  $F$ can have the same characteristic as its residue field, in which case it is a field of formal Laurent power series $F = \F_q (\!(t)\!)$, or it can have characteristic zero, \textit{i.e.} it is a finite extension of $\mathbb{Q}_p$. 

Fix a connected reductive group $G$ over $F$. We denote by $\sigma$ both the Frobenius map on $\breve{F}$ and the map induced on $G(\breve{F})$. Let $I$ be a $\sigma$-invariant Iwahori subgroup of $G(\breve{F})$. Let $T$ be a maximal torus in $G$ such that the alcove corresponding to $I$ lies in the apartment of the Bruhat-Tits building attached to $T$. To this data we attach the extended affine Weyl group $\widetilde{W} = N_T(\breve{F})/T(\breve{F})\cap I$, where $N_T$ is the normalizer of $T$ in $G$. In the following we often write $w \in \widetilde{W}$ for both an element in the extended affine Weyl group and a representative in $N_T(\breve{F})$. 

Recall that fixing a special vertex in the base alcove gives a decomposition $\widetilde{W} = X_*(T)_{\Gamma} \rtimes W_0$, where $W_0 = N_T(\breve{F})/T(\breve{F})$ is the finite Weyl group, $X_*(T)$ is the coweight lattice of $T$, and $\Gamma$ denotes the Galois group $\Gal(\bar{F}/F^{\mathrm{un}})$. For a cocharacter $\mu^\vee$ we denote by $t^{\mu^\vee}$ the corresponding element in the extended affine Weyl group.
The choice of the base alcove determines also a set $\widetilde{\mathbb{S}}$ of simple affine reflections generating the affine Weyl group $W_a \subset \widetilde{W}$. Both $\widetilde{W}$ and $\widetilde{\mathbb{S}}$ are equipped with an action of $\sigma$. 

Denote by $\Omega$ the set of elements of $\widetilde{W}$ normalizing the base alcove. Recall that the affine Weyl group $W_a$ is an infinite Coxeter group. There is a decomposition $\widetilde{W} = W_{a} \rtimes \Omega$, which allows us to extend to $\widetilde{W}$ the notion of length on $W_a$ by setting to zero the length of any element in $\Omega$. Similarly, the Bruhat order can be extended from $W_a$ to $\widetilde{W}$ by setting $w\tau \le w'\tau'$ if and only if $\tau = \tau' \in \Omega$ and $w \le w'$ in  $W_a$. For any subset $J \subset \widetilde{\mathbb{S}}$ we denote by $W_J$ the subgroup of $W_a$ generated by the reflections in $J$ and by $^J\widetilde{W}$ the set of minimal length representatives for the cosets $W_{J} \backslash \widetilde{W}$. 

Two elements $b,b' \in G(\breve{F})$ are $\sigma$-conjugate if there exists $g \in G(\breve{F})$ such that $b = g^{-1}b'\sigma(g)$. Denote by $B(G)$ the set of $\sigma$-conjugacy classes in $G(\breve{F})$. A class $[b] \in B(G)$ is completely determined by its Newton point $\nu_{b}\in X_{*}(T)^{\Gamma}_{\Q, \mathrm{dom}}$, and its image under the Kottwitz map $\kappa: B(G) \rightarrow \pi_1(G)_{\Gamma}$, compare \cites{kot, kot1} and \cite{rr}. Here the fundamental group $\pi_1(G)$ is defined as the quotient of  $X_*(T)$ by the coroot lattice.

We consider the restriction of $\sigma$-conjugation to $\widetilde{W}$ and study the set of conjugacy classes $B(\widetilde{W})$ with respect to this restricted action. It is proved in \cite{ghkr} and \cite{he_geohom}*{Sec.\ 3} that the inclusion $N_T \hookrightarrow G$ gives a surjection from the set $B(\widetilde{W})$ of $\sigma$-conjugacy classes of $\widetilde{W}$ to the set of $\sigma$-conjugacy classes $B(G)$. This map becomes a bijection if we restrict it to classes in $B(\widetilde{W})$ containing a $\sigma$-straight element of $\widetilde{W}$, compare \cite{hn_minii}*{Thm.\ 3.3}. Recall that an element $ w \in \widetilde{W}$ is said to be $\sigma$-straight if it satisfies $\ell((w\sigma)^n) = n \ell(w)$ for all integers $n$. By \cite{heminlength}*{Lem.\ 1.1}, this is equivalent to $\ell(w) = \langle \nu_w, 2\rho \rangle$, where $\rho$ denotes half the sum of all positive roots and $\nu_w$ is the Newton point of $w$. An example of $\sigma$-straight elements is given by $\sigma$-Coxeter elements, that are elements in $\widetilde{W}$ given by the product of one reflection for each $\sigma$-orbit in $\widetilde{\mathbb{S}}$, compare \cite{hn_minii}*{Prop.\ 3.1}.

For any $w \in \widetilde{W}$ there is an integer $n$ such that
\begin{equation}\label{eq:newton}
  (w\sigma)^n = w\sigma(w)\dotsm\sigma^n(w) = t^{\mu^\vee}
\end{equation} 
for some cocharacter $\mu^{\vee}$. Then the Newton point of $w$ is the unique dominant element $\nu_w \in X_{*}(T)\otimes \Q$ lying in the $W_0$-orbit of $\frac{1}{n}\mu^{\vee}$. One can see that this does not depend on the choice of the exponent $n$. Moreover, by \cite{ric}*{Lem\ 1.2} there is a bijection between the fundamental group $\pi_1(G)_{\Gamma}$ and the subgroup of length zero elements $\Omega$. With this bijection one can identify the Kottwitz map on $B(\widetilde{W})$ with the projection $\widetilde{W} \rightarrow \Omega$. 

For $b \in G(\breve{F})$ and $w \in \widetilde{W}$ the corresponding \emph{affine Deligne-Lusztig variety} is defined as
\begin{equation*}
  X_{w}(b) = \{ g \in G(\breve{F})/I \mid g^{-1}b\sigma(g) \in IwI\},
\end{equation*} 
where we are identifying the element $w$ in the extended affine Weyl group with a representative in $N_T(\breve{F})$. In the following, we are going to study some so-called \emph{fine affine Deligne-Lusztig varieties}, compare \cite{gh_cox1}*{Sec.\ 3.4}. First, let $\mu^{\vee}$ be a minuscule coweight in $X_{*}(T)_\Gamma$, the \emph{admissible set} associated to $\mu^\vee$ is 
\begin{equation*}
  \mathrm{Adm}(\mu^\vee) = \{ w \in \widetilde{W} \mid w \le t^{x(\mu^\vee)} \text{ for some } x \in W_0\}.
\end{equation*}
Fix a subset $J \subset \widetilde{\mathbb{S}}$ and denote by $P_J$ the corresponding parahoric subgroup of $G(\breve{F})$. For $w \in {}^J \widetilde{W}$ and $b \in G(\breve{F})$ the associated fine affine Deligne-Lusztig variety is
\begin{equation*}
  X_{J,w}(b) = \{ g \in G(\breve{F})/P_J \mid g^{-1}b\sigma(g) \in P_J \cdot_{\sigma} IwI \}.
\end{equation*} In other words, it is the image of the affine Deligne-Lusztig variety for $I, b, w$ under the map $G/I \rightarrow G/P_J$. For a minuscule cocharacter $\mu$ we also consider the union 
\begin{equation*}
  X(\mu,b)_{J} = \bigcup_{w \in \mathrm{Adm}(\mu)} \{ g \in G(\breve{F})/P_J \mid g^{-1}b\sigma(g) \in  P_JwP_J \}. 
\end{equation*} 
The varieties appearing in the union above are called \emph{coarse} affine Deligne-Lusztig varieties, and it is proved in \cite{gh_cox1}*{Thm.\ 4.1.2} that $X(\mu, b)_J$ can be actually written as a union of \textit{fine} affine Deligne-Lusztig varieties as follows
\begin{equation}\label{eq:genadlv}
  X(\mu,b)_{J} = \bigsqcup_{w \in \mathrm{Adm}(\mu) \cap {}^J\widetilde{W}} \{ g \in G(\breve{F})/P_J \mid g^{-1}b\sigma(g) \in P_J \cdot_{\sigma} IwI \} = \bigsqcup_{w \in \mathrm{Adm}(\mu) \cap {}^J\widetilde{W}} X_{J,w}(b).
\end{equation}

The reason why we are interested in $X(\mu,b)_J$ is that it naturally arises in the study of Rapoport-Zink spaces. One can associate to a Rapoport-Zink space a quadruple $(G, \mu, b, J)$, as explained in \cite{zhu}*{Def.\ 3.8} and the corresponding union of affine Deligne-Lusztig varieties $X(\mu, b)_J$ over $F$ a mixed characteristic field. If the axioms of \cite{hr}*{Sec.\ 5} are satisfied, there is an isomorphism of perfect schemes
\begin{equation*}
  \N^{0,\mathrm{perf}} \cong X(\mu, \id)_J,
\end{equation*}
compare \cite{zhu}*{Prop.\ 3.11} and \cite{ghn_fully}*{Sec.7}. The axioms of \cite{hr}*{Sec.\ 5} have been shown to hold for ramified unitary groups in odd dimension in \cite{zhu}*{Prop.\ 0.4}, but are still to be proven in even dimension. In any case, by \cite{rap} there is in general a bijection between the $\F$-valued points of $\N^0$ and those of the corresponding $X(\mu, b)_J$, again defined over a field of mixed characteristic.

Before describing the group theoretical datum attached to our specific problem, we recall some more general results that we need in the sequel. We start with the reduction method \`a la Deligne and Lusztig as stated and proved in \cite{he_stab}.

\begin{thm}\label{thm:red}\cite{he_stab}*{Prop.\ 3.3.1}
  Let $w \in \widetilde{W}$, $s \in \widetilde{\mathbb{S}}$ and $b \in G(\breve{F})$ and assume $F$ has equal characteristic. 
  \begin{itemize}
    \item[(i)] If $\ell(sw\sigma(s)) = \ell(w) $ then there is a universal homeomorphism $X_{w}(b) \rightarrow X_{sw\sigma(s)}(b)$.
    \item[(ii)] If  $\ell(sw\sigma(s)) = \ell(w) - 2$, then $X_{w}(b) = X_1 \sqcup X_2$, with $X_1$ open and universally homeomorphic to a Zariski-locally trivial $\mathbb{G}_m$-bundle over $X_{sw}(b)$, while $X_2$ is closed and universally homeomorphic to a Zariski-locally trivial $\mathbb{A}^1$-bundle over $X_{sw\sigma(s)}(b)$.
  \end{itemize}
  If $F$ has mixed characteristic the statements above still hold, provided one replaces $\mathbb{G}_m$ and $\mathbb{A}^1$ with their perfections.
\end{thm}

Applying the reduction method repeatedly delivers a decomposition of an affine Deligne-Lusztig variety $X_w(b)$ into pieces homeomorphic to sequences of one-dimensional fiber bundles over affine Deligne-Lusztig varieties for elements in the Weyl group that have minimal length in their $\sigma$-conjugacy class. Recall that for an element $x$ of minimal length in its $\sigma$-conjugacy class the affine Deligne-Lusztig variety $X_x(b)$ is non-empty if and only if $x \in [b]$, see \cite{he_geohom}*{Thm. 3.2}. 

For an element $w$ in the affine Weyl group $W_a$, we denote by $\supp(w)$ the support of $w$, \textit{i.e.}\ the subset of affine reflections in $\widetilde{\mathbb{S}}$ appearing in a reduced expression for $w$. For $w\tau \in \widetilde{W} = W_a \rtimes \Omega$ following \cite{hn_minii}*{Sec. 4.3} we define the $\sigma$-support as
\begin{equation*}
  \supp_{\sigma}(w\tau) = \bigcup_{n \in \mathbb{Z}} (\tau \sigma)^n (\supp(w)).
\end{equation*}
For $w \in \widetilde{W}$ and a subset $J$ of the affine  reflections $\widetilde{\mathbb{S}}$ we define $I(w,J, \sigma)$ to be the maximal subset of $J$ that is stable under $\Ad(w) \sigma$, where $\Ad(w)$ is just the conjugation action of $w$, compare \cite{gh_cox1}*{3.1}.

\begin{thm} \label{thm:finecoarse} \cite{gh_cox1}*{Thm.\ 4.1.2}
  For any $J \subset \widetilde{\mathbb{S}}$ and $w \in { }^{J} \widetilde{W}$ and $b \in G(\breve{F})$, the fine affine Deligne-Lusztig variety satisfies
  \begin{equation*}
    X_{J,w}(b) \cong \{gP_{I(w,J, \sigma)} \mid g^{-1}b\sigma(g) \in P_{I(w,J,\sigma)}wP_{I(w,J,\sigma)}\}. 
  \end{equation*} 
\end{thm}

For $b \in G(\breve{F})$ we consider the $\sigma$-centralizer $\mathbb{J}_b = \{g \in G(\breve{F}) \mid g^{-1}b\sigma(g) = b\}$, and its action on the affine Deligne-Lusztig variety $X_w(b)$.
Combining \cite{gh_cox1}*{Thm 4.1.1-2} we obtain the following result.
\begin{thm}\label{thm:finsupp}
  Let $J \subset \widetilde{\mathbb{S}}$ and $w \in {}^J\widetilde{W} \cap W_a\tau$ be such that $W_{\supp_\sigma(w) \cup I(w, J, \sigma)}$ is finite. Then 
  \begin{equation*}
    X_{J, w}(\tau) \cong \bigsqcup_{i \in \mathbb{J}_{\tau}/ (\mathbb{J}_{\tau} \cap P_{\supp_{\sigma}(w) \cup I(w, J, \sigma)})} i X_{I(w, J, \sigma)}(w), 
  \end{equation*}
  where $X_{I(w, J, \sigma)}(w) = \{ g \in P_{\supp_{\sigma}(w) \cup I(w, J, \sigma)}/P_J \mid g^{-1}\tau\sigma(g) \in P_{I(w,J, \sigma)}wP_{I(w,J, \sigma)}\}$ is a classical Deligne-Lusztig variety in the partial flag variety $P_{\supp_{\sigma}(w) \cup I(w, J, \sigma)}/P_{I(w, J, \sigma)}$.
\end{thm}

We conclude with two simple results on the non-emptiness pattern, which are surely known to experts, but for which we could not find any reference in the literature.
\begin{prop}\label{prop:minlen} 
  Let $b$ be a basic element, that is $b \in [\tau]$ for a length-zero element $\tau \in \Omega$. Let $w\tau \in \widetilde{W} = W_a \rtimes \Omega$ be a minimal length element in its $\sigma$-conjugacy class.  Then $X_{w\tau}(b) \neq \emptyset$ if and only if $\supp_{\sigma}(w\tau)$ generates a finite subgroup of $W_a$.
\end{prop}
\begin{proof}
  By \cite{he_geohom}*{Thm.\ 2.3} there is a set $J \subset \widetilde{\mathbb{S}}$, a $\sigma$-straight element $x \in {}^{J}\widetilde{W}^{\sigma(J)}$, and an element $u$ with $\sigma$-support in the finite subgroup $\widetilde{W}_J$ such that $w\tau = ux$ and $\Ad(x)\sigma(J) = J$. By \cite{he_geohom}*{Thm.\ 3.2} $X_{ux}(b)$ is non-empty if and only if $X_x(b)$ is non-empty. Since $b$ is $\sigma$-conjugate to $\tau$, $X_x(b)$ is non-empty if and only if the same is true for $X_x(\tau)$. For a $\sigma$-straight $x$ the affine Deligne-Lusztig variety $X_x(\tau)$ is non-empty, if and only if $x$ is $\sigma$-conjugate to $\tau$, compare \cite{he_geohom}*{Prop.\ 4.5}. Since both $x$ and $\tau$ are $\sigma$-straight and $\sigma$-conjugate, $x$ has length zero, too. As we have seen, the set of length-zero elements $\Omega$ is in bijection with the image of the Kottwitz map, which can then be identified with the projection $\widetilde{W} \rightarrow \Omega$. Since $\sigma$-conjugate elements have the same image under the Kottwitz map, if $x$ and $\tau$ are $\sigma$-conjugate and both have length zero they are actually equal.
  It follows that $X_{w\tau}(\tau)$ is non-empty if and only if $w\tau = u\tau$. This means that $\supp_{\sigma}(w\tau) = \supp_{\sigma}(u\tau)$, which is finite by definition of $u$ and $x = \tau$.

  Assume $\supp_{\sigma}(w\tau)$ is finite. Since the elements $(\tau\sigma)^i(w)$ belong to the finite subgroup of $\widetilde{W}$ generated by $\supp_{\sigma}(w)$, there is an integer $n$ such that $w(\tau\sigma)(w)\dotsm (\tau\sigma)^n(w) = 1$. It is easy to see that $(w\tau\sigma)^n = w(\tau\sigma)(w)\dotsm (\tau\sigma)^n(w)\tau^n = \tau^n$. By the formula (\ref{eq:newton}) above to compute the Newton point of elements of $\widetilde{W}$, it follows that $w\tau$ and $\tau$ have the same Newton point. As we have seen, the Kottwitz map can be identified with the projection $\widetilde{W} \rightarrow \Omega$, from which it follows that $\kappa(w\tau) = \kappa(\tau)$ and therefore $w\tau$ and $\tau$ are $\sigma$-conjugate.
\end{proof}

\begin{lem}\label{lem:estell}
  If $\ell(w) \le 2\langle \nu_w, \rho \rangle +1$ then $w$ has minimal length in its $\sigma$-conjugacy class.
\end{lem}
\begin{proof}
  Observe that $\sigma$-conjugation preserves the parity of the length. If for some $v \in \widetilde{W}$ we have $ \ell(vw\sigma(v^{-1})) < \ell(w)$ it follows that $\ell(vw\sigma(v)^{-1}) \le \ell(w) -2 \le 2\langle \nu_w, \rho \rangle -1$ which is smaller than the length of a $\sigma$-straight element with same Newton point. By \cite{he_geohom}*{Thm.\ 2.3} this is a contradiction.
\end{proof}

\subsection{The group-theoretical datum associated to $\GU(2,4)$ over a ramified prime}\label{sec:adlvs3}
As we have mentioned above, we can associate to our Rapoport-Zink space a group-theoretical datum $(W_a, J, \sigma, \mu)$ and study the corresponding union of fine affine Deligne-Lusztig varieties $X(\mu,b)_J$. As explained in \cite{he_hecke}*{Ex.\ 2.2} the extended affine Weyl group associated to the ramified unitary group in even dimension $2m$ is the same in both split and non-split case. It has affine Dynkin diagram of type $BC_m$ (or ${ }^{2}BC_m$, in the non-split case, which only differs in the orientation, which is irrelevant for the Weyl group), as depicted below. 
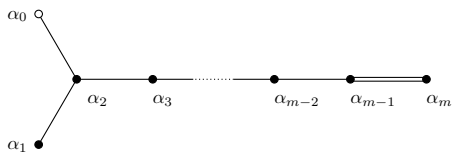
\begin{center}
  \begin{dynkinDiagram} [extended, arrows = false, edge length=1cm, labels={\alpha_0,\alpha_1,\alpha_2,\alpha_3,\alpha_{m-2},\alpha_{m-1},\alpha_m}, label macro/.code={\drlap{#1}}]B{}
  \end{dynkinDiagram}
\end{center}
By looking at the Dynkin diagram we immediately see that the subsets of reflections $\widetilde{\mathbb{S}} \setminus \{s_0\}$ and $\widetilde{\mathbb{S}} \setminus \{s_1\}$ generate two finite Weyl groups of type $C_m$, while the reflections in  $\widetilde{\mathbb{S}} \setminus \{s_m\}$ generate a finite group of type $D_m$. Following \cite{he_hecke}*{Ex.\ 2.2} we observe that there is exactly one symmetry of the Dynkin diagram, namely the transformation given by exchanging $s_0$ and $s_1$. It follows that the length-zero subset $\Omega$ consists of exactly two elements. The action of $\sigma$ on the extended affine Weyl group is then given by the adjoint action of one of these two elements. If the form is split then the action of $\sigma$ is trivial, if the form is non-split, then the Frobenius is given by the action of the non-trivial element $\tau \in \Omega$.

The choice of the subset $J$ of affine simple reflections is determined by the level structure. As in \cite{rtw} the level structure for our Rapoport-Zink space is given by the parahoric subgroup stabilizing a lattice in the vector space $C$ which is self-dual with respect to the Hermitian form. By \cite{gh_cox1}*{Sec.\ 7.4} this parahoric level structure corresponds to the subset $J = \{s_0, s_1,\dots, s_{m-1}\}$.

Last, the cocharacter $\mu^\vee$ corresponds to the choice of the signature. In our case $\mu^\vee$ is then the fundamental coweight $\omega_2^\vee$ corresponding to the simple root $\alpha_2$. Observe that for $m \ge 3$ the data $(BC_{m}, \widetilde{\mathbb{S}}\setminus \{s_m\}, 1, \omega_2^\vee)$ and $({}^2BC_{m}, \widetilde{\mathbb{S}}\setminus \{s_m\}, \tau, \omega_2^\vee)$ are not among those appearing in \cite{ghn_fully}*{Sec.\ 3}. This means that the corresponding union $X(\omega_2^\vee, 1)_J$ of affine Deligne-Lusztig varieties is not fully Hodge-Newton decomposable, which is the main source of difference with the case studied in \cite{rtw}. For example, we cannot expect a decomposition of $X(\omega_2^\vee, 1)_J$ as a disjoint union of classical Deligne-Lusztig varieties, which matches our results in Section \ref{sec:geometry}.

\subsubsection{The split case} 
Consider the group-theoretical datum $({BC}_3, J = \{0,1,2\}, \id, \omega_2^\vee)$ associated to the  group $\GU(2,4)$ ramified over $p$ and such that the Hermitian form on $C$ is split. We first need to compute the admissible set and its representatives in ${}^JW$.  Let $\Adm(\omega_2^\vee)^J = \Adm(\omega_2^\vee) \cap {}^J\widetilde{W}$ denote the set of minimal length representatives in ${}^J\widetilde{W}$ of the admissible elements. This set can be easily computed with the mathematical software Sagemath \cite{sagemath} and a code can be found in Appendix \ref{AppendixC}. We obtain
\begin{align*}
  \Adm(\omega_2^\vee)^J = \{&1, s_3, s_3s_2, s_3s_2s_1, s_3s_2s_3, s_3s_2s_3s_1, s_3s_2s_3s_1s_2, s_3s_2s_0, s_3s_2s_1s_0, \\ 
  &s_3s_2s_3s_0,s_3s_2s_1s_0s_2, s_3s_2s_3s_0s_2, s_3s_2s_3s_1s_0, s_3s_2s_3s_0s_2s_1, s_3s_2s_3s_1s_0s_2, \\
  &s_3s_2s_3s_1s_0s_2s_1, s_3s_2s_3s_1s_2s_0, s_3s_2s_3s_1s_0s_2s_0, s_3s_2s_3s_1s_0s_2s_1s_0\}.
\end{align*}
In the following proposition we show that the $J$-admissible elements can be grouped into three families, corresponding to three different behaviors of the affine Deligne-Lusztig variety $X_{J, w}(1)$.
\begin{prop}\label{prop:split}
  Consider the group theoretical datum $({BC}_3, J = \{0,1,2\}, \id, \omega_2^\vee)$ associated to ramified $\GU(2,4)$. Then $w \in \Adm(\omega_2^\vee)^J$ satisfies one of the following properties.
  \begin{itemize}
    \item[(i)] $w$ has finite support in a subgroup of type $C_r$ of $W_a$ with $r \le 3$. In this case the affine Deligne-Lusztig variety $X_{J,w}(1)$ has irreducible components isomorphic to (generalized) Deligne-Lusztig varieties for the symplectic group $\mathrm{Sp}_{2r}$. The set of irreducible components is in bijection with the set of vertex lattices of type $2r$.
    \item[(ii)] $w$ has full $\sigma$-support and can be reduced by one step of Deligne and Lusztig's reduction method. In this case the affine Deligne-Lusztig variety $X_{J, w}(1)$ has irreducible components universally homeomorphic to  $\mathbb{A}^1$-bundles over a classical Deligne-Lusztig variety for $\mathrm{SO}_6$. The set of irreducible components of $X_{J, w}(1)$ is in bijection with the set of $2$-modular lattices.
    \item[(iii)] $w$ has full $\sigma$-support and minimal length in its $\sigma$-conjugacy class, in which case $X_{J,w}(1)$ is empty. 
  \end{itemize}
\end{prop}
\begin{proof}
  \par{(i)} We first consider the elements of $\Adm(\omega_2^\vee)^J$ with finite $\sigma$-support
  \begin{equation*}
  1, s_3, s_3s_2, s_3s_2s_3, s_3s_2s_1, s_3s_2s_0, s_3s_2s_3s_1, s_3s_2s_3s_0, s_3s_2s_3s_1s_2, s_3s_2s_3s_0s_2.
  \end{equation*}
 It is clear that their support generates a subgroup of type $C_r$, compare also the Dynkin diagram above.  As we have recalled in Theorem \ref{thm:finsupp} the corresponding fine affine Deligne-Lusztig variety $X_{J,w}(1)$ can be decomposed as a disjoint union of classical Deligne-Lusztig varieties for the group $\mathrm{Sp}_{2r}$. Since the support of $w$ generates the Weyl group $C_r$, it satisfies the hypothesis of Theorem \ref{thm:br} and hence the corresponding (classical) Deligne-Lusztig variety is irreducible.
 
 By Theorem \ref{thm:finsupp} the index set of the disjoint decomposition of $X_{J,w}(1)$ depends on the set of reflections $\supp_{\sigma}(w) \cup I(w, \sigma, J)$. If $w = 1$ the set $\supp_{\sigma}(w) \cup I(1, \sigma, J)$ coincides with $J$. If $w = s_3$ it is $\{s_0, s_1, s_3\}$. If the reflection $s_2$ appears in a reduced expression of $w$, then $I(w, \sigma, J)$ is empty. Observe that the subset $\Adm(\omega_1^\vee)^J = \{1, s_3, s_3s_2, s_3s_2s_1, s_3s_2s_0\}$, which corresponds to the admissible set for $\GU(1,5)$, produces the same collection of sets $\supp_{\sigma}(w) \cup I(w, \sigma, J)$. These were studied in \cite{gh_cox1}*{Ex.\ 7.4.1}. In particular, it is proved there that the index set $\mathbb{J}_1/\mathbb{J}_1 \cap P_{\supp_{\sigma}(w) \cup I(w,J, \sigma)}$ in the decomposition of $X_{J,w}(1)$ is in bijection with the set of vertex lattices of type $0,2,4$ or $6$, respectively. These observations are summarized in the following table.
 \begin{center}
  \bgroup
\def\arraystretch{1.5}

  \begin{tabular}{|l|l|l|}
  \hline
  Elements &  $\supp_{\sigma}(w) \cup I(w, \sigma, J)$ & Type \\
  \hline
  $1$  & $J=\{s_0, s_1, s_2\}$ & $0$ \\
  \hline
  $s_3$ &  $\{s_0, s_1, s_3\}$ & 2 \\
  \hline
  $s_3s_2, ~~~ s_3s_2s_3$ & $\{s_2, s_3\}$ & 4 \\
  \hline
  $s_3s_2s_1, ~~~  s_3s_2s_3s_1, ~~~ s_3s_2s_3s_1s_2$ & $\{s_1, s_2, s_3\}$ & $6$ \\
  \hline
  $s_3s_2s_0, ~~~ s_3s_2s_3s_0, ~~~  s_3s_2s_3s_0s_2$ & $\{s_0, s_2, s_3\}$ & $6$ \\
  \hline
 \end{tabular}\egroup
 \end{center}
 Since $\sigma = 1$ we actually have two $\sigma$-stable subgroups of type $C_3$ in $W_a$, one is generated by $\{s_1,s_2, s_3\}$ and the other by $\{s_0, s_2, s_3\}$. This corresponds to the fact that if the form is split there are two orbits of self-dual lattices in $C$, as remarked in \cite{gh_cox1}*{Ex.\ 7.4}, and explains why the elements above come in pairs. Observe that the elements appearing in the list above are exactly the same elements in the stratification (\ref{eq:strC}) of $S_V$, and consequently in the stratification of the irreducible components of type $\N_{\mathcal{L}}$ of Proposition \ref{prop:irred}. 

 \par{(ii)} There is only one element in $\Adm(\omega_2^\vee)^J$ with full support that can be reduced via Deligne and Lusztig's method. Indeed, by conjugating $s_3s_2s_3s_1s_0$ with $s_3$ we obtain the shorter element $s_2s_1s_0$ that is a Coxeter element for the finite subgroup of type $D_3$ generated by $\{s_0, s_1, s_2\}$. The other element produced by the reduction method is $s_3s_2s_1s_0$ which is a $\sigma$-Coxeter element for $W_a$, and it is therefore $\sigma$-straight with non-basic Newton point $(\frac{1}{2}, \frac{1}{2}, 0)$. It follows that $X_{s_3w}(1)$ is empty. By Theorem \ref{thm:finecoarse} the fine affine Deligne-Lusztig variety $X_{J,w}(b)$ is isomorphic to the  affine Deligne-Lusztig variety $X_w(1)$, as $I(J, \sigma, w ) = \emptyset$. By the reduction method and the previous observations, the latter is universally homeomorphic to a line bundle over the affine Deligne-Lusztig variety $X_{s_3ws_3}(1)$. Using Theorem \ref{thm:finsupp}, we obtain the disjoint decomposition of $X_{s_3ws_3}(1)$ into classical Deligne-Lusztig varieties for $\mathrm{SO}_6$. Again, since $s_3ws_3 = s_2s_1s_0$ has full support in the finite subgroup of type $D_3$, the classical Deligne-Lusztig varieties $X_B(s_3ws_3)$ are irreducible. It follows that they are the irreducible components of $X_{s_3ws_3}(1)$. Last, observe that $\supp(s_3ws_3) \cup I(s_3ws_3, J, \sigma) = \{s_0, s_1, s_2\} = J$. We have seen in the proof of (i) that in this case the index set $\mathbb{J}_1/\mathbb{J}_1 \cap P_{\supp_{\sigma}(w) \cup I(w,J, \sigma)}$ of the decomposition of $X_{s_3ws_3}(1)$ into classical Deligne-Lusztig varieties is in bijection with the set of vertex lattices of type $0$. By Proposition \ref{prop:2vl} these are in bijection with the set of $2$-modular lattices.

 \par{(iii)} Last, by Proposition \ref{prop:minlen}, in order to prove that $X_{J,w}(1)$ is empty for the remaining elements, it is enough to prove that these elements have minimal length in their $\sigma$-conjugacy classes. By the formula (\ref{eq:newton}), we can compute their Newton points, the corresponding SageMath code can be found in Appendix \ref{AppendixC},
 \begin{center}
  \bgroup
\def\arraystretch{1.2}

  \begin{tabular}{|l|c|}
    \hline
    Element & Newton point \\
    \hline
    $s_3s_2s_1s_0$  &$(\tfrac{1}{2},\tfrac{1}{2}, 0)$\\ 
    \hline
    $s_3s_2s_1s_0s_2$ &$(1, 0, 0)$ \\ 
    \hline
    $s_3s_2s_3s_1s_2s_0, ~~~ s_3s_2s_3s_0s_2s_1$ &$(\tfrac{2}{3}, \tfrac{2}{3}, \tfrac{2}{3})$ \\
    \hline
    $s_3s_2s_3s_1s_0s_2$  &$(1, 0, 0)$\\
    \hline
    $s_3s_2s_3s_1s_0s_2s_0, ~~~ s_3s_2s_3s_1s_0s_2s_1$ &$(1, \tfrac{1}{2},\tfrac{1}{2})$  \\
    \hline
    $s_3s_2s_3s_1s_0s_2s_1s_0$  &$(1,1,0)$ \\
    \hline
  \end{tabular}\egroup
 \end{center}
 Recall that for an affine Weyl group of type $\widetilde{B}_3$, the half-sum of the positive roots $\rho$ is $(\tfrac{5}{2}, \tfrac{3}{2}, \tfrac{1}{2})$. It is then straightforward to see that all elements in the list above, except for $w = s_3s_2s_3s_1s_0s_2$ are $\sigma$-straight. Observe that the remaining element has length $6$ and Newton point $(1,0,0)$. It satisfies then the hypothesis $\ell(w) = 2 \langle \rho, \nu_w \rangle +1$ of Lemma \ref{lem:estell} which implies that it has minimal length in its $\sigma$-conjugacy class. 
\end{proof}

\subsubsection{The non-split case} 
Consider now the group-theoretical datum $({}^2BC_3, J = \{0,1,2\}, \sigma, \omega_2^\vee)$ associated to the group $\GU(2,4)$ over a ramified prime and with non-split Hermitian form on $C$. Recall that in this case the Frobenius $\sigma$ on the extended affine Weyl group exchanges the reflections $s_0$ and $s_1$. The admissible set $\Adm(\omega_2^\vee)^J$ does not depend on $\sigma$, hence it coincides with the admissible set computed for the split case. The following proposition is the analogue of Proposition \ref{prop:split} for the non-split case.

\begin{prop}\label{prop:nonsplitadlv}
  Consider the group theoretical datum $({}^2BC_3, J = \{0,1,2\}, \sigma, \omega_2^\vee)$ associated to the non-split ramified group $\GU(2,4)$. Then $w \in \Adm(\omega_2^\vee)^J$ satisfies one of the following properties.
  \begin{itemize}
    \item[(i)] $w$ has $\sigma$-support in a finite subgroup of $W_a$ of type $C_r$, with $r \le 2$. In this case the affine Deligne-Lusztig variety $X_{J,w}(1)$ has irreducible components isomorphic to (generalized) Deligne-Lusztig varieties for the symplectic group $\mathrm{Sp}_{2r}$. The set of irreducible components of $X_{J,w}(1)$ is in bijection with the set of vertex lattices of type $2r$.
    \item[(ii)] $w$ has full $\sigma$-support, and can be reduced with Deligne and Lusztig's method to an element with finite $\sigma$-support in a subgroup of type $D_3$ of $W_a$. In this case the affine Deligne-Lusztig variety $X_{J, w}(1)$ has irreducible components universally homeomorphic to vector bundles of dimension $1$ or $2$ over a classical Deligne-Lusztig variety for $\mathrm{SO}_6$. The set of irreducible components of $X_{J,w}(1)$ is in bijection with the set of $2$-modular lattices.
    \item[(iii)] $w$ has full $\sigma$-support and minimal length in its $\sigma$-conjugacy class, in which case $X_{J,w}(1)$ is empty. 
  \end{itemize}
\end{prop}
\begin{proof}
  \par{(i)} We first consider the elements of $\Adm(\omega_2^\vee)^J$ with finite $\sigma$-support
  \begin{equation*}
  1, s_3, s_3s_2, s_3s_2s_3.
  \end{equation*}
 It is clear that their support generates a subgroup of type $C_r$, compare also the Dynkin diagram above.  By Theorem \ref{thm:finsupp} the corresponding fine affine Deligne-Lusztig variety $X_{J,w}(1)$ can be decomposed as a disjoint union of classical Deligne-Lusztig varieties for the group $\mathrm{Sp}_{2r}$. Since the $\sigma$-support of $w$ generates the Weyl group $C_r$, it satisfies the hypothesis of Theorem \ref{thm:br} and hence the corresponding Deligne-Lusztig variety is irreducible.
 
 By Theorem \ref{thm:finsupp} the index set of the decomposition of $X_{J,w}(1)$ depends on the set of reflections $\supp_{\sigma}(w) \cup I(w, \sigma, J)$. If $w = 1$ this coincides with $J$, if $w = s_3$ this is $\{s_0, s_1, s_3\}$, otherwise it coincides with the support of $w$, so it is $\{s_2, s_3\}$. Again by \cite{gh_cox1}*{Ex. 7.4.2} we know that the index set $\mathbb{J}_1/\mathbb{J}_1 \cap P_{\supp(w) \cup I(w,J, \sigma)}$ in the decomposition of $X_{J,w}(1)$ is in bijection with vertex lattices of type $0,2$ or $4$, respectively. If we compare this with the first part of Proposition \ref{prop:split}, we see that the elements corresponding to vertex lattices of type $6$ are missing. This is due to the fact that if the Hermitian form on $C$ is non-split,  such vertex lattices do not exist, as we have recalled in Section \ref{sec:lattices}.  Last, observe that the elements appearing in the list above are exactly the same elements as in the stratification (\ref{eq:strC}) of $S_V$, and consequently in the stratification of the irreducible closed subschemes $\N_{\mathcal{L}}$ as in (\ref{eq:strata}). 

 \par{(ii)} There are five elements in $\Adm(\omega_2^\vee)^J$ with full $\sigma$-support that can be reduced via Deligne and Lusztig's method, namely 
\begin{equation*}
    s_3s_2s_3s_1, s_3s_2s_3s_0, s_3s_2s_3s_1s_0, s_3s_2s_3s_1s_2s_0, s_3s_2s_3s_0s_2s_1.
\end{equation*}
Indeed, by $\sigma$-conjugating the first three with $s_3$ we obtain the shorter elements $s_2s_1$, $s_2s_0$ and $s_2s_1s_0$, respectively. The first two are $\sigma$-Coxeter elements for the finite $\sigma$-stable subgroup of $W_a$ of type $D_3$ generated by $\{s_0, s_1, s_2\}$, the third still has full $\sigma$-support in this subgroup. The three elements of the form $ws_3$ produced by the reduction method are $s_3s_2s_1$, $s_3s_2s_0$ and $s_3s_2s_1s_0$, respectively. The first two are $\sigma$-Coxeter elements for $W_a$ and therefore $\sigma$-straight with non-basic with Newton point $(\tfrac{1}{3}, \tfrac{1}{3}, \tfrac{1}{3})$. The latter has Newton point $(\tfrac{1}{2}, \tfrac{1}{2}, 0)$ and length $4$, so it is again $\sigma$-straight.  
 
For $w$ one of these three elements $\{s_3s_2s_3s_1, s_3s_2s_3s_0, s_3s_2s_3s_1s_0\}$, by Theorem \ref{thm:finecoarse} the fine affine Deligne-Lusztig varieties $X_{J,w}(b)$ is isomorphic to the  affine Deligne-Lusztig variety $X_w(1)$. By the reduction method the latter is then universally homeomorphic to a line bundle over the affine Deligne-Lusztig variety $X_{s_3ws_3}(1)$. Using Theorem \ref{thm:finsupp}, we further obtain a disjoint decomposition of $X_{s_3ws_3}(1)$ into classical Deligne-Lusztig varieties for $\mathrm{SO}_6$. Again, since $s_3ws_3$ has full support in the finite subgroup of type $D_3$, the varieties $X(s_3ws_3)$ are irreducible. It follows that they are the irreducible components of $X_{s_3ws_3}(1)$. Last, observe that $\supp_{\sigma}(w) \cup I(w, J, \sigma) = \{s_0, s_1, s_2\} = J$. We have already seen that in this case the index set $\mathbb{J}_1/\mathbb{J}_1 \cap P_J$ in the decomposition of $X_{s_3ws_3}(1)$ into classical Deligne-Lusztig varieties is in bijection with the set of $2$-modular lattices.

Consider now $w = s_3s_2s_3s_1s_2s_0$ and observe that it is $\sigma$-conjugate to $s_3s_2s_3s_0s_2s_1$ by the length-zero element $\tau$. Therefore, it is enough to study $X_{J,w}(1)$ as the two are universally homeomorphic. The reduction method consists first of two length-preserving $\sigma$-conjugations, namely by $s_1$ and $s_3$. We obtain the element $s_2s_3s_1s_2s_3s_2$ that can be reduced via $\sigma$-conjugation by $s_2$ to the shorter element $s_3s_1s_2s_3$. Another conjugation by $s_3$ brings us to $s_1s_2$, which has finite $\sigma$-support in a subgroup of type $D_3$. It remains to check the other two elements produced by the two length-decreasing steps of the reduction method, namely $s_3s_1s_2s_3s_2$ and $s_3s_1s_2$. The latter is $\sigma$-Coxeter as we already remarked, hence $\sigma$-straight, so the corresponding affine Deligne-Lusztig variety is empty. We compute the Newton point of $s_3s_1s_2s_3s_2$ by taking $\sigma$-powers and obtain $(\tfrac{1}{2}, \tfrac{1}{2}, 0)$. Then we can see that this element satisfies the hypothesis of Lemma \ref{lem:estell} and hence has minimal length in its $\sigma$-conjugacy class and the corresponding affine Deligne-Lusztig variety is again empty. By the reduction method it follows that $X_w(1)$ is universally homeomorphic to a $2$-dimensional vector bundle over $X_{s_1s_2}(1)$, whose irreducible components are the classical Deligne-Lusztig varieties $X_B(t_2t_1)$ in the notation of Section \ref{sec:dlvs}. We have then obtained the analogous decomposition as in Proposition \ref{prop:nonsplit} and Remark \ref{rem:last}.

\par{(iii)} Last, we need to prove that the remaining admissible elements have minimal length in their conjugacy classes. Using SageMath, we first compute their Newton points
\begin{center}
  \bgroup
\def\arraystretch{1.2}
\begin{tabular}{|l|c|}
\hline
Element & Newton point \\
\hline 
  $s_3s_2s_1, s_3s_2s_0$ & $(\tfrac{1}{3}, \tfrac{1}{3}, \tfrac{1}{3})$ \\
  \hline
  $s_3s_2s_1s_0$  & $(\tfrac{1}{2},\tfrac{1}{2}, 0)$ \\
  \hline 
  $s_3s_2s_1s_0s_2$ &$(1, 0, 0)$ \\ 
  \hline
  $s_3s_2s_1s_3s_2, s_3s_2s_0s_3s_2$ & $(\tfrac{1}{2},\tfrac{1}{2}, 0)$\\
  \hline 
  $s_3s_2s_3s_1s_0s_2$  &$(1, 0, 0)$\\
  \hline
  $s_3s_2s_3s_1s_0s_2s_0, ~~~ s_3s_2s_3s_1s_0s_2s_1$ &$(1, 0,0)$  \\
  \hline
  $s_3s_2s_3s_1s_0s_2s_1s_0$  &$(1,1,0)$ \\
  \hline
\end{tabular}
\egroup
\end{center}
One can easily check that the elements the first three rows, together with the last one are $\sigma$-straight. They are followed by three elements that satisfy the hypothesis of Lemma \ref{lem:estell} and therefore have minimal length in their $\sigma$-conjugacy class.
It remains to check the length $7$ elements $s_3s_2s_3s_1s_0s_2s_0$ and $s_3s_2s_3s_1s_0s_2s_1$, which are $\sigma$-conjugate to each other by the length $0$ element $\tau$. Hence, it is enough to prove the statement for one of them. We can $\sigma$-conjugate $w = s_3s_2s_3s_1s_0s_2s_0$ first by $s_2$ and then by $s_0$ to obtain $w' = s_0s_2s_3s_2s_0s_3s_2$, which still has length seven. Observe that the subword $x = s_0s_2s_3s_2s_0$ of $w'$ is $\sigma$-straight. Indeed, it is obtained by $\sigma$-conjugation from $s_3s_2s_1s_0s_2$, which we have already seen is $\sigma$-straight. Moreover, one can directly check that $x$ fixes the reflections $\{s_2,s_3\}$. It follows that $w' = xs_3s_2$ is factored as the product of a straight element and an element of finite order fixed by $x$, which by \cite{he_geohom}*{Thm.\ 2.3} implies that $w'$ is in the class of $x$ in $B(G)$. Suppose $w'$ does not have minimal length in its $\sigma$-conjugacy class in $B(\widetilde{W})$. Then since $\sigma$-conjugation preserves the parity of the length, $w'$ is conjugate to an element of the same length as $x$, which means to a $\sigma$-straight element with the same Newton point as $w'$. If follows that $w'$ has to be conjugate to $x$, by the bijection of \cite{hn_minii}*{Thm.\ 3.3} between conjugacy classes in $B(G)$ and classes in $B(\widetilde{W})$ containing a $\sigma$-straight element. Observe that the reflection $s_3$ appears once in any reduced expression for $x$ and twice in $w'$. Then these two cannot be $\sigma$-conjugate by the next lemma, and we can conclude that $w'$, and therefore $w$ has minimal length in its $\sigma$-conjugacy class.
\end{proof}

\begin{lem} 
  Let $W_a$ be the affine Weyl group of type $\widetilde{B_m}$. Let $n_m(w)$ be the number of times the reflection $s_m$ appears in any reduced expression of $w$. Then $n_m(w)$ is well-defined and its parity is preserved by $\sigma$-conjugation.
\end{lem}
\begin{proof}
  Recall that two reduced expression for $w$ are connected by a sequence of so-called \emph{braid moves}, see \cite{bb}*{Thm.\ 3.3.1}. The only braid move involving the reflection $s_m$ consists of substituting the subword $s_ms_{m-1}s_ms_{m-1}$ with the subword $s_{m-1}s_ms_{m-1}s_m$, and therefore it does not change the number of times $s_m$ appears in an expression for $w$. It follows that $n_m(w)$ is well-defined.

  It is enough to prove the second statement for $s_iw\sigma(s_i)$ where $s_i$ is a reflection in $W_a$. If $\ell(s_iw\sigma(s_i)) = \ell(w) + 2$, since $s_m$ is fixed by $\sigma$, the number $n_m(s_iw\sigma(s_i))$ is either equal to $n_m(w)$ or, if $s_i = s_m$, it increases by $2$, hence the parity is preserved. By the exchange property of Coxeter groups, see \cite{bb}*{Sec.\ 1.5}, if $\ell(s_iw) < \ell(w)$ then $s_iw$ has a reduced expression obtained by deleting one reflection $s_j$ from a reduced expression for $w$. Moreover, in this case $s_j$ and $s_i$ are conjugate. By \cite{bb}*{Ex.\ 1.16}, the only reflection conjugate to $s_m$ is $s_m$ itself, so the only case to consider is $s_i = s_m$. If $\ell(s_mws_m) = \ell(w)$ it means that multiplication on the left with $s_m$ deletes one instance of $s_m$ and the multiplication on the right replaces it. Therefore, the number of times $s_m$ appears does not change. If $\ell(s_mws_m) = \ell(w) - 2$ it means that $s_m$ gets deleted twice from a reduced expression of $w$, and again parity is preserved.
\end{proof}

\begin{rem}
  In \cite{he_aug} the authors study a family of elements in $\widetilde{W}$ called \emph{of finite Coxeter part}. For such elements the corresponding affine Deligne-Lusztig varieties can be decomposed via the reduction method as iterated fibrations over classical Deligne-Lusztig varieties for Coxeter elements. The $J$-admissible elements we have just studied give fibrations over classical Deligne-Lusztig varieties, too, which are however not of Coxeter type, in general. For example the Deligne-Lusztig variety $X_B(t_1t_2t_3)$ for the non-split orthogonal group found in the proof of Proposition \ref{prop:nonsplit} is not of Coxeter type. 
\end{rem}

\appendix
\section{The Gröbner basis of Proposition \ref{prop:radical6}}\label{Appendix}
We list here the polynomials of the Gröbner basis $G$ used in the proof of Proposition \ref{prop:radical6}. To make the notation more readable and the lexicographic order more intuitive we have substituted the variables $x_{ij}$ of Proposition \ref{prop:pap} with the twenty-one letters of the Italian alphabet. The symmetric matrix $X$ used in the proof of Proposition \ref{prop:radical6} becomes in the new notation the following matrix with entries in $\F_p[a, b, \dots, z]$
\begin{equation*} 
  X = \left( \begin{array}{cccccc}
    a & b & c & d & e & f \\ 
    b & g & h & i & l & m \\ 
    c & h & n & o & p & q \\
    d & i & o & r & s & t \\
    e & l & p & s & u & v\\
    f & m & q & t & v & z
  \end{array} \right).
\end{equation*}
The monomial order is then simply the usual alphabetical order. We list here the elements of the Gr\"obner basis used in the proof of Proposition \ref{prop:radical6} and already divide them into the subsets $G_{ij}$ according to Lemma \ref{lem:grob}. We also underline the distinguished generators used in the last part of the proof.

\begin{longtable}{| L | M |}
  \hline
  G_{11} = G_{a} &  a + g +n + r + u+ z \\
  \hline
  G_{12} = G_{b} & b^2 + g^2 + h^2 + i^2 + l^2 + m^2 \\ 
  &bc + gh + hn + io + lp + mq \\ 
  & bd + gi + ho + ir + ls + mt\\
  & be + gl + hp + is + lu + mv \\
  & bf + gm + hq + it + lv + mz \\
  & bh - cg - cr - cu - cz + do + ep + fq \\
  &bi + co - dg - dn - du - dz + es + ft \\
  &bl + cp + ds - eg - en - er - ez + fv \\
  &bm + cq + dt + ev - fg - fn - fr - fu \\
  &bn + br + bu + bz - ch - di - el - fm \\ 
  &bo^2 + br^2 + bs^2 + bt^2 - cio - dho + din - dir + diu + diz - dls - dmt - eis - fit \\ 
  &bop + brs + bsu + btv - clo - dhp - dis + dln + eiz - els - emt - fiv \\
  &boq + brt + btu + btz - cmo - dhq - dit + dmn - elt - fmt \\
  &bos - bpr - dhs + dip + ehr - eio \\
  &bot - bqr - dht + diq + fhr - fio \\
  &bou - bps - dhu + dlp + ehs - elo \\
  & bov - bqs - dhv + dmp + eiq - emo + fhs - fip \\
  & boz - bqt - dhz + dmq + fht - fmo \\
  & bp^2 + bs^2 + bu^2 + bv^2 - clp - dls - ehp - eis + eln + elr - elu + elz - emv - flv \\
  & bpq + bst + buv + bvz - cmp - dms - ehq - eit - elv + emn + emr - fmv \\
  & bpt - bqs - eht + eiq + fhs - fip \\
  & bpv - bqu - ehv + elq + fhu - flp \\
  & bpz - bqv - ehz + emq + fhv - fmp \\
  & bq^2 + bt^2 + bv^2 + bz^2 - cmq - dmt - emv - fhq - fit - flv + fmn + fmr  + fmu - fmz \\
  & bru - bs^2 - diu + dls + eis - elr \\
  & brv - bst - div + dms + eit - emr \\
  & brz - bt^2 - diz + dmt + fit - fmr \\
  & bsv - btu - eiv + elt + fiu - fls \\
  & bsz - btv - eiz + emt + fiv - fms \\
  & \boxed{buz - bv^2 - elz + emv + flv - fmu} \\
  \hline
  G_{13} = G_c &  c^2 + h^2 + n^2 + o^2 + p^2 + q^2 \\
  & cd + hi + no + or + ps + qt \\
  & ce + hl + np + os + pu + qv \\
  & cf + hm + nq + ot + pv + qz \\
  & cgi + cho + cir + cls + cmt - dgh - dhn - dio - dlp - dmq \\
  & cgl + chp + clr + clu + cmv + dhs - dlo - egh - ehn - ehr - elp - emq \\
  & cgm + chq + cmr + cmu + cmz + dht - dmo + ehv - emp - fgh - fhn - fhr - fhu - fmq \\
  &chi + cno + cor + cps + cqt - dh^2 - dn^2 - do^2 - dp^2 - dq^2 \\
  & chl + cnp + cpr + cpu + cqv + dns - dop - eh^2 - en^2 - enr - ep^2 - eq^2 \\
  & chm + cnq + cqr + cqu + cqz + dnt - doq + env - epq - fh^2 - fn^2  - fnr - fnu - fq^2 \\ 
  & cho^2 + chr^2 + chs^2 + cht^2 - cino - cior - clpr - cmqr - dhno - dhor - 2dhps - 2dhqt + din^2 + dio^2 - dis^2 - dit^2 - diu^2 - 2div^2 - diz^2 + 2dlop + dlrs + dlsu + 2dmoq + dmrt + 2dmtu + dmtz + ehpr + eirs + eisu + 2eitv - elo^2 - elr^2 - els^2 - elt^2 + emrv - 2emst+ fhqr + firt + fitz + flrv - fmo^2  - fmr^2 - 2fmru + fms^2 - fmt^2 \\
  & chop + chrs + chsu + chtv - clno - clor - clps - cmqs - dhnp - dhpr - dhpu - dhqv - dins + diop + dln^2 + dlnr + dlp^2 + dmpq - ehqt - eit^2 - eiv^2 - eiz^2 + emoq + emrt + emsv + emtz + fhqs + fist + fiuv + fivz - fmop - fmrs - fmsu - fmtv \\
  &choq + chrt + chtu + chtz - cmno - cmor - cmps - cmqt - dhnq - dhqr - dhqu - dhqz - dint + dioq + dmn^2 + dmnr + dmp^2 + dmq^2 - elnt + eloq + emns - emop \\
  & chp^2 + chs^2 + chu^2 + chv^2 - clnp - clpr - clpu - cmqu - dlns + dlop - ehnp - ehpr - ehpu - 2ehqv - eins + eiop + eln^2 + 2elnr - elo^2 + elp^2 - elt^2 - elv^2 - elz^2 + 2empq + emst + emuv + emvz + fhqu + flst + fluv + flvz - fmp^2 - fms^2 - fmu^2 - fmv^2 \\
  & chpq + chst + chuv + chvz - cmnp - cmpr - cmpu - cmqv - dmns \\
  & + dmop - ehnq - ehqr - ehqu - ehqz - eint + eioq - elnv + elpq \\
  & + emn^2 + 2emnr + emnu - emo^2 + emq^2 \\
  & chq^2 + cht^2 + chv^2 + chz^2 - cmnq - cmqr - cmqu - cmqz - dmnt + dmoq - emnv + empq- fhnq  - fhqr - fhqu - fhqz - fint + fioq - flnv + flpq + fmn^2 + 2fmnr + 2fmnu - fmo^2 - fmp^2 + fmq^2 \\
  &ci^2 + co^2 + cr^2 + cs^2 + ct^2 - dhi - dno - dor - dps - dqt \\
  &cil + cop + crs + csu + ctv - ehi - eno - eor - eps - eqt \\
  &cim + coq + crt + ctu + ctz + eov - ept - fhi - fno - for - fou - fqt \\
  & cip - clo - dhp + dln + eho - ein \\
  & ciq - cmo - dhq + dmn + fho - fin \\
  & cis - clr - dhs + dlo + ehr - eio \\
  & cit - cmr - dht + dmo + fhr - fio \\
  & ciu - cls - dhu + dlp + ehs - eip \\
  & civ - cms - dhv + dmp + fhs - fip \\
  & ciz - cmt - dhz + dmq + fht - fiq \\
  & cl^2 + cp^2 + cs^2 + cu^2 + cv^2 - ehl - enp - eos - epu - eqv \\
  & clm + cpq + cst + cuv + cvz - fhl - fnp - fos - fpu - fqv \\
  & clq - cmp - ehq + emn + fhp - fln \\
  & clt - cms - eht + emo + fhs - flo \\
  & clv - cmu - ehv + emp + fhu - flp \\
  & clz - cmv - ehz + emq + fhv - flq \\
  & cm^2 + cq^2 + ct^2 + cv^2 + cz^2 - fhm - fnq - fot - fpv - fqz \\
  & cos - cpr - dns + dop + enr - eo^2 \\
  & cot - cqr - dnt + doq + fnr - fo^2 \\
  & cou - cps - dnu + dp^2 + ens - eop \\
  & cov - cqs - dnv + dpq + fns - fop \\
  & coz - cqt - dnz + dq^2 + fnt - foq \\
  & cpt - cqs - ent + eoq + fns - fop \\
  & cpv - cqu - env + epq + fnu - fp^2 \\
  & cpz - cqv - enz + eq^2 + fnv - fpq \\
  & cru - cs^2 - dou + dps + eos - epr \\
  & crv - cst - dov + dqs + eot - eqr \\
  & crz - ct^2 - doz + dqt + fot - fqr \\
  & csv - ctu - eov + ept + fou - fps \\
  & csz - ctv - eoz + eqt + fov - fqs \\
  & \boxed{cuz - cv^2 - epz + eqv + fpv - fqu} \\
  \hline
  G_{14} = G_d & d^2 + i^2 + o^2 + r^2 + s^2 + t^2 \\
  & de + il + op + rs + su + tv \\
  & df + im + oq + rt + sv + tz \\
  & dgl + dhp + dis + dlu + dmv - egi - eho - eir - els - emt \\
  & dgm + dhq + dit + dmu + dmz + eiv - ems - fgi - fho - fir - fiu - fmt \\
  & dhl + dnp + dos + dpu + dqv - ehi - eno - eor - eps - eqt \\
  & dhm + dnq + dot + dqu + dqz + eov - eqs - fhi - fno - for - fou - fqt \\
  & dhop + dhrs + dhsu + dhtv - dinp - dios - dipu - diqv - eho^2 - ehr^2 - ehs^2 - eht^2 + eino + eior + eips + eiqt \\
  &dhoq + dhrt + dhtu + dhtz - dinq - diot - diqu - diqz - elot + elqr - fho^2 - fhr^2- fhs^2 - fht^2  + fino + fior + fips + fiqt + flos - flpr \\
  & dhp^2 + dhs^2 + dhu^2 + dhv^2 - dlnp - dlos - dlpu - dmqu - ehop - ehrs - ehsu- ehtv- eiqv + elno  + elor + elps + elqt + emqs + fiqu - flqs \\
  & dhpq + dhst + dhuv + dhvz - dmnp - dmos - dmpu - dmqv - einq - eiot - eiqu - eiqz- elov + elqs + emno + emor + emou + emqt - fhop - fhrs - fhsu - fhtv + finp + fios + fipu + fiqv \\
  &dhq^2 + dht^2 + dhv^2 + dhz^2 - dmnq - dmot - dmqu - dmqz - emov + emqs - fhoq - fhrt- fhtu - fhtz - flov + flpt + fmno + fmor+ 2fmou - fmps + fmqt \\
  &dil + dop + drs + dsu + dtv - ei^2 - eo^2 - er^2 - es^2 - et^2 \\
  & dim + doq + drt + dtu + dtz + erv - est - fi^2 - fo^2 - fr^2 - fru - ft^2 \\
  &dip^2 + dis^2 + diu^2 + div^2 - dlop - dlrs - dlsu - dmtu - eiop - eirs - eisu - 2eitv + elo^2 + elr^2 + els^2 + elt^2 + emst + fitu - flst \\
  & dipq + dist + diuv + divz - dmop - dmrs - dmsu - dmtv - eioq - eirt - eitu - eitz- elrv + elst + emo^2 + emr^2 + emru + emt^2 \\
  & diq^2 + dit^2 + div^2 + diz^2 - dmoq - dmrt - dmtu - dmtz - emrv + emst - fioq - firt- fitu - fitz - flrv + flst + fmo^2 + fmr^2 + 2fmru - fms^2 + fmt^2 \\
  & dl^2 + dp^2 + ds^2 + du^2 + dv^2 - eil - eop - ers - esu - etv \\
  & dlm + dpq + dst + duv + dvz - fil - fop - frs - fsu - ftv \\
  & dlq - dmp - eiq + emo + fip - flo \\
  & dlt - dms - eit + emr + fis - flr \\
  & dlv - dmu - eiv + ems + fiu - fls \\
  & dlz - dmv - eiz + emt + fiv - flt \\
  & dm^2 + dq^2 + dt^2 + dv^2 + dz^2 - fim - foq - frt - fsv - ftz \\
  & dpt - dqs - eot + eqr + fos - fpr \\
  & dpv - dqu - eov + eqs + fou - fps \\
  & dpz - dqv - eoz + eqt + fov - fpt \\
  & dsv - dtu - erv + est + fru - fs^2 \\
  & dsz - dtv - erz + et^2 + frv - fst \\
  & \boxed{duz - dv^2 - esz + etv + fsv - ftu} \\
  \hline
  G_{15} = G_{e} & e^2 + l^2 + p^2 + s^2 + u^2 + v^2 \\
  & \boxed{ef + lm + pq + st + uv + vz} \\
  & egm + ehq + eit + elv + emz - fgl - fhp - fis - flu - fmv \\
  & ehm + enq + eot + epv + eqz - fhl - fnp - fos - fpu - fqv \\
  & ehoq + ehrt + ehtu + ehtz - einq - eiot - eiqu - eiqz - elpt + elqs - fhop - fhrs - fhsu - fhtv + finp + fios + fipu + fiqv \\
  & ehpq + ehst + ehuv + ehvz - elnq - elot - elpv - elqz - fhp^2 - fhs^2 - fhu^2 - fhv^2 + flnp + flos + flpu + flqv \\
  & ehq^2 + eht^2 + ehv^2 + ehz^2 - emnq - emot - empv - emqz - fhpq - fhst - fhuv - fhvz + fmnp + fmos + fmpu + fmqv \\
  & eim + eoq + ert + esv + etz - fil - fop - frs - fsu - ftv \\
  & eipq + eist + eiuv + eivz - eloq - elrt - elsv - eltz - fip^2 - fis^2 - fiu^2 - fiv^2 + flop + flrs + flsu + fltv \\
  & eiq^2 + eit^2 + eiv^2 + eiz^2 - emoq - emrt - emsv - emtz - fipq - fist - fiuv - fivz + fmop + fmrs + fmsu + fmtv \\
  & elm + epq + est + euv + evz - fl^2 - fp^2 - fs^2 - fu^2 - fv^2 \\
  & elq^2 + elt^2 + elv^2 + elz^2 - empq - emst - emuv - emvz - flpq - flst - fluv - flvz + fmp^2 + fms^2 + fmu^2 + fmv^2 \\
  & em^2 + eq^2 + et^2 + ev^2 + ez^2 - flm - fpq - fst - fuv - fvz \\
  \hline
  G_{16} = G_f & f^2 + m^2 + q^2 + t^2 + v^2 + z^2 \\
  \hline
  G_{22} = G_g & gn + gr + gu + gz - h^2 - i^2 - l^2 - m^2 + nr + nu + nz - o^2 - p^2 - q^2 + ru + rz - s^2 - t^2 + uz - v^2 \\
  & go^2 + gr^2 + gs^2 + gt^2 - 2hio + i^2n - i^2r + i^2u + i^2z - 2ils - 2imt + nr^2 + ns^2 + nt^2 - o^2r + o^2u + o^2z - 2ops - 2oqt + r^2u + r^2z - rs^2 - rt^2 + s^2z - 2stv + t^2u \\
  & gop + grs + gsu + gtv - hip - hlo - i^2s + iln + ilz - imv - l^2s - lmt + nrs + nsu + ntv - o^2s + opz - oqv - p^2s - pqt + rsu + rsz - s^3 - st^2 + suz - sv^2 \\
  & goq + grt + gtu + gtz - hiq - hmo - i^2t + imn - l^2t - m^2t + nrt + ntu + ntz - o^2t - p^2t - q^2t + rtu + rtz - s^2t - t^3 + tuz - tv^2 \\
  & gos - gpr - his + hlr + i^2p - ilo \\
  & got - gqr - hit + hmr + i^2q - imo \\
  & gou - gps - hiu + hls + ilp - l^2o \\
  & gov - gqs - hiv + hms + ilq - lmo \\
  & goz - gqt - hiz + hmt + imq - m^2o \\
  & gp^2 + gs^2 + gu^2 + gv^2 - 2hlp - 2ils + l^2n + l^2r - l^2u + l^2z - 2lmv + ns^2 + nu^2 + nv^2 - 2ops + p^2r - p^2u + p^2z - 2pqv + ru^2 + rv^2 - s^2u + s^2z - 2stv + u^2z - uv^2 \\
  & gpq + gst + guv + gvz - hlq - hmp - ilt - ims - l^2v + lmn + lmr - m^2v + nst + nuv + nvz - opt - oqs - p^2v + pqr - q^2v + ruv + rvz - s^2v - t^2v + uvz - v^3 \\
  & gpt - gqs - hlt + hms + ilq - imp \\
  & gpv - gqu - hlv + hmu + l^2q - lmp \\
  & gpz - gqv - hlz + hmv + lmq - m^2p \\
  & gq^2 + gt^2 + gv^2 + gz^2 - 2hmq - 2imt - 2lmv + m^2n + m^2r + m^2u - m^2z + nt^2 + nv^2 + nz^2 - 2oqt - 2pqv + q^2r + q^2u - q^2z + rv^2 + rz^2 - 2stv + t^2u - t^2z + uz^2 - v^2z \\
  & gru - gs^2 - i^2u + 2ils - l^2r \\
  & grv - gst - i^2v + ilt + ims - lmr \\
  & grz - gt^2 - i^2z + 2imt - m^2r \\
  & gsv - gtu - ilv + imu + l^2t - lms \\
  & gsz - gtv - ilz + imv + lmt - m^2s \\
  & \boxed{guz - gv^2 - l^2z + 2lmv - m^2u} \\
  \hline
  G_{23} = G_{h} &   h^2o^2 + h^2r^2 + h^2s^2 + h^2t^2 - 2hino - 2hior - 2hips - 2hiqt + i^2n^2 + i^2o^2 - i^2s^2 - i^2t^2 - i^2u^2 - 2i^2v^2 - i^2z^2 + 2ilop + 2ilrs + 2ilsu + 2iltv + 2imoq + 2imrt + 2imtu + 2imtz - l^2o^2 - l^2r^2 - l^2s^2 - l^2t^2 + 2lmrv - 2lmst - m^2o^2 - m^2r^2 - 2m^2ru + m^2s^2 - m^2t^2 + n^2r^2 + n^2s^2 + n^2t^2 - 2no^2r - 2nops - 2noqt + o^4 + o^2p^2 + o^2q^2 - o^2u^2 - 2o^2v^2 - o^2z^2 + 2opsu + 2optv + 2oqtu + 2oqtz - p^2s^2 + 2pqrv - 4pqst - 2q^2ru + 2q^2s^2 - q^2t^2 - r^2u^2 - 2r^2v^2 - r^2z^2 + 2rs^2u + 4rstv + 2rt^2z - s^4 - 2s^2t^2 - s^2v^2 -s^2z^2 + 2stuv + 2stvz - t^4 - t^2u^2 - t^2v^2 \\
  & h^2op + h^2rs + h^2su + h^2tv - hinp - hipr - hipu - hiqv - hlno - hlor - hlps - hlqt - i^2ns + i^2op + iln^2 + ilnr + ilp^2 - ilt^2 - ilv^2 - ilz^2 + impq + imst + imuv + imvz + lmoq + lmrt + lmsv + lmtz - m^2op - m^2rs - m^2su - m^2tv + n^2rs + n^2su + n^2tv - no^2s - nopr - nopu - noqv - np^2s - npqt + o^3p + o^2su + o^2tv + op^3 + opq^2 - opru - ops^2 - opt^2 - opv^2 - opz^2 - oqrv + oquv + oqvz + p^2rs + p^2tv + pqrt - pqtu + pqtz - q^2tv - rsv^2 - rsz^2 + rtuv + rtvz + s^2tv - st^2u + st^2z - suz^2 + sv^2z - t^3v + tuvz - tv^3 \\
  & h^2oq + h^2rt + h^2tu + h^2tz - hinq - hiqr - hiqu - hiqz - hmno - hmor - hmps - hmqt - i^2nt + i^2oq + imn^2 + imnr + imp^2 + imq^2 - l^2nt + l^2oq + lmns - lmop + n^2rt + n^2tu + n^2tz - no^2t - noqr - noqu - noqz - np^2t - nq^2t + o^3q + o^2tu + o^2tz + op^2q - 2opst + oq^3 - oqru - oqrz + oqs^2 - oqt^2 + p^2rt + p^2tz - pqsz - pqtv + q^2rt + q^2sv \\
  & h^2p^2 + h^2s^2 + h^2u^2 + h^2v^2 - 2hlnp - 2hlpr - 2hlpu - 2hlqv - 2ilns + 2ilop + l^2n^2 + 2l^2nr - l^2o^2 + l^2p^2 - l^2t^2 - l^2v^2 - l^2z^2 + 2lmpq + 2lmst + 2lmuv + 2lmvz - m^2p^2 - m^2s^2 - m^2u^2 - m^2v^2 + n^2s^2 + n^2u^2 + n^2v^2 - 2nops - 2np^2u - 2npqv + o^2p^2 + o^2u^2 + o^2v^2 - 2opsu - 2oqtu + p^4 + p^2q^2 + p^2s^2 - p^2t^2 - p^2z^2 - 2pqrv + 4pqst + 2pqvz + 2q^2ru - 2q^2s^2 - q^2v^2 - s^2v^2 - s^2z^2 + 2stuv + 2stvz - t^2u^2 - t^2v^2 - u^2z^2 + 2uv^2z - v^4 \\
  & h^2pq + h^2st + h^2uv + h^2vz - hlnq - hlqr - hlqu - hlqz - hmnp - hmpr - hmpu - hmqv- ilnt + iloq - imns + imop - l^2nv + l^2pq + lmn^2 + 2lmnr + lmnu - lmo^2 + lmq^2 + n^2st + n^2uv + n^2vz - nopt - noqs - np^2v - npqu - npqz - nq^2v + o^2pq + o^2uv + o^2vz - 2optu - 2oqtv + p^3q - p^2rv + 2p^2st + p^2vz + pq^3 + pqru - pqrz - pqs^2 + pqt^2 - pquz - pqv^2 + q^2rv + q^2uv \\
  & h^2q^2 + h^2t^2 + h^2v^2 + h^2z^2 - 2hmnq - 2hmqr - 2hmqu - 2hmqz - 2imnt + 2imoq - 2lmnv + 2lmpq + m^2n^2 + 2m^2nr + 2m^2nu - m^2o^2 - m^2p^2 + m^2q^2 + n^2t^2 + n^2v^2 + n^2z^2 - 2noqt - 2npqv - 2nq^2z + o^2q^2 + o^2v^2 + o^2z^2 - 2optv - 2oqtz + p^2q^2 + p^2t^2 + p^2z^2 - 2pqvz + q^4 + q^2t^2 + q^2v^2 \\
  & hiop + hirs + hisu + hitv - hlo^2 - hlr^2 - hls^2 - hlt^2 - i^2np - i^2os - i^2pu - i^2qv + ilno + ilor + ilps + ilqt + nors + nosu + notv - npr^2 - nps^2 - npt^2 - o^3s + o^2pr - o^2pu - o^2qv + op^2s + opqt + orsu + ortv - os^3 - ost^2 - pr^2u + prs^2 + pstv - pt^2u - qr^2v + qrst - qs^2v + qstu \\
  & hioq + hirt + hitu + hitz - hmo^2 - hmr^2 - hms^2 - hmt^2 - i^2nq - i^2ot - i^2qu - i^2qz + imno + imor + imps + imqt - l^2ot + l^2qr + lmos - lmpr + nort + notu + notz - nqr^2 - nqs^2 - nqt^2 - o^3t + o^2qr - o^2qu - o^2qz - op^2t + 2opqs + oq^2t + ortu + ortz - os^2t - ot^3 + pstz - pt^2v - qr^2u - qr^2z + qrs^2 + qrt^2 - qs^2z + qstv \\
  & hip^2 + his^2 + hiu^2 + hiv^2 - hlop - hlrs - hlsu - hltv - ilnp - ilos - ilpu - ilqv + l^2no + l^2or + l^2ps + l^2qt + nos^2 + nou^2 + nov^2 - nprs - npsu - nptv - o^2ps + op^2r - op^2u - opqv + oru^2 + orv^2 - os^2u - ot^2u + p^3s + p^2qt - prsu - prtv + ps^3 + pst^2 + psv^2 - ptuv - qrsv + qrtu - qsuv + qtu^2 \\
  & hipq + hist + hiuv + hivz - hmop - hmrs - hmsu - hmtv - ilnq - ilot - ilqu - ilqz - l^2ov + l^2qs + lmno + lmor + lmou + lmqt + nost + nouv + novz - nqrs - nqsu - nqtv - o^2pt - op^2v + opqr - opqz + oruv + orvz - ostu - ot^2v + p^2qs + pq^2t - prsv + ps^2t + psvz - ptv^2 - qrsz + qst^2 - qsuz + qtuv \\
  & hiq^2 + hit^2 + hiv^2 + hiz^2 - hmoq - hmrt - hmtu - hmtz - imnq - imot - imqu - imqz - 2lmov + lmpt + lmqs + m^2no + m^2or + 2m^2ou - m^2ps + m^2qt + not^2 + nov^2 + noz^2 - nqrt - nqtu - nqtz - o^2qt - 2opqv + oq^2r + oq^2u - oq^2z + orv^2 + orz^2 - ot^2u - ot^2z + p^2qt - 2prtv + 2pst^2 + psz^2 - ptvz + q^3t + qrtu - qrtz - qs^2t - qsvz + qt^3 + qtv^2 \\
  & hloq + hlrt + hltu + hltz - hmop - hmrs - hmsu - hmtv - ilnq - ilot - ilqu - ilqz + imnp + imos + impu + imqv - l^2pt + l^2qs + nprt + nptu + nptz - nqrs - nqsu - nqtv - o^2pt + o^2qs - opqz + oq^2v - p^3t + p^2qs + prtu + prtz - ps^2t - pt^3 + ptuz - ptv^2 - qrsu - qrsz + qs^3 + qst^2 - qsuz + qsv^2 \\
  & hlpq + hlst + hluv + hlvz - hmp^2 - hms^2 - hmu^2 - hmv^2 - l^2nq - l^2ot - l^2pv - l^2qz + lmnp + lmos + lmpu + lmqv + npst + npuv + npvz - nqs^2 - nqu^2 - nqv^2 - op^2t + opqs - p^3v + p^2qu - p^2qz + pq^2v + pruv + prvz - ps^2v - pt^2v + puvz - pv^3 - qru^2 - qrv^2 + qs^2u - qs^2z + 2qstv - qu^2z + quv^2 \\
  & hlq^2 + hlt^2 + hlv^2 + hlz^2 - hmpq - hmst - hmuv - hmvz - lmnq - lmot - lmpv - lmqz + m^2np + m^2os + m^2pu + m^2qv + npt^2 + npv^2 + npz^2 - nqst - nquv - nqvz - opqt + oq^2s - p^2qv + pq^2u - pq^2z + prv^2 + prz^2 - 2pstv + pt^2u - pt^2z + puz^2 - pv^2z + q^3v - qruv - qrvz + qs^2v + qt^2v - quvz + qv^3 \\
  & hos - hpr - ins + iop + lnr - lo^2 \\
  & hot - hqr - int + ioq + mnr - mo^2 \\
  & hou - hps - inu + ip^2 + lns - lop \\
  & hov - hqs - inv + ipq + mns - mop \\
  & hoz - hqt - inz + iq^2 + mnt - moq \\
  & hpt - hqs - lnt + loq + mns - mop \\
  & hpv - hqu - lnv + lpq + mnu - mp^2 \\
  & hpz - hqv - lnz + lq^2 + mnv - mpq \\
  & hru - hs^2 - iou + ips + los - lpr \\
  & hrv - hst - iov + iqs + lot - lqr \\
  & hrz - ht^2 - ioz + iqt + mot - mqr \\
  & hsv - htu - lov + lpt + mou - mps \\
  & hsz - htv - loz + lqt + mov - mqs \\
  & \boxed{huz - hv^2 - lpz + lqv + mpv - mqu} \\
  \hline
  G_i = G_{24} & i^2p^2 + i^2s^2 + i^2u^2 + i^2v^2 - 2ilop - 2ilrs - 2ilsu - 2iltv + l^2o^2 + l^2r^2 + l^2s^2 + l^2t^2 + o^2s^2 + o^2u^2 + o^2v^2 - 2oprs - 2opsu - 2optv + p^2r^2 + p^2s^2 + p^2t^2 + r^2u^2 + r^2v^2 - 2rs^2u - 2rstv + s^4 + s^2t^2 + s^2v^2 - 2stuv + t^2u^2 \\
  & i^2pq + i^2st + i^2uv + i^2vz - iloq - ilrt - iltu - iltz - imop - imrs - imsu - imtv - l^2rv + l^2st + lmo^2 + lmr^2 + lmru + lmt^2 + o^2st + o^2uv + o^2vz - oprt - optu - optz - oqrs - oqsu - oqtv - p^2rv + p^2st + pqr^2 + pqru + pqt^2 + r^2uv + r^2vz - rs^2v - rstu - rstz - rt^2v + s^3t + s^2vz + st^3 - stuz - stv^2 + t^2uv \\
  & i^2q^2 + i^2t^2 + i^2v^2 + i^2z^2 - 2imoq - 2imrt - 2imtu - 2imtz - 2lmrv + 2lmst + m^2o^2 + m^2r^2 + 2m^2ru - m^2s^2 + m^2t^2 + o^2t^2 + o^2v^2 + o^2z^2 - 2oqrt - 2oqtu - 2oqtz - 2pqrv + 2pqst + q^2r^2 + 2q^2ru - q^2s^2 + q^2t^2 + r^2v^2 + r^2z^2 - 2rstv - 2rt^2z + s^2t^2 + s^2z^2 - 2stvz + t^4 + t^2v^2 \\
  & ilpq + ilst + iluv + ilvz - imp^2 - ims^2 - imu^2 - imv^2 - l^2oq - l^2rt - l^2sv - l^2tz + lmop + lmrs + lmsu + lmtv + opst + opuv + opvz - oqs^2 - oqu^2 - oqv^2 - p^2rt - p^2sv - p^2tz + pqrs + pqsu + pqtv + rsuv + rsvz - rtu^2 - rtv^2 - s^3v + s^2tu - s^2tz + st^2v + suvz - sv^3 - tu^2z + tuv^2 \\
  & ilq^2 + ilt^2 + ilv^2 + ilz^2 - impq - imst - imuv - imvz - lmoq - lmrt - lmsv - lmtz + m^2op + m^2rs + m^2su + m^2tv + opt^2 + opv^2 + opz^2 - oqst - oquv - oqvz - pqrt - pqsv - pqtz + q^2rs + q^2su + q^2tv + rsv^2 + rsz^2 - rtuv - rtvz - s^2tv + st^2u - st^2z + suz^2 - sv^2z + t^3v - tuvz + tv^3 \\
  & ipt - iqs - lot + lqr + mos - mpr \\
  & ipv - iqu - lov + lqs + mou - mps \\
  & ipz - iqv - loz + lqt + mov - mpt \\
  & isv - itu - lrv + lst + mru - ms^2 \\
  & isz - itv - lrz + lt^2 + mrv - mst \\
  & \boxed{iuz - iv^2 - lsz + ltv + msv - mtu} \\
  \hline
  G_{25} = G_l & l^2q^2 + l^2t^2 + l^2v^2 + l^2z^2 - 2lmpq - 2lmst - 2lmuv - 2lmvz + m^2p^2 + m^2s^2 + m^2u^2 + m^2v^2 + p^2t^2 + p^2v^2 + p^2z^2 - 2pqst - 2pquv - 2pqvz + q^2s^2 + q^2u^2 + q^2v^2 + s^2v^2 + s^2z^2 - 2stuv - 2stvz + t^2u^2 + t^2v^2 + u^2z^2 - 2uv^2z + v^4 \\
  \hline
  G_{33} = G_n & nru - ns^2 - o^2u + 2ops - p^2r \\
  & nrv - nst - o^2v + opt + oqs - pqr \\
  & nrz - nt^2 - o^2z + 2oqt - q^2r \\
  & nsv - ntu - opv + oqu + p^2t - pqs \\
  & nsz - ntv - opz + oqv + pqt - q^2s \\
  & \boxed{nuz - nv^2 - p^2z + 2pqv - q^2u} \\
  \hline
  G_{34} = G_o & osv - otu - prv + pst + qru - qs^2 \\
  & osz - otv - prz + pt^2 + qrv - qst \\
  & \boxed{ouz - ov^2 - psz + ptv + qsv - qtu} \\
  \hline
  G_{44} = G_r & \boxed{ruz - rv^2 - s^2z + 2stv - t^2u} \\
  \hline 
\end{longtable}

\section{Code for Chapter \ref{sec:moduli}}\label{AppendixB}
The following script can be run in SageMath \cite{sagemath} and produces the Gr\"obner basis above together with the computations needed in the proof of Proposition \ref{prop:radical6}. One can slightly modify the matrix in the beginning to adapt the code to higher dimension $n$. We caution the reader that the function for computing the set of unlucky primes, in the sense of Proposition \ref{prop:win}, is highly inefficient. Especially  the last part of this code requires about one day running time on a laptop.

\begin{verbatim}

  # Define the polynomial ring, fix the lexicographic order and the matrix X

  R.< a,b,c,d,e,f,g,h,i,l,m,n,o,p,q,r,s,t,u,v,z> = 
  PolynomialRing(QQ, 21, order= "lex")
  M = MatrixSpace(R,6,6)
  X = M([a,b,c,d,e,f,b,g,h,i,l,m,c,h,n,o,p,q,d,i,o,r,s,t,e,l,p,s,u,v,
  f,m,q,t,v,z])

  # Define the ideal J

  L = [X.trace()]
  for row in X*X : 
    for entry in row:
        if not entry in L:
            L.append(entry)
  
  for minor in X.minors(3):
    if not minor in L:
        L.append(minor)
        
  J = R.ideal(L)

  # Compute the Gr\"obner basis, the output is listed above 

  grob = J.groebner_basis()

  # The following function takes two sets of polynomials F and G,
  # computes the matrices Z, Y, R such that G = Z.L,   L = Y.G,   R.G = 0,
  # inspects their coefficients and produces the set of coefficients
  # that are not 1 or -1. 

  def unlucky_primes(F, G, ring):
    #F is a set of generators for the ideal and G is the Groebner basis over Q
    #we need to check the entries of three matrices Z,Y,R with F = Z.G; 
    #G = Y.F; R matrix of syzygies.
    generators = list(F)
    Ideal_F = ring.ideal(generators)
    Ideal_G = ring.ideal(list(G))
    unlucky = []
    for poly in G:
        row_Z = poly.lift(Ideal_F) #this produces the entries of the matrix Z 
        #on the line corresponding to poly in G
        for index in range(len(row_Z)):
            entry = row_Z[index]
            for coeff in entry.coefficients():
                if not coeff == 1 and not coeff == -1:
                  if coeff not in unlucky:
                    unlucky.append(coeff)
                        
    for poly in F:
        row_Y = poly.lift(Ideal_G) #this produces the entries of the matrix Y 
        #on the line corresponding to poly in G
        for index in range(len(row_Y)):
            entry = row_Y[index]
            for coeff in entry.coefficients():
              if not coeff == 1 and not coeff == -1:
                if coeff not in unlucky:
                  unlucky.append(coeff)
    
    Sy = Ideal_G.syzygy_module()
    for row in Sy:
        for entry in row:
            for coeff in entry.coefficients():
              if not coeff == 1 and not coeff == -1:
                if coeff not in unlucky:
                  unlucky.append(coeff)
    if not unlucky:
        print("There are no unlucky primes")
    return(unlucky)

  # We apply our function on the polynomials defining J and the Groebner basis 
  # found above

  unlucky_primes(L, grob, R)

  #Output: [2]

  # We have to check now that the leading coefficients of the marked generators 
  # are non zero divisors modulo J. This is done by comparing the division 
  # ideal (J : lc) with J.

  # We start with the the leading coefficient of all but one marked generators 
  # of degree one

  lc_1 = u*z -v^2  
  J == J.quotient(ideal(lc_1))

  # Then the leading coefficient of the remaining marked generator 
  # of degree one

  lc_2 = f 
  J == J.quotient(ideal(lc_2))

  # Last, the leading coefficient of the marked generator of degree two 
  # which is not already monic

  lc_3 = q^2 + t^2 + v^2 + z^2 
  J == J.quotient(ideal(lc_3))

  # Again, computing the division ideal requires computing a Groebner basis 
  # for (xJ, lc(x-1)) hence we have to check again if there are unlucky primes.
  # We define a new polynomial ring obtained by adding the auxiliary variable x

  S.<x, a,b,c,d,e,f,g,h,i,l,m,n,o,p,q,r,s,t,u,v,z> = 
  PolynomialRing(QQ, 22, order= "lex")
  division1 = []
  division2 = []
  division3 = []
  # we construct the ideal xJ (one for each leading coefficient)
  for poly in L:
    division1.append(x*poly)
    division2.append(x*poly)
    division3.append(x*poly)

  # We add the remaining polynomial lc*(x-1), compute a Groebner basis and
  # find the unlucky primes
  # This last part requires about a day running time on a laptop.

  division1.append(lc_1*x -lc_1)
  grob_division1 = S.ideal(division1).groebner_basis()

  print(unlucky_primes(division1, grob_division1, S))

  # Output: [2, 6, 3] 

  division2.append(lc_2*x -lc_2)
  grob_division2 = S.ideal(division2).groebner_basis()

  print(unlucky_primes(grob_division2, division2, S))

  # Output: [2, 4, 3, 6]

  division3.append(lc_3*x -lc_3)
  grob_division3 = S.ideal(division3).groebner_basis()

  print(unlucky_primes(division3, grob_division3, S))
  
  # Output: [2, -2]

  # Last, we need to check that the two discriminants of the polynomials 
  # of degree two are not zero-divisors. 

  # The discriminant of the quadratic polynomial in the variable l is

  delta1 =  (m*p*q + m*s*t + m*u*v + m*v*z)^2 -( m^2*p^2 + m^2*s^2 + 
  m^2*u^2 + m^2*v^2 + p^2*t^2 + p^2*v^2 + p^2*z^2 - 2*p*q*s*t - 2*p*q*u*v
  - 2*p*q*v*z + q^2*s^2 + q^2*u^2 + q^2*v^2 + s^2*v^2 + s^2*z^2 - 
  2*s*t*u*v - 2*s*t*v*z + t^2*u^2 + t^2*v^2 + u^2*z^2 - 2*u*v^2*z + v^4)
  
  # Observe that it is enough to check that it is not a zero divisor modulo
  # J_m. Recall that we need to add one auxiliary varible x to compute
  # the quotient ideal.

  J_m = S.ideal(J).elimination_ideal([m,n,o,p,q,r,s,t,u,v,z])
  J_m == J_m.quotient(S.ideal(delta1))

  # Output: True

  # Similarly, for the discriminant of the quadratic polynomial in f
  
  delta2 =  m^2 + q^2 + t^2 + v^2 + z^2
  J_g = S.ideal(J).elimination_ideal([g,h,i,l,m,n,o,p,q,r,s,t,u,v,z])
  J_g == J_g.quotient(S.ideal(delta2))

  # Output: True

  # Last, since we have computed two new Groebner bases, we have to check
  # for unlucky primes

  division1 = []

  for poly in J_m.gens():
    division1.append(x*poly)
  
  division1.append(delta1*x -delta1)
  grob_division1 = S.ideal(division1).groebner_basis()

  unlucky_primes(division1, grob_division1, S)

  # Outuput: [2,-2, 3, -3, 4]

  division2 = []
  
  for poly in J_g.gens():
    division2.append(x*poly)
    
  division2.append(delta2*x -delta2)
  grob_division2 = S.ideal(division2).groebner_basis()

  unlucky_primes(grob_division2, division2, S)

  # Interestingly for this last computation the number of unlucky primes
  # explodes:
  # [2, 3, 5, 7, 11, 13, 17, 19, 23, 29, 31, 37, 41, 43, 47, 53,
  # 59, 61, 67, 71, 73, 79, 101, 103, 107, 109, 113, 127, 131, 137, 167,
  # 173, 179, 193, 211, 223, 263, 283, 313, 359, 461, 809]
\end{verbatim}

\section{Code for Chapter \ref{sec:adlvs}}\label{AppendixC}
The following script can be run in SageMath \cite{sagemath} and produces the list of admissible elements for the group theoretical datum $(\widetilde{B}_3, J = \{0, 1, 2\}, \sigma, \omega_2^\vee)$ studied in Section \ref{sec:adlvs3}. The function \verb|newtonPoint| also computes the Newton point of a given element in the extended affine Weyl group.

\begin{verbatim}
  # We define the extended affine Weyl group we will be working with and
  # fix the cocharacter omega_2 = (1,1,0) and the non-trivial length-zero 
  # element tau, which gives the action of the Frobenius in the non-split case

  E = ExtendedAffineWeylGroup(["B", 3, 1])
  WF = E.WF()
  F = E.fundamental_group()
  b = E.lattice_basis()
  Wa = E.affine_weyl()

  omega_2 = PW0(b[2])

  tau = F[1]

  # Here we compute the set Adm(omega_2)^J: first we find all elements
  # smaller in the Bruhat order than omega_2 or any conjugate via the
  # finite Weyl group, then take minimal length representatives
  # in the left coset W_J\W

  compare = [] 
  #contains all the t^{x(omega_2)}
  for x in W0:
    e = WF(x*omega_2*x^-1)
    if e in compare:
        continue
    compare.append(e)

  adm = [] 
  #will contain all the elements smaller than t^{x(omega_2)} modulo W_J
  for w in compare:
    for i in range(w.length() + 1):
        for u in Wa.elements_of_length(i):
            if (WF(u)).bruhat_le(w):
                x = u.coset_representative([0,1,2], side = "left").
                reduced_word() 
                if WF.from_reduced_word(x) in adm:                                
                    continue
                adm.append(WF.from_reduced_word(x))

  print(adm) # The output is listed in Section 7.2

  # The following function computes the Newton point of a given element in
  # the extended affine Weyl group. Observe that changing the parameter
  # coweight_space one can use it for other groups. The parameter tau
  # is a length-zero element whose adjoint action is the Frobenius

  R = RootSystem(['B',3])
  coweight_space = R.coweight_space()

  def newtonPoint(w, tau, coweight_space):
    powers = w
    sigma = tau
    order = 1
    while not powers.to_classical_weyl().is_one():
        powers = powers*sigma*w*sigma^-1
        sigma = sigma*tau
        order = order +1
    
    newton = coweight_space(powers.to_translation_right()).
    to_dominant_chamber()/order
    return(newton)

    # We compute the Newton points of the admissible element in the split case
    for w in adm:
        print(w)
        print(newtonPoint(w, WF.from_reduced_word([ ]), coweight_space))
        print( " ")
    
    # Newton points of the admissible elements in the non-split case
    for w in adm:
        print(w)
        print(newtonPoint(w, tau, coweight_space))
        print(" ")
    
    # The output is presented in Section 7.2
\end{verbatim}

\bibliographystyle{alpha} 
\bibliography{GU.bbl} 

\end{document}